\documentclass{amsart}

\usepackage{amsmath, amssymb, amsthm, amsfonts, marginnote, stmaryrd, comment,tikz,tikz-cd}

\usepackage{xcolor}
\usepackage{mathrsfs}
\usepackage{calligra}
\usepackage{hyperref}
\usepackage{mathtools}
\usepackage{pifont}
\usepackage{enumitem}
\usepackage{quiver}
\usepackage{array}
\usepackage{multirow}

\numberwithin{equation}{subsection}

\theoremstyle{plain}
\newtheorem{theorem}{Theorem}[subsection]
\newtheorem{lemma}{Lemma}[subsection]

\newtheorem{proposition}{Proposition}[subsection]
\newtheorem{corollary}{Corollary}[subsection]
\newtheorem{conjecture}{Conjecture}[subsection]

\theoremstyle{definition}
\newtheorem{definition}{Definition}[subsection]

\newtheorem{remark}{Remark}[subsection]
\newtheorem{example}{Example}[subsection]

\newcommand{\C}{\mathbb{C}}
\newcommand{\Q}{\mathbb{Q}}
\newcommand{\Z}{\mathbb{Z}}
\newcommand{\A}{\mathbb{A}}
\newcommand{\wt}[1]{\widetilde{#1}}
\newcommand{\wh}[1]{\widehat{#1}}
\newcommand{\cA}{\mathcal{A}}

\newcommand{\g}{\mathfrak{g}}

\let\O\relax
\newcommand{\O}{\mathcal{O}}
\let\K\relax
\newcommand{\K}{\mathcal{K}}
\let\mod\relax
\newcommand{\mod}{\mathrm{mod}}

\newcommand{\cL}{\mathcal{L}}
\newcommand{\cD}{\mathcal{D}}
\newcommand{\rD}{\mathrm{D}}
\newcommand{\bD}{\mathbf{D}}

\newcommand{\cM}{\mathcal{M}}
\newcommand{\cR}{\mathcal{R}}
\newcommand{\cT}{\mathcal{T}}
\newcommand{\cS}{\mathcal{S}}
\newcommand{\cE}{\mathcal{E}}
\newcommand{\bH}{\mathbb{H}}

\newcommand{\bW}{\overline{\mathcal{W}}}
\newcommand{\cH}{\mathcal{H}}

\newcommand{\Gr}{\mathrm{Gr}}
\newcommand{\ttimes}{\mathbin{\widetilde{\times}}}
\newcommand{\hGO}{{\widehat{G}_{\O}}}
\newcommand{\hGOX}{{\widehat{G}_{\O,{{X^I}}}}}
\newcommand{\hGOfac}[1]{{\widehat{G}_{\O,#1}}}
\newcommand{\hGK}{\widehat{G}_{\K}}

\newcommand{\rot}{\mathrm{rot}}
\newcommand{\dil}{\mathrm{dil}}
\newcommand{\Gm}{\mathbb{G}_{\mathrm{m}}}
\newcommand{\Gmrot}{\Gm^{\rot}}
\newcommand{\Gmdil}{\Gm^{\dil}}
\newcommand{\cGmrot}[1]{{\cG_{\mathrm{m},#1}^{\rot}}}
\newcommand{\QCoh}{\mathsf{QCoh}}
\newcommand{\Coh}{\mathsf{Coh}}
\newcommand{\IndCoh}{\mathsf{IndCoh}}

\newcommand{\cl}{\mathrm{cl}}
\newcommand{\cF}{\mathcal{F}}
\newcommand{\cG}{\mathcal{G}}
\newcommand{\cCAlg}{\mathrm{CAlg}^{\cl}}
\newcommand{\CAlg}{\mathsf{CAlg}}

\newcommand{\Mod}{\mathsf{Mod}}

\newcommand{\VB}{\mathrm{VB}}
\newcommand{\rmat}[1]{\mathbf{r}_{#1}}

\newcommand{\D}{\mathbb{D}}
\newcommand{\id}{\mathrm{id}}
\newcommand{\Set}{\mathrm{Set}}
\newcommand{\hPla}{\widehat{P}_{\lav}}
\newcommand{\Pla}{{P}_{\lav}}

\newcommand{\pla}{{\mathfrak{p}}_{\lav}}

\newcommand{\Pcoh}{\mathrm{P}_{\mathrm{coh}}}
\newcommand{\KPcoh}{\mathrm{KP}_{\mathrm{coh}}}
\newcommand{\KP}{\mathrm{KP}}
\newcommand{\KPqcoh}{\mathrm{KP}_{\mathrm{qcoh}}}
\newcommand{\KPindcoh}{\mathrm{KP}_{\mathrm{indcoh}}}

\newcommand{\Olamu}{\mathcal{O}_{{\lav},\mu}}
\newcommand{\clamunu}{c_{\lav,\muv}^{\nuv}}
\newcommand{\td}{\widetilde{d}}
\newcommand{\tm}{\widetilde{m}}

\newcommand{\Rat}{\mathrm{Rat}}

\newcommand{\rmIC}{\mathrm{IC}}
\newcommand{\ICHg}{\mathbf{IC}^{\mathrm{Hg}}}

\newcommand{\cIC}{\mathcal{IC}}
\newcommand{\Vect}{\mathrm{Vect}}

\newcommand{\HK}{\mathcal{H}_{{K}}}
\newcommand{\lav}{{{\check{\lambda}}}}
\newcommand{\muv}{{\check{\mu}}}
\newcommand{\nuv}{{\check{\nu}}}

\newcommand{\pt}{\mathrm{pt}}

\newcommand{\Rep}[1]{\mathrm{Rep}(#1)}
\newcommand{\Kp}{\mathrm{K}p}
\newcommand{\tbox}{\mathbin{\widetilde{\boxtimes}}}
\newcommand{\rred}{\mathrm{red}}

\newcommand{\conv}{{\mathbin{\scalebox{1.1}{$\mspace{1.5mu}*\mspace{1.5mu}$}}}}

\newcommand{\Tate}{\mathrm{Tate}}

\newcommand{\cC}{\mathcal{C}}

\newcommand{\lr}[1]{\langle#1\rangle}
\newcommand{\op}{\mathrm{op}}

\newcommand{\sigi}{\sigma_1^*{i_1}_*}
\newcommand{\sjgj}{\sigma_2^*{i_2}_*}
\newcommand{\HM}{\mathrm{HM}}

\newcommand{\Stk}{\mathsf{Stk}}
\newcommand{\GStk}{\mathsf{GStk}}
\newcommand{\indGStk}{\mathsf{indGStk}}
\newcommand{\Hg}{\mathrm{Hg}}

\newcommand{\Pv}{\check{P}}

\newcommand{\Grpd}{\mathsf{Grpd}}
\newcommand{\Sch}{\mathsf{Sch}}
\newcommand{\cSch}{\mathrm{Sch}^{\cl}}
\newcommand{\Cat}{\mathsf{Cat}}
\newcommand{\dom}{\mathrm{dom}}
\newcommand{\swap}{\mathrm{swap}}
\newcommand{\sw}{\mathrm{sw}}

\newcommand{\rp}{\mathrm{p}}
\renewcommand{\rq}{\mathrm{q}}
\newcommand{\rpc}{{\mathrm{p}}_{\circ}}
\newcommand{\rqc}{{\mathrm{q}}_{\circ}}

\newcommand{\rank}{\operatorname{rank}}

\newcommand{\Xo}{{X^{\circ}}}

\newcommand{\Yo}{{Y^{\circ}}}

\newcommand{\onetwo}{1\! / 2}

\newcommand{\mbf}[1]{\mathbf{#1}}
\newcommand{\FM}{\mathrm{FM}}
\newcommand{\rT}{\mathrm{T}}
\newcommand{\MHM}{\mathrm{MHM}}
\newcommand{\strong}{\text{-strong}}
\newcommand{\weak}{\text{-weak}}
\newcommand{\pol}{\mathrm{pol}}

\DeclareMathOperator*{\colim}{colim}
\DeclareMathOperator{\Hom}{Hom}

\DeclareMathOperator{\cHom}{\mathscr{H}\text{\kern -3pt {\calligra\large om}}\,}

\DeclareMathOperator{\DR}{DR}
\DeclareMathOperator{\gr}{gr}
\DeclareMathOperator{\Spec}{Spec}
\DeclareMathOperator{\Irr}{Irr}
\DeclareMathOperator{\Perv}{Perv}
\DeclareMathOperator{\codim}{codim}
\DeclareMathOperator{\Supp}{Supp}
\newcommand{\For}{\mathrm{For}}

\usepackage[mathscr]{euscript}

\title{A Coherent Version of Geometric Satake Equivalence for type A}

\author{Shiyixin Liang}
\address{Qiuzhen College\\Tsinghua University\\Beijing, China}
\email{liangsyx21@mails.tsinghua.edu.cn}

\subjclass[]{}

\dedicatory{}	
\date{\today}
\keywords{}

\begin{document}

\begin{abstract}
	In this paper we prove a coherent version of geometric Satake equivalence proposed in \cite[Conj. 1.7]{CW23} for type A. In Cautis-Williams' work \cite{CW23}, they studied an abelian version of the classical limit Satake category, namely, the Koszul perverse heart of the categorified Coulomb branch for adjoint representations. In this paper we study a subcategory generated by a collection of simple objects. We endow this subcategory with a neutral Tannakian structure and identify it with the finite dimensional representation category $\Rep{{\check{G}}}$ for the Langlands dual group ${\check{G}}$. Our method uses tools in Cautis-Williams theory \cite{CW19,CW23,CW24stk,CW24mor} and a Hodge module description of the coherent IC extensions of differential sheaves in \cite{Xin25}.
\end{abstract}

\maketitle
\tableofcontents

\section{Introduction}

\subsection{Overview of the classical limit problems}

Let $G$ be a connected reductive algebraic group over  $\C$ and ${\check{G}}$ be its Langlands dual group over $\C$; let $\g$ and $\check{\g}$ be their Lie algebras respectively. Let $\O:=\C[[t]]$, $\K=\C((t))$. Write $G_{\K}$ and $G_{\O}$ for the loop group and positive loop group respectively. The celebrated geometric Satake equivalence states that there is an equivalence of monoidal abelian categories:
\[\Perv(G_{\O}\backslash G_{\K} /G_{\O},\C)\cong \Rep{{\check{G}}}.\]
Here $\Perv(G_{\O}\backslash G_{\K} /G_{\O},\C)$ is called the Satake category, which is the abelian category of ($\C$-coefficient) perverse (constructible) sheaves on the quotient stack $[G_{\O}\backslash G_{\K} /G_{\O}]$, equipped with a monoidal structure given by the convolution product; $\Rep{{\check{G}}}$ denotes the abelian category of finite dimensional representations of ${\check{G}}$, equipped with a monoidal structure given by the tensor product. 

Using the Riemann-Hilbert correspondence, one can pass from perverse sheaves to $\cD$-modules, and view the Satake category as the abelian category of (regular) holonomic $\cD$-modules on $[G_{\O}\backslash G_{\K} /G_{\O}]$, denoted by $\cD_{\mathrm{hol}}(G_{\O}\backslash G_{\K} /G_{\O})$. The quantization principle suggests that it's interesting to study the category of (quasi-/ind-) coherent sheaves on the cotangent stack
\[\mathrm{T}^*[G_{\O}\backslash G_{\K} /G_{\O}],\]
and compare it with the Satake category. 

Here, the definition of $\mathrm{T}^*[G_{\O}\backslash G_{\K} /G_{\O}]$ is subtle because of the infinite nature of the loop group: $G_{\O}$ is a scheme of infinite type over $\C$, and $G_{\K}$ is an ind-scheme of ind-infinite type. So we adopt procedure below as the definition of this cotangent stack, following \cite[Sec. 7]{BFM05}; this construction is an affine analogue of the finite Steinberg stack (see e.g. \cite[Ch. 3]{CG}). Write $\cT_{G,\g}:=G_{\K}\times^{G_{\O}}\g_{\O}$, and view it as the definition of the cotangent bundle of the affine grassmannian $\Gr_G:=G_{\K}/G_{\O}$. Now we perform the Hamiltonian reduction with respect to the left $G_{\O}$-action on $\Gr_G$. Define $\cR_{G,\g}$ as the \textbf{derived} fiber product
% https://q.uiver.app/#q=WzAsNCxbMCwwLCJcXGNSX3tHLFxcZ30iXSxbMCwxLCJHX3tcXEt9XFx0aW1lc157R197XFxPfX1cXGdfe1xcT30iXSxbMiwwLCIwIl0sWzIsMSwiXFxnX3tcXEt9L1xcZ197XFxPfSJdLFswLDFdLFsxLDMsIltnLHhdXFxtYXBzdG8gW1xcbWF0aHJte0FkfV9nKHgpXSJdLFswLDJdLFsyLDNdXQ==
\begin{equation}\label{eq:Hamiltonian reduction}
	\begin{tikzcd}
		{\cR_{G,\g}} && 0 \\
		{\cT_{G,\g}} && {\g_{\K}/\g_{\O};}
		\arrow[from=1-1, to=1-3]
		\arrow[from=1-1, to=2-1]
		\arrow[from=1-3, to=2-3]
		\arrow["{[g,x]\mapsto [\mathrm{Ad}_g(x)]}", from=2-1, to=2-3]
	\end{tikzcd}
\end{equation}
see section \ref{subsec:R} for more details. Then we refer $[G_{\O}\backslash \cR_{G,\g}]$ as our definition for $\mathrm{T}^*[G_{\O}\backslash G_{\K} /G_{\O}]$.

In a series of work \cite{CW23,CW24stk,CW24mor} by Cautis and Williams, they systematically studied the geometry of $[G_{\O}\backslash \cR_{G,\g}]$, and developped a coherent sheaf theory on it; namely, they defined the stable $\infty$-category $\Coh([G_{\O}\backslash \cR_{G,\g}])$, which serves as a categorified Coulomb branch (in the sense of \cite{BFNII}, for the specific case $(G,N)=(G,\g)$). We write it as $\Coh^{G_{\O}}(\cR_{G,\g})$. It admits a monoidal structure given by convolution product. Its $K$-group $K^{G_{\O}}(\cR_{G,\g})$ is computed in \cite[Thm. 7.3]{BFM05}: we fix a maximal torus $T\subset G$, and let ${\check{T}}\subset {\check{G}}$ be its dual torus, then $K^{G_{\O}}(\cR_{G,\g})$ can be identified with $\C[T\times {\check{T}}]^W$. Thus the evident symmetry $\C[T\times {\check{T}}]^W\cong \C[{\check{T}}\times T]^W$ gives isomorphism
\begin{equation}\label{eq:BFM}
	K^{G_{\O}}(\cR_{G,\g})\cong K^{{\check{G}}_{\O}}(\cR_{{\check{G}},\check{\g}}),
\end{equation}
which leads to the following conjecture in \cite{BFM05}:
\begin{conjecture}[{\cite[Conj. 7.9]{BFM05}}]\label{conj:BFMconj}
	The isomorphism \eqref{eq:BFM} can be lifted to an equivalence of monoidal (stable $\infty$-) categories \footnote{%In \cite{BFM05}, they considered another similar (triangulated) categories with the same $K$-group respectively. 
	The original conjecture in \cite{BFM05} considered a slightly different category with the same $K$-group. Namely, they considered the full (triangulated) subcategory of the derived category of $G_{\O}$-equivariant coherent sheaves on $\cT_{G,\g}$, consisting of objects whose supports are contained in the classical locus of $\cR_{G,\g}$; the ${\check{G}}$ side uses the similar category.}
	\[\Coh^{G_{\O}}(\cR_{G,\g})\cong \Coh^{{\check{G}}_{\O}}(\cR_{{\check{G}},\check{\g}}).\]

\end{conjecture}
As explained in \textit{ibid.}, this conjecture should be viewed as a classical, local manifestation of the geometric Langlands correspondence.

\subsection{An abelian version by Cautis-Williams}
Next, we explain more on the Cautis-Williams theory. Define
\[\hGO:=(G_{\O}\rtimes \Gmrot)\times \Gmdil,\]
which acts on the spaces involved in the diagram \eqref{eq:Hamiltonian reduction}.
Here the $\Gmrot$ is the $\Gm$ acting by rotating the loop variable $t\in \C((t))$, and $\Gmdil$ is the $\Gm$ acting by scaling on $\g$. They play important roles in our paper, see Remark \ref{rmk:Gm_importance}. \cite{CW23} considered $\Coh^{\hGO}(\cR_{G,\g}):=\Coh([\hGO\backslash \cR_{G,\g}])$, and defined a finite length t-structure on $\Coh^{\hGO}(\cR_{G,\g})$, called the Koszul perverse t-structure. This t-structure is a combination of the (middle) perverse coherent t-structure (\cite{AB10}) on $\Coh^{G_{\O}}(\Gr_G)$ and the Koszul t-structure on each fiber of $\cR_{G,\g}\to \Gr_G$. We denote the heart of this t-structure by $\KPcoh^{\hGO}(\cR_{G,\g})$\footnote{This category is denoted by $\mathcal{KP}_{G,\g}$ in \cite{CW23}}. This is a finite length rigid monoidal abelian category. Let $W$ be the Weyl group, and $P$ (resp. ${\Pv}$) be the character (resp. cocharacter) lattice of $T$; we fix a Borel subgroup $B\subset G$, and write $P_+$ (resp. $\Pv_+$) for the dominant part of $P$ (resp. $\Pv$). Then the simple objects of $\KPcoh^{\hGO}(\cR_{G,\g})$ (up to grading shifts) are in bijection with $({\Pv}\times P)/W$. We will fix an appropriate grading shift normalization (with details in Section \ref{subsubsec:Koszul_perverse}), and write the simple object labeled by $(\lav,\mu)$ as $\cL_{\lav,\mu}$.

\cite{CW23} proposed an abelian version of the Conjecture \ref{conj:BFMconj}.
\begin{conjecture}[{\cite[Conj. 1.6]{CW23}}]\label{conj:abelian_version}
	There is a monoidal equivalence $\KPcoh^{\hGO}(\cR_{G,\g})\xrightarrow{\sim}\KPcoh^{\wh{{\check{G}}}_{\O}}(\cR_{{\check{G}},\check{\g}})$. The induced bijection between the simple objects on each side is compatible with their labeling by $({\Pv}\times P)/W$.
\end{conjecture}
This conjecture implies in particular the following coherent version of the geometric Satake equivalence. 
\begin{conjecture}[{\cite[Conj. 1.7]{CW23}}]\label{conj:Rep_embed}
	There is a monoidal functor $\Rep{{\check{G}}}\to \KPcoh
	^{\hGO}{(\cR_{G,\g})}$ which takes the irreducible ${\check{G}}$-representation with highest weight $\lav$ to the simple object labeled by $(\lav,0)\in {\Pv}\times P$.
\end{conjecture}

\subsection{Main theorem}
The main theorem of this paper proves the Conjecture \ref{conj:Rep_embed} for $G$ of type A. Let $\KP_0$ be the full subcategory of $\KPcoh^{\hGO}(\cR_{G,\g})$ generated by simple objects $\cL_{\lav,0}$ for all $\lav\in \Pv_+$ under direct sums, convolutions, subquotients and (right and left) duals (Definition \ref{def: def of KP_0}).
\begin{theorem}[Theorem \ref{main thm}]\label{main thm in intro}
	Assume $G$ is of type A. Then the category $\KP_0$ admits a neutral Tannakian structure (in the sense of \cite{DM82}), and is equivalent to $\Rep{{\check{G}}}$ as neutral Tannakian categories over $\C$.
\end{theorem}

The reason for assuming $G$ is of type A is discussed in Section \ref{subsec_intro:Hodge}.

\subsection{The strategy of proof}
\subsubsection{}
The proof of the main theorem \ref{main thm in intro} is roughly divided into two steps:
\begin{enumerate}
	\item We prove the desired result on the $K$-group level. 
	
	\item We endow $\KP_0$ with a neutral Tannakian structure. Then we identify the Tannakian group as ${\check{G}}$.
\end{enumerate}

The first step is summarized as the following theorem (which is the same as Proposition \ref{Prop for calculation of dual}+ Theorem \ref{thm: determine KP_0 on K-group}).
\begin{theorem}
	Assume $G$ is of type A.
	\begin{enumerate}
		\item The left dual and right dual of $\cL_{\lav,0}$ are both isomorphic to $\cL_{-w_0(\lav),0}$.
		\item In the $K$-group $K(\KP_0)\subset K^{\hGO}(\cR_{G,\g})$, we have the following identity:
		\begin{equation*}
			[\cL_{\lav,0}]\conv [\cL_{\muv,0}]=\sum_{\nuv\in \Pv_+}\clamunu [\cL_{\nuv,0}]. 
		\end{equation*}
		Here, $\clamunu$ is the structure coefficient of tensor products in $\Rep{{\check{G}}}$, see \ref{def of clamunu}.
	\end{enumerate}
	
	By definition of $\KP_0$, these two results imply that simple objects in $\KP_0$ are precisely those $\cL_{\lav,0}$ for $\lav\in \Pv_+$.
\end{theorem}

For the second step, we first use the renormalized $r$-matrices constructed in \cite[Sec. 7]{CW23} to give a system of commutativity constraints in $\KP_0$. Then we construct the fiber functor as follows. Consider the morphisms $i_1:\cR_{G,\g}\to \Gr_G\times \g_{\O},[g,x]\mapsto ([g],\mathrm{Ad}_g(x))$, and $\sigma_1:\Gr_G\to \Gr_G\times \g_{\O},[g]\mapsto ([g],0)$. Define
\[\Xi:=\sigma_1^*{i_1}_*:\Coh^{\hGO}(\cR_{G,\g})\to \Coh^{\hGO}(\Gr_G).\]
By the construction of Koszul perverse t-structure in \cite[Thm. 4.1]{CW23}, this functor is Koszul perverse t-exact. Thus it induces a functor
\[\Xi:\KPcoh^{\hGO}(\cR_{G,\g})\to \KPcoh^{\hGO}(\Gr_G).\]
\begin{theorem}[Corollary \ref{cor:fiber functor}, Theorem \ref{thm: fibre functor}]\label{thm:Tannakian}
	Assume $G$ is of type A. Consider the composition of functors
	\[\KP_0\subset \KPcoh^{\hGO}(\cR_{G,\g})\xrightarrow{\Xi} \KPcoh^{\hGO}(\Gr_G)\xrightarrow{\pi_*} \Coh^{\hGO}(\pt) \xrightarrow{\cH^i}\Rep{\hGO}\xrightarrow{\mathrm{For}}\Vect.\]
	Denote it by $\mathbf{F}^i$.
	\begin{enumerate}
		\item When $i\neq 0$, the functor $\mathbf{F}^i$ is $0$;
		\item When $i=0$, $\mathbf{F}^i$ can be equipped with a fiber functor structure, which gives a neutral Tannakian structure on $\KP_0$.
	\end{enumerate}
\end{theorem}

Finally, we identify the Tannakian group with ${\check{G}}$ using the $K$-group level result and the trick from \cite[Thm. 1.2]{KLV}.

\subsubsection{Hodge module theory}\label{subsec_intro:Hodge}
The Hodge module theory plays an important role in the proof. The reason is as follows. Using the charaterizations of $\cL_{\lav,0}$ ($\lav\in\Pv_+$), we can deduce that (see Corollary \ref{sigi of KPla}):
\[\Xi(\cL_{\lav,0})\cong \bigoplus_{0\le k\le d_{\lav}}\cIC(\Omega_{\Gr_{\lav}}^k)\lr{k-\tfrac12 d_{\lav}}[k].\]
Here $\cIC$ is the coherent IC-extension in Definition \ref{def:cIC}; $\Gr_{\lav}\subset \Gr_G$ is the spherical Schubert cell correspond to $\lav$, with dimension $d_{\lav}$, and $\Omega_{\Gr_{\lav}}^k$ is the $k$-th differential sheaf relative to $\Spec \C$; $\lr{-}$ is the grading shift due to $\Gmdil$-equivariance (Definition \ref{shift functors}), and $[-]$ is the cohomology grading shift.

The key point is that, we can use the Hodge module theory to give a description of $\cIC(\Omega_{\Gr_{\lav}}^k)$. In Saito's Hodge module theory, there is a (polarizable) pure Hodge module of weight $d_{\lav}$ on $\overline{\Gr}_{\lav}$, called the IC Hodge module $\ICHg_{\lav}$. Consider the graded de Rham complexes $\gr_p\DR(\ICHg_{\lav}),p\in \Z$, which is a complex of coherent sheaves; see Appendix \ref{subsec:grDR} for more details. In fact, we will need a (weak) equivariant version of graded de Rham functor constructed in Appendix \ref{subsec:equivariant}, and obtain an object $\gr_p\DR^{\hGO}(\ICHg_{\lav})\in \Coh^{\hGO}(\overline{\Gr}_{\lav})$.
\begin{proposition}[Corollary \ref{cor: Hodge-de Rham desciption for IC lav}]
	Asssume $G$ is of type A. There is an isomorphism in $\Coh^{\hGO}(\overline{\Gr}_{\lav})$:
	\begin{equation}\label{eq:intro_Hg}
		\cIC(\Omega_{\Gr_{\lav}}^k)\cong \gr_{-k}\DR^{\hGO}(\ICHg_{\lav})[k-d_{\lav}].
	\end{equation}
\end{proposition}

This description uses the following result in \cite{Xin25}.
\begin{theorem}[{\cite[Thm. 5.4]{Xin25}}]\label{thm:Xin_intro}
	Let $X$ be a symplectic singularity of dimension $2n$, which admits a \textbf{symplectic resolution}. Then the following inequality holds
	\[\codim \Supp \cH^{i}\left(\gr_{-k}\DR(\ICHg_{X})[k- 2n]\right)\ge 2i+2 \text{ \ for any $i\ge 1$}.\]
	for all $i\ge 1$.
\end{theorem}
This inequality (together with the inequality for its Grothendieck-Serre dual) gives the characterization of coherent IC-extension functor by \cite[Prop. 4.7]{Xin25}, see Proposition \ref{prop:cIC_and_support}. 

The singularity of (spherical) affine Schubert variety $\overline{\Gr}_{\lav}$ at $t^{\muv}\in \overline{\Gr}_{\lav}$ is captured by the affine grassmannian slice $\overline{\mathcal{W}}_{\lav}^{\muv}$ defined as \eqref{eq:slice}. It's known in \cite[Thm. 2.9]{KWWY} that this slice has symplectic singularity in the sense of \cite{Bea00}, and it admits a symplectic resolution \textbf{if $G$ is of type A}. Then we can use Xin's result (Theorem \ref{thm:Xin_intro}) to deduce the Hodge module theoretic description \eqref{eq:intro_Hg} of coherent IC-sheaves. 

For $G$ of other types, there are slices without symplectic resolutions using \cite[Thm. 2.9]{KWWY}. It will be interesting to explore whether the coherent $\cIC$-extensions could be characterized by Hodge modules, or we should modify the perverse coherent t-structure using other perversity functions. Once this issue is resolved, the method presented in this paper may be generalized​ to other types of​ reductive group $G$.

We finally mention another possible approach towards the classical limit of Satake equivalence. In Fedorov's work \cite{Fe21} on Satake equivalence for Hodge modules, they considered the $G_{\O}$-equivariant abelian category of (polarizable) pure Hodge modules on $\Gr_G$:
\[\HM^{G_{\O}}(\Gr_G).\]
It contains a monoidal abelian full subcategory generated by Tate twists of $\ICHg_{\lav}$:
\[\Tate^{G_{\O}}(\Gr_G).\]
They endowed this category with a Tannakian structure, and identified this category with the representation category of Deligne’s modified Langlands dual group ${\check{G}}_{\Q}\times^{\mu_2}\mathbb{G}_{\mathrm{m},\Q}$. We refer to \textit{ibid.} for more details.

In \textit{ibid.} the author says, in a subsequent work, they want to construct an associated graded functor
\[\gr:\HM^{G_{\O}}(\Gr_G)\to  \Coh(\rT^*(G_{\O}\backslash\Gr_G))^{\heartsuit},\]
and then get a functor
\[\mathrm{Rep}_{\Q}(\check{G}_{\Q}\times^{\mu_2}\mathbb{G}_{\mathrm{m,\mathbb{Q}}})\to \Coh(\rT^*(G_{\O}\backslash\Gr_G))^{\heartsuit}.\]
The functor $\gr$ above seems not easy to define because $\Gr_G$ is not a limit of smooth projective schemes. 

It would be interesting to compare this approach with ours.

\subsection{Further direction}
Define $\wt{\KP}_0$ to be the full subcategory of $\KPcoh^{\hGO}(\cR_{G,\g})$ generated by objects in $\KP_0$ under extensions. Assume $G$ is of type A. Applying the same method as in the Theorem \ref{thm:Tannakian}, $\wt{\KP}_0$ can be endowed with a neutral Tannakian structure, thus is equivalent to $\Rep{\wt{G}'}$ for some affine group scheme $\wt{G}'$ over $\C$. We propose the following conjecture, which serves as​ a further step towards Conjecture \ref{conj:abelian_version}.
\begin{conjecture}
	This Tannakian group scheme $\wt{G}'$ is isomorphic to the positive loop group $\check{G}_{\O}$. In other words, $\wt{\KP}_0\cong \Rep{\check{G}_{\O}}$ as neutral Tannakian categories over $\C$. 
\end{conjecture}

\subsection{Structure of the paper}
In section \ref{sec: AG background}, we review some algebro-geometric background in Cautis-Williams theory.

In section \ref{sec:CW_theory}, we review basic constructions on the categorified Coulomb branch, such as the convolution product and its rigidity, Koszul perverse t-structure, renormalized $r$-matrices, etc.

In section \ref{sec:KP_0}, we define the full subcategory $\KP_0\subset \KPcoh^{\hGO}(\cR_{G,\g})$, and compute the duals for objects in this subcategory.

In section \ref{sec:K_group}, we determine $\KP_0$ at the $K$-group level.

In section \ref{sec:Tannakian}, we endow $\KP_0$ with a neutral Tannakian structure, and identify its Tannakian group with ${\check{G}}$.

In appendix \ref{sec:appendix}, we review some basic facts on Hodge module theory and associated graded constructions used in the main context.

\subsection{Notations and conventions}

\begin{enumerate}
    \item Unless specified otherwise, all vector spaces, schemes and stacks will be defined over $\C$; we also write $\pt=\Spec \C$. Unlike the terminology in \cite{CW23,CW24stk,CW24mor}, a category refers to an ordinary category in the traditional way. We will point out $\infty$-category when we use it\footnote{We also deliberately used the $\mathsf{mathsf}$ font for $\infty$-categories and the $\mathrm{mathrm}$ font for ordinary categories.}. Similarly, a scheme means a classical scheme, and we will point out derived scheme when we use it. In section \ref{sec: AG background}, we will briefly review some algebro-geometric background in Cautis-Williams theory. 
    
    In this paper, a variety over $\C$ means an integral quasi-projective scheme over $\C$.

    \item An $\infty$-category $\cC$ is called stable if it has a zero object, every morphism has a fiber and a cofiber, and a triangle in $\cC$ is a fiber sequence if and only if it a cofiber sequence (c.f. \cite[Sec. 1.1.1]{Lur17HA}). The homotopy category $\mathrm{h}(\cC)$ is a triangulated category (c.f. \cite[Sec. 1.1.2]{Lur17HA}). A t-structure on $\cC$ is a t-structure on its homotopy category $\mathrm{h}(\cC)$ (c.f. \cite[Sec. 1.2.1]{Lur17HA}). This data contains information on two subcategories $(\cC^{\le0},\cC^{\ge0}) $, equipped with truncation functors $\tau^{\le n}:\cC \to \cC^{\le n} $, $\tau^{\ge n}:\cC \to \cC^{\ge n} $. The heart of the t-structure $\cC^{\heartsuit}:=\cC^{\le 0}\bigcap\cC^{\ge 0}$ is an abelian category. Write $\cH^n:=\tau^{\ge n}\tau^{\le n}[-n]:\cC\to \cC^{\heartsuit}$ for the cohomology functor of this t-structure. We also write $\cC^+:=\bigcup_{n\in \Z}\cC^{\ge n}$ (resp. $\cC^-:=\bigcup_{n\in \Z}\cC^{\le n}$) for the bounded below part (resp. bounded above part) of $\cC$, and $\cC^b:=\cC^+\bigcap \cC^-$. The t-structure is called bounded if $\cC=\cC^b$.
    
    The stable $\infty$-categories $\QCoh$, $\Coh$, $\IndCoh$ admit standard t-structures (we will introduce these sheaf theories in Section \ref{sec: AG background}), and we will write $(-)^{\heartsuit}$ for the heart of the standard t-structures; we also call the objects in $\QCoh(-)^{\heartsuit}$ (resp. $\Coh(-)^{\heartsuit}$) ordinary quasi-coherent (resp. coherent) sheaves. We will also consider the Koszul perverse t-structures (for specific spaces), and their hearts will be denoted by $\KPcoh$ $\KPqcoh$ $\KPindcoh$ respectively. The cohomology functor for these Koszul perverse t-structures are denoted by $\cH_K^n,n\in \Z$.
    
    \item For a ($\infty$-)category $\cC$, a grading shift functor on $\cC$ means an autoequivalence on $\cC$. When $\cC$ is a stable $\infty$-category, the cohomology grading shift functor on $\cC$ is denoted by $[1]$. There are two other grading shift functors $\{1\}$ and $\lr{1}$ playing important roles in this paper, which are given by some $\Gm$-equivariance. See Definition \ref{shift functors}.

    \item For an abelian category $\cA$, let $K_0(\cA)$ denote its $K_0$-group (which is also called the Grothendieck group); the notation $K_0$ will be abbreviated to $K$ throughout this paper. It is the free abelian group on the isomorphism classes of objects in $\cA$, modulo the relations $[A] - [B]+[C]=0$ associated to exact sequences $0\to A\to B\to C\to 0$.
    
    For a stable $\infty$-category $\cC$, let $K(\cC)$ denote its $K_0$-group. It is the free abelian group on the isomorphism classes of objects in $\cC$, modulo the relations $[A]-[B]+[C]=0$ associated to fiber-cofiber sequences $A\to B\to C$. Let $(\cC^{\le 0},\cC^{\ge 0})$ be a bounded t-structure on $\cC$, and write $\cC^{\heartsuit}$ for the heart of this t-structure. Then the embedding $\cC^{\heartsuit}\to \cC$ induces an isomorphism of $K$-groups, with the inverse given by $[C]\mapsto \sum_{i\in \Z}(-1)^i[\cH^i(C)]$.

    \item Let $X$ be a derived scheme with a classical affine group scheme $H$-action. We write $\QCoh^H(X):=\QCoh([H\backslash X])$ for the quasi-coherent sheaf category on the fpqc quotient $[H\backslash X]$; the quasi-coherent sheaves on $[H\backslash X]$ can also be described as equivariant sheaves, see Example \ref{exam:QCoh(X/G)}. Similar notations are employed for coherent sheaves and ind-coherent sheaves. We also write $K^H(X):=K(\Coh^{H}(X))$. 

    We often use the same name for morphisms before and after quotient. We will also abuse notations for some equivariant sheaves if there's no confusions. For example, we write $\O_{X}$ for the structure sheaf $\O_{[H\backslash X]}\in \QCoh^H(X)$, and write $\omega_X=\omega_{X/\pt}$ for the dualizing complex $\omega_{[H\backslash X]/[H\backslash \pt]}:=a^!\O_{[H\backslash \pt]}$, where $a:[H\backslash X]\to [H\backslash \pt]$ is the natural projection.
    
    Assume $X$ is a (classical) scheme. We write $\Omega_{X}:=\Omega_{X/\pt}$ for the (ordinary) K\"ahler differential sheaf relative to $\pt$. It admits a natural $H$-equivariant structure, thus we view $\Omega_{X}$ as object in $\QCoh^H(X)^{\heartsuit}$. We also write $\Omega_X^k:=\bigwedge^k\Omega_X\in \QCoh^H(X)^{\heartsuit}$ for the (ordinary) $k$-th differential sheaf; here $\bigwedge$ stands for the wedge product of ordinary quasi-coherent sheaves.

    Assume $X$ is a (classical) scheme. For $\cF\in \Coh^H(X)^{\heartsuit}$, we view $\cF$ as an $H$-equivariant ordinary coherent sheaf on $X$. Then we can define the support of $\cF$ as usual, which is denoted by $\Supp \cF$. It is an $H$-invariant closed subscheme of $X$. In general, for $\cF\in\Coh^H(X)$, we write $\Supp \cF$ for $\bigcup_k \Supp\cH^k(\cF)$.

\end{enumerate}

\begin{figure}
	\renewcommand{\arraystretch}{1.5}  
	\centering
\scalebox{0.7}{

\begin{tabular}
	{|
		>{\arraybackslash}m{3cm}
		|>{\arraybackslash}m{15.5cm}|>{\arraybackslash}m{2.7cm}|
	}
	\hline
	Notation & Description & Location defined \\
	\hline
	$Y_{\O},Y_{\K}$ & Positive loop space and loop space & \multirow[c]{4}{*}{Sec.\ref{subsec: def of spaces}} \\
	\cline{1-2} $(G,N)$ &
	{$G$: reductive group; $N$: finite dimensional representation of $G$; e.g. $N=\g$ the adjoint representation} &  \\
	\cline{1-2}
	$\hGO,\hGK$ & (Positive) loop groups extended by extra two $\Gm$, called $\Gmrot$ and $\Gmdil$ &  \\
	\cline{1-2}
	$\Gr_G,\cT_{G,N},\cR_{G,N}$ & 
	{Spaces used in defining the space of triple $\cR_{G,N}$; usually abbreviate as $\Gr,\cT,\cR$} & \\
	\hline 
	\shortstack{$Y_{\O,X^I},Y_{\K,X^I}$\\$\Gr_{X^I},\cT_{X^I},\cR_{X^I}$} & Global versions of these spaces, for a fixed algebraic curve $X=\A^1$ & \multirow{1}{1.5cm}{Sec.\ref{subsubsec:gobal}} \\
	\hline
	\shortstack{$i_1:\cR\hookrightarrow \Gr\times N_{\O}$\\ $i_2:\cR\hookrightarrow \cT$} & Closed embeddings in the definition of $\cR$ & \multirow{4}{1.5cm}{{Sec.\ref{subsec:R}; \\ \ \\ Global version: Sec.\ref{subsubsec:gobal}}} \\
	\shortstack{$\sigma_1:\Gr\hookrightarrow \Gr\times N_{\O}$\\ $\sigma_2:\Gr\hookrightarrow \cT$} & Zero section embeddings & \\
	\shortstack{$p_1:\Gr\times N_{\O}\to \Gr$\\ $p_2:\cT\to \Gr$\\ $p:\cR\to \Gr$} & Natural projections to $\Gr$ & \\
	\hline
	$\Xi:=\sigi$ & Composition functor $ \Coh^{\hGO}(\cR)\xrightarrow{i_{1*}}\Coh^{\hGO}(\Gr\times N_{\O}) \xrightarrow{\sigma_1^*} \Coh^{\hGO}(\Gr)$ & Sec. \ref{sec:K_group} \\
	\hline
	$X\ttimes Y$ & Twisted product for spaces &\multirow{2}{*}{Sec.\ref{subsubsec:twisted_prod}}\\
	$\cF\tbox \cG$ & Twisted (external) product for sheaves & \\
	\hline
	$\Pv$, $P$ & $\Pv=X_*(T)$ is the cocharacter lattice, $P=X^*(T)$ is the character lattice & Sec. \ref{sec:CW_theory} beginning
	\\
	$(\Pv\times P)_{\dom}$ & Index set (up to shift) of simple objects in $\Pcoh^{\hGO}(\Gr)$, $\KPcoh^{\hGO}(\Gr)$, $\KPcoh^{\hGO}(\cR)$ & After Thm. \ref{theorem:simple objects on Gr}
	\\
	\hline
	$\Gr_{\lav},\Gr_{\le\lav}$ & $\Gr_{\lav}\subset \Gr$ is a (spherical) Schubert cell, $\Gr_{\le \lav}\subset \Gr$ is the closure $\overline{\Gr}_{\lav}$ & {Sec. \ref{subsubsec:Gr_G}} \\
	$\Gr_{\lav,\muv}$ & $\Gr_{\lav,\muv}=\Gr_{\lav}\ttimes \Gr_{\muv}$ & {Sec. \ref{subsec: twisted prod is IC}} \\ 
	$\cR_{\lav},\cT_{\lav},\cR_{\le\lav},\cT_{\le\lav}$ & The base change of $\Gr_{\lav}$ and $\Gr_{\le\lav}$ to $\cR$ and $\cT$ respectively & Eq. \eqref{eq:RT} \\
	$i_{\le \lav},j_{\lav}$ & $i_{\le \lav}:\Gr_{\le\lav}\hookrightarrow \Gr$ is the closed embedding, $j_{\lav}:\Gr_{\lav}\hookrightarrow \Gr_{\le \lav}$ is the open embedding & Sec. \ref{subsubsec:Gr_G} \\ 
	\hline
	$j_{\lav,p!*}$ & $\Pcoh^{\hGO}(\Gr_{\lav})\to \Pcoh^{\hGO}(\overline{\Gr}_{\lav})$, intermediate extension for perverse sheaves & Eq. \eqref{intermediate extension for P}
	\\
	$j_{\lav,!*}$ & $\KPcoh^{\hGO}(\Gr_{\lav})\to \KPcoh^{\hGO}(\overline{\Gr}_{\lav})$, intermediate extension for Koszul perverse sheaves & Eq. \eqref{intermediate extension for KP}
	\\
	$\wt{\cIC}=\wt{\cIC}_{X}$ & $\mathrm{Refl}^H(X)\to \Coh^H(X)$,  $\wt{\cIC}$-functors & Eq. \eqref{eq:cIC for Refl}
	\\
	${\cIC}={\cIC}_{U\subset X}$ & $\mathrm{VB}^H(U)\to \Coh^H(X)$,  coherent $\cIC$-extension functors; $\cIC_{U\subset X}=\wt{\cIC}_X\circ \cH^0(j_*(-))$ & Eq. \eqref{eq:cIC for VB}
	\\
	relations between above functors & Let $\cE\in \VB^{\hGO}(\Gr_{\lav})$ with $\Gmdil$-weight $n$, then $j_{\lav,!*}(\cE[n+\tfrac12 d_{\lav}])=j_{\lav,p!*}(\cE[\tfrac12 d_{\lav}])[n]=\cIC(\cE)[n+\tfrac12 d_{\lav}]$ & Eq. \eqref{eq:!*relations}
	\\
	\hline
	$\rmIC^{\Q}$, $\rmIC^{\C}$ & Constructible IC complex with coefficient $\Q$ and $\C$ respectively & \cite{BBD} \\
	$\ICHg$ & IC Hodge module & Sec. \ref{subsec:Hodge} \\
	\hline
	$\overline{\cL}^p_{\lav,\mu}$ & $\overline{\cL}^p_{\lav,\mu}=i_{\le \lav *}j_{\lav,p!*}(\Olamu[\tfrac{1}{2}d_{\lav}])$, simple object in $\Pcoh^{\hGO}(\Gr)$ & Eq. \eqref{eq:cLp}\\
	$\overline{\cL}_{\lav,\mu}$ & $\overline{\cL}_{\lav,\mu}=i_{\le \lav *}j_{\lav,!*}(\Olamu\lr{-\tfrac{1}{2}d_{\lav}})$, simple object in $\KPcoh^{\hGO}(\Gr)$ & Eq. \eqref{eq:cLb} \\
	$\cL_{\lav,\mu}'$; $\cL_{\lav}'$ & simple object in $\KPcoh^{\hGO}(\cR_{\le \lav})$; $\cL_{\lav}':=\cL_{\lav,0}'$ & Eq. \eqref{def of KPlamu'} \\
	$\cL_{\lav,\mu}$; $\cL_{\lav}$ & $\cL_{\lav,\mu}=i_{\le \lav*}(\cL_{\lav,\mu}')$, simple object in $\KPcoh^{\hGO}(\cR)$; $\cL_{\lav}:=\cL_{\lav,0}$ & Eq. \eqref{def of KPlamu}
	\\
	\hline
	$\D_X$ & $\D_X(-)=\cHom(-,\omega_X)$ is the Grothendieck-Serre dual & Eq. \eqref{eq:tly} \\
	$\cD_X$ & $\cD_X(-)=\cHom(-,\omega_X)[-\dim X]$ is the shifted Grothendieck-Serre dual & Eq. \eqref{eq:cD}\\
	$\bD_X$ & Verdier dual functor for Hodge modules & Eq. \eqref{eq:wjy}\\
	\hline
	$[1],\lr{1},\{1\}$ & $[1]$: cohomology grading shift; $\lr{1}$ / $\{1\}$: grading shift w.r.t $\Gmdil$ / $\Gmrot$-equivariance & Def. \ref{shift functors} \\
	$(1)$ & Tate twist for mixed Hodge modules & Eq. \eqref{eq:Tate twist}
	\\
	\hline
\end{tabular}
}
\caption{Notation used in the paper}
\label{notation}
\end{figure}

We abuse certain symbols as follows.
\begin{enumerate}
	\item We view $ \cR,\Gr\times N_{\O},\cT$ as spaces over $\Gr$. The base changes of $i_{\le \lav}$ and $j_{ \lav}$ to $\cR,\Gr\times N_{\O},\cT$ is also denoted by the same symbols.
	\item We denote the base change of $i_1,i_2,\sigma_1,\sigma_2,p_1,p_2,p$ over any locally closed subscheme $S\subset \Gr$ by the same symbols. We also write $\Xi=\sigi$ in these situations.
	\item The morphisms $i_1,i_2,\sigma_1,\sigma_2,p_1,p_2,p,\pi_1$ all have global versions, and we denote their global versions by the same symbols.
\end{enumerate}

\subsection{Acknowledgements}
The author sincerely thanks his advisor, Professor Peng Shan, for suggesting this problem, and for her patient guidance and encouragement throughout the work. The author thanks Zhengze Xin for his results on the graded de Rham complex, which serves as a key tool for this paper. The author thanks Professor Lin Chen for discussing many details in the paper, especially those on (derived) algebraic geometry; the author learned a lot from these discussions. The author thanks Professor Penghui Li for his discussions on an early version of the paper. The author thanks one student study group on higher algebras for teaching him a lot, especially Jinyi Wang. The author thanks Fulin Xu for discussing reflexive modules with him.

The main idea of this paper originated during the preparation of the author's bachelor thesis. He sincerely thanks his family, teachers, and friends for their unwavering support throughout his undergraduate years.

This work is supported by NSFC Grant 12225108 through the author's advisor Peng Shan.

\section{Algebro-geometric background in Cautis-Williams theory}\label{sec: AG background}

In this section, we review some algebro-geometric notions in Cautis-Williams theory, following \cite{CW23,CW24stk,CW24mor}; we only collect what we use in this paper, and refer to \textit{ibid.} for more details. We also refer to \cite[Sec. 5.2]{VV25} for a quick review.

\subsection{Ind-geometric derived stacks}

\subsubsection{DG algebras}
Let $\CAlg_{\C}$ be the $\infty$-category of non-positively graded commutative\footnote{The DG algebras below will by default be non-positively graded commutative.} DG algebras; and given $A \in \CAlg_\C$, let $\CAlg_A$ be the $\infty$-category of non-positively graded commutative DG $A$-algebras. We say $B \in \CAlg_A$ is \textit{$n$-truncated} if $H^k(B) = 0$ for all $k < -n$. Let $\tau_{\le n}\CAlg_A$ denote the full $\infty$-subcategory of $\CAlg_A$ consisting of $n$-truncated DG algebras. There is a truncation functor $\tau_{\le n}:\CAlg_A\to \tau_{\le n}\CAlg_A$. We say $B\in\CAlg_A$ is \textit{truncated} if it is $n$-truncated for some $n$.

Let $B\in \CAlg_A$ be a DG $A$-algebra.
\begin{itemize} 
	\item It is called \textit{finitely $n$-presented} if it is a compact object in $\tau_{\le n}\CAlg_A$, i.e. $\mathrm{Hom}_A(B, -)$ commutes with arbitrary direct sums. It is \textit{almost finitely presented} if its truncation $\tau_{\le n}B$ is finitely $n$-presented for all $n$. 
	\item It is called \textit{strictly tamely $n$-presented} if it is a filtered colimit of finitely $n$-presented DG $A$-algebras $B_{\alpha}$ such that $B$ is flat over each $B_{\alpha}$. It is \textit{strictly tamely presented} if $\tau_{\le n}B$ is strictly tamely $n$-presented for all $n$. A typical example of tamely presented $A$-algebra is 
	\begin{equation}\label{eq:Ainfty}
		\C[\mathbb{A}^{\infty}_A]:=A[x_i,i\in \mathbb{N}]\cong \colim_n A[x_1,\dots,x_n].
	\end{equation}
\end{itemize}

\subsubsection{Derived stacks}
We first introduce some basic notions in derived algebraic geometry following \cite[Ch. 2]{GR17}. Let $\Grpd$ be the $\infty$-category of $\infty$-groupoids, and $\Cat_{\infty}$ be the $\infty$-category of $\infty$-categories. A derived (resp. classical) stack (implicitly over $\C$) is a functor from $\CAlg_{\C}$ (resp. $\cCAlg_{\C}$) to $\Grpd$ which is a sheaf for the fpqc topology. Let $\Stk_{\C}$ and $\Stk_{\C}^{\cl}$ be the $\infty$-category of derived and classical\footnote{Note that the notion for classical stacks here refers to $\infty$-stacks in much of literature.} stacks respectively. The embedding $\cCAlg_{\C}\hookrightarrow \CAlg_{\C}$ gives a restriction functor $\Stk_{\C}\to \Stk_{\C}^{\cl}$. This functor has a fully faithful left adjoint $\Stk_{\C}^{\cl}\hookrightarrow \Stk_{\C}$, thus we can view classical stacks as derived stacks. Write $\cl:\Stk_{\C}\to \Stk_{\C}^{\cl}\hookrightarrow \Stk_{\C}$ for the composition. 

The Yoneda embedding gives a fully faithful embedding $\Spec:\CAlg_{\C}\to \Stk_{\C}$, with objects in the image called affine derived schemes. Let $\Sch_{\C}\subset \Stk_{\C}$ be the full $\infty$-subcategory of derived schemes; this $\infty$-subcategory is stable under fiber products in $\Stk_{\C}$. Let $\cSch_{\C}$ be the (ordinary) category of classical schemes; we can view it as a full subcategory of $\Stk^{\cl}_{\C}$, and further a full subcategory of $\Stk_{\C}$. The fully faithful embedding $\Stk^{\cl}_{\C}\hookrightarrow \Stk_{\C}$ restricts to $\cSch_{\C}\hookrightarrow \Sch_{\C}$, thus we can view classical schemes as derived schemes; the composition functor $\cl$ also restricts to $\cl:\Sch_{\C}\to \cSch_{\C}\hookrightarrow \Sch_{\C}$.

A derived stack $X$ is called \textit{convergent} if $X(A)\cong \lim_n X(\tau_{\le n}A)$ for any $A\in \CAlg_{\C}$. Let $\wh{\Stk}_{\C}\subset \Stk_{\C}$ denote the full $\infty$-subcategory of convergent derived stacks.

A derived stack $X$ is called \textit{geometric} if its diagonal $X\to X\times X$ is affine and there exists a faithfully flat morphism $\Spec B\to X$ in $\Stk_{\C}$ (c.f. \cite[Ch. 9]{Lur18SAG}). According to \cite[Prop. 3.25]{CW24stk}, a geometric derived stack is convergent. A geometric derived stack $X$ is called \textit{truncated} if $A$ is truncated for any flat morphism $\Spec A\to X$, or equivalently, if there exists a truncated faithfully flat morphism $\Spec A\to X$ (c.f. \cite[Prop. 9.1.6.1]{Lur18SAG}). Let $\GStk_{\C}\subset {\Stk}_{\C}$ denote the full $\infty$-subcategory of geometric derived stacks, and $\GStk_{\C}^+\subset \GStk_{\C}$ denote the full $\infty$-subcategory of truncated geometric derived stack. 
\begin{example}\label{example:quotient_geom_stk}\ 
	\begin{itemize}
		\item A quasi-compact and semi-separated derived scheme $X$ is geometric;
		\item Furthermore, let $G$ be a classical affine group scheme acting on $X$, then the fpqc quotient derived stack $[G\backslash X]$ is geometric;
		\item $\GStk_{\C}$ is closed under fiber products in $\Stk_{\C}$. 
	\end{itemize}
\end{example}

\subsubsection{Some conventions about morphisms}\label{subsubsec:mor}
Let $f:X\to Y$ be a morphism between derived stacks. 
\begin{itemize}
	\item It is a \textit{closed immersion} (or called \textit{closed embedding}) if for any $\Spec A \to Y$, the morphism $\tau_{\le 0}(X\times_Y\Spec A) \to \tau_{\le 0}(\Spec A)$ is a closed immersion of ordinary affine schemes.
	\item It is \textit{proper} if for any morphism $\Spec A \to Y$ the fiber product $X \times_Y \Spec A$ is proper over
	$\Spec A$ in the sense of \cite[Def. 5.1.2.1]{Lur18SAG}.
	\item It is \textit{(locally) almost finitely presentated}\footnote{As in \cite{CW24stk}, we omit the word locally by default, as all morphisms we consider will be quasi-compact or effectively so.} if for any $n$ and any filtered colimit $A\cong \colim A_{\alpha}$ in $\tau_{\le n}\CAlg_{\C}$, the canonical map $$\colim X(A_{\alpha})\to X(A)\times_{Y(A)}\colim Y(A_{\alpha})$$ is an isomorphism (c.f. \cite[Def. 17.4.1.1]{Lur18SAG}). For example, if $X$ and $Y$ are Noetherian schemes and $f$ is finitely presented in classical sense, then $f$ is almost finitely presented.
	\item It is \textit{geometric} if, for any morphism $\Spec A\to Y$, the fiber product $X\times_{Y}\Spec A$ is geometric. 
	\item It is \textit{strictly tamely presented} if it is affine and for any $\Spec A\to Y$, the coordinate ring of $X\times_Y\Spec A$ is strictly tamely presented. It is \textit{tamely presented} if it is geometric and for any $\Spec A\to Y$ there is a strictly tamely presented flat cover $\Spec B \to X \times_Y \Spec A $ such that $B$ is a strictly tamely presented $A$-algebra. (c.f. \cite[Def. 4.6]{CW24mor})
\end{itemize}
These properties of morphisms are stable under composition and base change in $\Stk_{\C}$ (for the tamely presented property, see \cite[Prop. 4.9]{CW24mor}). 

A geometric derived stack $X$ is said to be tamely presented if the map $X\to \Spec \C$ is tamely presented.

\begin{example} \ 
\begin{itemize}
    \item If $X,Y$ are geometric then $f:X\to Y$ is also geometric.
    \item If $f$ is representable (by spectral Deligne-Mumford stack), geometric, and almost finitely presented, then it is tamely presented (c.f. \cite[Prop. 4.7]{CW24mor}).
    \item Let $G$ be a classical affine group scheme acting on a tamely presented derived scheme $X$, then the fpqc quotient derived stack $[G\backslash X]$ is tamely presented (c.f. \cite[Prop. 4.11]{CW24mor}).
\end{itemize}
\end{example}

\subsubsection{Ind-geometric derived stacks}
An \textit{ind-geometric derived stack} is a convergent derived stack $X$ which admits an expression $X\cong \colim_{\alpha}X_{\alpha}$ as filtered colimit of truncated geometric derived stacks along closed immersions in the $\infty$-category $\wh{\Stk}_{\C}$ of {convergent derived stacks}. We also call $\colim_{\alpha}X_{\alpha}$ an ind-geometric presentation of this ind-geometric derived stack.

\begin{example}\label{example: X/G}\ 
    \begin{itemize}
    	\item In the definition of ind-geometric derived stack, if all $X_{\alpha}$ are truncated quasi-compact and semi-separated derived schemes, we call $X$ an \textit{ind-derived scheme}. An ind-derived scheme is an ind-geometric derived stack.
    	\item The fpqc quotient $[G\backslash X]$ of the ind-derived scheme by a classical affine group scheme $G$ is also an ind-geometric derived stack.
    	
    	\noindent In fact, assume $X_{\alpha}$ are $G$-invariant, then $[G\backslash X]\cong \colim_{\alpha}[G\backslash X_{\alpha}]$ gives $[G\backslash X]$ a presentation as an ind-geometric derived stack (c.f. \cite[Sec. 4.2]{CW24stk}).
    \end{itemize}  
\end{example}

A morphism $f:X\to Y$ between ind-geometric derived stacks is \textit{ind-proper} (resp. \textit{ind-closed immersion}, \textit{almost ind-finitely presented}) if for some (or equivalently any) ind-geometric presentation $X\cong\colim_{\alpha}X_{\alpha}$ and for any $X_{\alpha}$, we can choose a truncated geometric closed substack $Y_\alpha \hookrightarrow Y$, such that $f$ factors through some $f_{\alpha}:X_{\alpha}\to Y_{\alpha}$ and each $f_{\alpha}$ is proper
(resp. a closed immersion, almost finitely presented). These properties of morphisms are stable under composition and base change in $\Stk_{\C}$; see \cite[Sec. 4]{CW24stk} for more details.

An ind-geometric derived stack $X\cong\colim_{\alpha}X_{\alpha}$ is \textit{reasonable} if the structure maps are almost finitely presented. It is \textit{ind-tamely presented} if it is reasonable and all the $X_{\alpha}$ are tamely presented.

\subsection{Coherent sheaves theory}\label{subsec: coh shvs}
\subsubsection{}\label{subsubsec:coh}
For a DG algebra $A\in\CAlg_{\C}$, let $\Mod_A$ be the stable $\infty$-category of DG $A$-modules; $\Mod_A$ admits a standard t-structure. We say an $A$-module $M$ is coherent if it is bounded and almost perfect (i.e. $\tau^{\ge n}M$ is compact in $\Mod_A^{\ge n}$ for all $n$).

For a derived stack $X$, let $\QCoh(X)$ be the stable $\infty$-category of quasi-coherent sheaves on $X$. It is the limit of the $\infty$-categories $\Mod_R$ over all morphisms $\Spec R\to X$. It admits a standard t-structure, and the heart is denoted by $\QCoh(X)^{\heartsuit}$. If $X$ is geometric, then $\QCoh(X)$ is equivalent to the corresponding limit over the Cech nerve of any faithfully flat cover (c.f. \cite[Prop. 9.1.3.1]{Lur18SAG}). 

When $X$ is a truncated geometric derived stack, let $\Coh(X)\subset \QCoh(X)$ denotes the full $\infty$-subcategory of coherent sheaves, i.e. those $\cF\in \QCoh(X)$ such that its pullback along some (equivalently, any) faithfully flat cover $\Spec A\to X$ is coherent. 

A geometric derived stack $X$ is \textit{locally coherent} (resp. \textit{locally Noetherian}, resp. \textit{flat locally almost finitely presented}) if it admits a faithfully flat cover $\Spec A\to X$ such that $A$ is coherent (resp. Noetherian, resp. almost finitely presented over $\pt=\Spec \C$). It is \textit{admissible} if it admits an affine morphism to a geometric derived stack which is flat locally almost finitely presented.\footnote{Here, the definition of admissible property is different to \cite[Def. 4.12]{CW24mor}, where a geometric derived stack $X$ is called admissible if it admits an affine morphism to a locally Noetherian geometric derived stack. 

An important feature for admissible property is that the class of admissible geometric derived stack should be closed under products over $\pt$ in $\Stk_{\C}$. But it's not clear whether the admissible property defined in \cite[Def. 4.12]{CW24mor} is closed under products over $\pt$, since Noetherian affine schemes are not closed under products over $\pt$. Note that almost finitely presented (over $\pt$) affine schemes are closed under products over $\pt$, and are Noetherian by Hilbert basis theorem \cite[Prop. 7.2.4.31]{Lur17HA}. Thus we have defined a subclass of geometric derived stacks which is closed under products over $\pt$, and satisfies the conditions in \cite[Def. 4.12]{CW24mor}. Although our definition here is more strict, the other instances where it is used in \cite{CW24mor,CW23} all meet the requirements of this stricter version.

The author thanks Prof. Lin Chen for pointing out this issue.} 
Admissible geometric derived stacks are closed under products over $\pt$, since flat locally almost finitely presentated geometric derived stacks are closed under products over $\pt$.

The standard t-structure $\QCoh(X)$ restricts to the one on $\Coh(X)$ if $X$ is locally coherent. An admissible tamely presented geometric derived stack $X$ is \textit{coherent}, which means $X$ is locally coherent and $\QCoh(X)^{\heartsuit}$ is compactly generated (by $\Coh(X)^{\heartsuit}$); this is deduced from \cite[Prop. 9.5.2.3]{Lur18SAG} together with \cite[Lem. 4.26]{CW24stk}. 

We also remark that an ind-geometric derived stack is \textit{{admissible}} if it admits an ind-geometric presentation whose terms are admissible geometric derived stacks, and it is \textit{coherent} if it admits a reasonable presentation whose terms are coherent geometric derived stacks. Admissible ind-tamely presented ind-geometric derived stacks provide a closed under products (over $\pt$) subclass of coherent ind-geometric derived stacks, while the whole class of coherent ind-geometric derived stacks is not closed under products.

\begin{example}\label{exam:QCoh(X/G)}
	Let $G$ be a classical affine group scheme acting on a quasi-compact and semi-separated derived scheme $X$ as in the quotient derived stack example \ref{example:quotient_geom_stk}. We write $\QCoh^G(X)$ (resp. $\Coh^G(X)$) for $\QCoh([G\backslash X])$ (resp. $\Coh([G\backslash X])$). Then using descent procedure, $\QCoh^G(X)^{\heartsuit}$ (resp. $\Coh^G(X)^{\heartsuit}$) is the abelian category of $G$-equivariant quasi-coherent (resp. coherent) ordinary sheaves on $X$. If $X$ is classical, then the bounded below part of $\QCoh^G(X)$ can be described by the bounded below derived category, i.e. we have a natural equivalence $$\mathsf{D}^+(\QCoh^G(X)^{\heartsuit}){\cong} \QCoh^G(X)^+;$$ 
	see for example \cite[Prop. 2.4.3]{GR17}. Under this equivalence, $\Coh^G(X)\subset \QCoh^G(X)^+$ corresponds to the full $\infty$-subcategory of $\mathsf{D}^+(\QCoh^G(X)^{\heartsuit})$ consisting of bounded complexes with coherent cohomologies when $X$ is locally coherent.
\end{example}

\subsubsection{Pull-push functors}\label{subsubsec:coh_geo}
Given a morphism $f:X\to Y$ in $\Stk_{\C}$, we have a natural pullback functor $f^*:\QCoh(Y)\to \QCoh(X)$. It admits a right adjoint $f_*:\QCoh(X)\to \QCoh(Y)$. 

Let $f:X\to Y$ be a morphism in $\GStk_{\C}$. 
\begin{itemize}
	\item It is \textit{of Tor-dimension $\le n$} if $f^*(\QCoh(Y)^{\ge 0})\subset \QCoh(Y)^{\ge n}$, and it is \textit{of finite Tor-dimension} if it is of Tor-dimension $\le n$ for some $n$.
	\item It is \textit{of cohomological dimension $\le n$} if $f_*(\QCoh(X)^{\le 0})\subset \QCoh(Y)^{\le n}$, and it is \textit{of finite cohomological dimension} if it is of cohomological dimension $\le n$ for some $n$.
\end{itemize}
Morphisms of finite Tor-dimension (resp. of finite cohomological dimension) are stable under base change in $\GStk_k$ and stable under composition. Morphisms of finite cohomological dimension are further flat on the target. (c.f. \cite[Prop. 3.6, Prop. 3.14]{CW24stk})

Let the following be a Cartesian diagram of geometric derived stacks
% https://q.uiver.app/#q=WzAsNCxbMSwwLCJZJyJdLFswLDEsIlgiXSxbMSwxLCJZIl0sWzAsMCwiWCciXSxbMSwyLCJmIl0sWzAsMiwiaCJdLFszLDAsImYnIl0sWzMsMSwiaCciLDJdXQ==
\begin{equation}\label{eq:bc}
	\begin{tikzcd}
		{X'} & {Y'} \\
		X & {Y.}
		\arrow["{f'}", from=1-1, to=1-2]
		\arrow["{h'}"', from=1-1, to=2-1]
		\arrow["h", from=1-2, to=2-2]
		\arrow["f", from=2-1, to=2-2]
	\end{tikzcd}
\end{equation}
If $h$ is of finite Tor-dimension (resp. $f$ is of finite cohomological dimension), then the Beck-Chevalley map 
\begin{equation}\label{eq:bcc}
	h^*f_*(\cG)\to f'_*{h'}^*(\cG)
\end{equation}
is an isomorphism for all $\cG\in \QCoh(X)^+$ (resp. $\cG\in \QCoh(X)$)

Let $f:X\to Y$ be a morphism of truncated geometric derived stacks. 
\begin{itemize}
	\item If $f$ is of finite Tor-dimension, then $f^*:\QCoh(Y)\to \QCoh(X)$ restricts to $f^*:\Coh(Y)\to \Coh(X)$ (\cite[Prop. 3.6]{CW24stk}).
	\item If $f$ is proper and almost finitely presented, then $f_*:\QCoh(X)\to \QCoh(Y)$ restricts to $f_*:\Coh(X)\to \Coh(Y)$ (\cite[Prop. 3.19]{CW24stk}).
\end{itemize}
In particular, in above diagram \eqref{eq:bc}, if $h$ is of finite Tor-dimension, $f$ is proper and almost finitely presented, then we can apply above base change isomorphism \eqref{eq:bcc} to coherent sheaves\footnote{These base change isomorphisms satisfy natural compatibilities. Following \cite{GR17}, all the information can be packaged into a functor $\Coh:\mathsf{Corr}(\GStk)_{prop;ftd}\to \Cat_{\infty}$. See \cite[Sec. 3.6]{CW24stk} for more details.}.

\subsubsection{Coherent sheaf theory for ind-geometric derived stacks}\label{subsub:coh_ind}
\cite[Def. 5.3]{CW24stk} extends coherent sheaf theory to reasonable ind-geometric derived stacks\footnote{According to \cite[Def. 5.3]{CW24stk}, they define $\Coh:\mathsf{Corr}(\indGStk_k^{reas})_{prop,ftd}\to \Cat_{\infty}$ by left Kan extention along $\mathsf{Corr}(\GStk_{\C}^+)_{prop;ftd}\subset\mathsf{Corr}(\indGStk_k^{reas})_{prop;ftd} $.}. \cite[Prop. 5.5]{CW24stk} implies that given a reasonable ind-geometric presentation $X\cong\colim_{\alpha} X_{\alpha}$, then $\Coh(X)$ can be computed by $$\colim_{\alpha} \Coh(X_{\alpha})$$ in $\Cat_{\infty}$. In particular any $\cF\in \Coh(X)$ can be written as $\cF\cong i_{\alpha*}\cF_{\alpha}$ for some $i_{\alpha}:X_{\alpha}\hookrightarrow X$ and $\cF_{\alpha}\in \Coh(X_{\alpha})$.

Let $f:X\to Y$ be a morphism of reasonable ind-geometric derived stacks.
\begin{itemize}
	\item It is of Tor-dimension $\le n$ (resp. of finite Tor-dimension) if it is geometric and its base change to any geometric derived stack is of Tor-dimension $\le n$ (resp. of finite Tor-dimension). 
	
	\noindent If $f$ is of finite Tor-dimension, then \cite[Def. 5.3]{CW24stk} gives the pullback functor $f^*:\Coh(Y)\to \Coh(X)$: for $\cF=i_{\alpha *}\cF_{\alpha}\in \Coh(X)$, $f^*\cF$ can be computed by $i'_{\alpha_*}h_{\alpha}^*(\cF_{\alpha})$, where $i'_{\alpha}$ and
	$h_{\alpha}$ are defined by base change from $i_{\alpha}$ and $h$.
	\item If $f$ is ind-proper and almost ind-finitely presented, then \cite[Def. 5.3]{CW24stk} gives the pushforward functor $ f_*:\Coh(X)\to \Coh(Y) $: for $\cF=i_{\alpha *}\cF_{\alpha}$, $f_*\cF$ can be computed by $j_{\alpha*}f_{\alpha*}\cF_{\alpha}$, where $X_{\alpha}\xrightarrow{f_{\alpha}}Y_{\alpha}\xrightarrow{j_{\alpha}}Y$ is any factorization of $f_{\alpha}\circ i_{\alpha}$ through a geometric substack $j_{\alpha}:Y_{\alpha}\hookrightarrow Y$ such that $j_{\alpha}$ is an almost finitely presented closed immersion.
\end{itemize}
Similar to the geometric case, we also have base change isomorphisms for pullback functors (along of finite Tor-dimension morphisms) and pushforward functors (along ind-proper and almost ind-finitely presented morphisms).

\subsubsection{Ind-coherent sheaves}\label{subsubsec:indcoh}
We will also encounter ind-coherent sheaves (e.g. \eqref{def of KPlamu'}). 

For a geometric derived stack $X$, the stable $\infty$-category of ind-coherent sheaves $\IndCoh(X)$ is defined to be the left anticompletion of $\QCoh(X)$ (c.f. \cite[Def. 5.10, Def. 5.13]{CW24stk}). There is a natural functor $\Psi_X:\IndCoh(X)\to \QCoh(X)$ which restricts to an equivalence for the bounded below part of these two stable $\infty$-categories with their natural t-structures. When $X$ is admissible and tamely presented (or more generally when $X$ is coherent), then $\IndCoh(X)$ is equivalent to the ind-completion of $\Coh(X)$ (c.f. \cite[Prop. 5.30]{CW24stk}\footnote{We also implicitly use \cite[Prop. 4.3]{CW24stk}.}). 

This ind-coherent sheaf theory also contains the information of pullback functors along finite Tor-dimension morphisms and pushforward functors along finite cohomological dimension morphisms, together with the base change isomorphisms (c.f. \cite[Def. 5.13]{CW24stk})\footnote{It can be summarized as a functor $\IndCoh:\mathsf{Corr}(\GStk_{\C})_{fcd;ftd}\to \Cat_{\infty}$.}. 

This ind-coherent sheaf theory can also be extended to ind-geometric derived stacks (c.f. \cite[Def. 5.16]{CW24stk})\footnote{It can be summarized as a functor $\IndCoh:\mathsf{Corr}(\indGStk_{\C})_{fcd;ftd}\to \Cat_{\infty}$.}. When $X$ is admissible and ind-tamely presented (or more generally when $X$ is coherent), then $\IndCoh(X)$ is also equivalent to the ind-completion of $\Coh(X)$ (c.f. \cite[Prop. 5.30]{CW24stk}).

\subsubsection{Stable coherent pullback}
We will need pullback functors for a class of morphisms extending the class of finite Tor-dimension morphisms. 

Let $f:X\to Y$ be a morphism of truncated geometric derived stacks. It is said to have \textit{coherent pullback} if $f^*:\QCoh(Y)\to \QCoh(X)$ takes $\Coh(Y)$ to $\Coh(X)$. It is said to have \textit{stable coherent pullback} if for any truncated $Y'$ and any tamely presented morphism $Y'\to Y$, the base change $f':X\times_Y Y'\to Y'$ has coherent pullback. Morphisms with stable coherent pullback are stable under composition and under base change along tamely presented morphisms in $\GStk^+$ (c.f. \cite[Prop. 4.14]{CW24mor}). 

\begin{example}\label{exam:stable_coh_pull}
	A typical example is the embedding $i:0\hookrightarrow \mathbb{A}^{\infty}_{\C}$. This has coherent pullback because any coherent sheaf on $\A_{\C}^{\infty}$ is the pullback of a coherent sheaf on some $\A_{\C}^n$ along the flat projection $\A_{\C}^{\infty}\to \A_{\C}^n$ (c.f. \cite[Exam. 3.10]{CW24mor}). Then according to \cite[Thm. 3.11]{CW24mor}\footnote{See the comments after \cite[Prop. 4.14]{CW24mor}.}, $i$ has stable coherent pullback.
	
	According to \cite[Prop. 4.14]{CW24mor}, a morphism $f:X\to Y$ in $\GStk_{\C}^+$ has stable coherent pullback if its base change along a faithfully flat cover has stable coherent pullback. Thus if the base change $f':X'\to Y'$ along some faithfully flat cover $Y'\to Y$ is isomorphic to the zero section embedding $X'\hookrightarrow X'\times \A^{\infty}_{\C}$, then $f$ has stable coherent pullback. 
\end{example}

Similarly, let $f:X\to Y$ be a morphism of ind-geometric derived stacks such that $Y$ is reasonable. Let $\colim_{\alpha} Y_{\alpha}$ be a reasonable presentation. We say $f$ has stable coherent pullback if its base change to every $Y_{\alpha}$ has stable coherent pullback (c.f. \cite[Prop. 5.13]{CW24mor}).

The coherent sheaf theory for truncated geometric derived stacks in Section \ref{subsubsec:coh_geo} and reasonable ind-geometric derived stacks in Section \ref{subsub:coh_ind} can be extended: we can apply pullback functors to morphisms with stable coherent pullback, and they also have base change isomorphisms with almost (ind-)finitely presentated (ind-)proper pushforward, c.f. \cite[Sec. 5.5]{CW24mor}\footnote{They can be summarized as functors $ \Coh: \mathsf{Corr}(\GStk^+_{\C})_{prop;coh}\to \Cat_{\infty}$, $\Coh:\mathsf{Corr}(\indGStk^{reas}_{\C})_{prop;coh}\to \Cat_{\infty}$}.

Let $X,Y$ be coherent ind-geometric derived stacks. In this case $\IndCoh(X)$ (resp. $\IndCoh(Y)$) is the ind-completion of $\Coh(X)$ (resp. $\Coh(Y)$). Then for a morphism $f:X\to Y$ with coherent pullback, $f^*:\IndCoh(Y)\to \IndCoh(X)$ can be defined to be the unique continuous functor whose restriction to $\Coh(Y)$ factors through the functor $f^*:\Coh(Y)\to \Coh(X)$.

\subsubsection{$!$-pullback}\label{subsubsec:!-pull}
Let $f:X\to Y$ be an ind-proper morphism of ind-geometric derived stacks. By the construction of ind-coherent sheaf theory (\cite[Def. 5.16]{CW24stk}), $f_*:\IndCoh(X)\to \IndCoh(Y)$ admits a right adjoint, which is denoted by $f^!$. 

We will write
$$\omega_{X/Y}:=f^!\O_{Y}.$$
When $Y=\pt$, we simplify the notation and write $\omega_{X}:=\omega_{X/\pt}$.

According to \cite[Prop. 6.17]{CW24mor}, $!$-pullback is compatible with stable coherent pullback. Let the following be a Cartesian diagram of ind-geometric derived stacks:
	\[
	\begin{tikzcd}
		X' \ar[r, "f'"] \ar[d, "h'"'] & Y' \ar[d, "h"] \\
		X \ar[r, "f"] & {Y,}
	\end{tikzcd}
	\]
such that all derived stacks in this diagram are coherent and ind-tamely presented, that $h$ is of ind-finite cohomological dimension and has stable coherent pullback, and that $f$ is ind-proper and almost ind-finitely presented. Then for any \(\mathcal{F} \in \IndCoh(Y)\), we have a Beck-Chevalley isomorphism 
\begin{equation}\label{eq:!and*}
	h'^* f^! (\mathcal{F}) \to f'^! h^* (\mathcal{F}) .
\end{equation}

\subsection{Some other functors}

\subsubsection{External product}
Let $X,Y$ be two derived stack, define the external product for quasi-coherent sheaves
\[\boxtimes:\QCoh(X)\times \QCoh(Y)\to \QCoh(X\times Y)\]
by $\cF\boxtimes \cG:=p_{X}^*(\cF)\otimes p_{Y}^*(\cG)$, where $p_{X}:X\times Y\to X$ and $p_{Y}:X\times Y\to Y$ are two projections. 

If $X$ and $Y$ are truncated geometric derived stacks, and $\cF\in \Coh(X)$, $\cG\in \Coh(Y)$, then $X\times Y$ is truncated geometric derived stack and $\cF\boxtimes \cG\in \Coh(X\times Y) $ (c.f. \cite[Prop. 7.9]{CW24stk}).

In \cite[Sec 7.3]{CW24stk}, they defined external product for coherent sheaves on reasonable ind-geometric derived stacks. Let $X\cong \colim_{\alpha}X_{\alpha}$ and $Z\cong \colim_{\alpha}Z_{\alpha}$ be reasonable presentations. \cite[Sec. 7.3]{CW24stk} shows that $X\times Y$ is also a reasonable ind-geometric derived stack. The external product $\boxtimes: \Coh(X)\times \Coh(Y)\to \Coh(X\times Y)$ can be characterized by fitting into a diagram
% https://q.uiver.app/#q=WzAsNCxbMSwwLCJcXENvaChYX3tcXGFscGhhfVxcdGltZXMgWl97XFxiZXRhfSkiXSxbMCwxLCJcXENvaChYKVxcdGltZXMgXFxDb2goWikiXSxbMSwxLCJcXENvaChYXFx0aW1lcyBaKSJdLFswLDAsIlxcQ29oKFhfe1xcYWxwaGF9KVxcdGltZXMgXFxDb2goWl97XFxiZXRhfSkiXSxbMSwyLCJcXGJveHRpbWVzIiwyXSxbMCwyLCIoaV97XFxhbHBoYX1cXHRpbWVzIGlfe1xcYmV0YX0pXyoiXSxbMywxLCJpX3tcXGFscGhhKn1cXHRpbWVzIGlfe1xcYmV0YSp9IiwyXSxbMywwLCJcXGJveHRpbWVzIl1d
\[\begin{tikzcd}
	{\Coh(X_{\alpha})\times \Coh(Z_{\beta})} & {\Coh(X_{\alpha}\times Z_{\beta})} \\
	{\Coh(X)\times \Coh(Z)} & {\Coh(X\times Z)}
	\arrow["\boxtimes", from=1-1, to=1-2]
	\arrow["{i_{\alpha*}\times i_{\beta*}}"', from=1-1, to=2-1]
	\arrow["{(i_{\alpha}\times i_{\beta})_*}", from=1-2, to=2-2]
	\arrow["\boxtimes"', from=2-1, to=2-2]
\end{tikzcd}\]
for all $\alpha,\beta$.

In \cite{CW24stk}[Sec. 7], they studied external product for other sheaf theories, see \textit{ibid.} for more details.

\subsubsection{Twisted product}\label{subsubsec:twisted_prod}
Let $H$ be a classical affine group scheme over $\C$, and $P\to X$ be an $H$-torsor, where $P$ and $X$ are both reasonable ind-geometric derived stacks over $\C$. Let $Y$ be a reasonable ind-geometric derived stack with an $H$-action. Define the twisted product as follows:
\[X\ttimes Y:=P\times^H Y:=[H\backslash(P\times Y)],\]
which is the fpqc quotient of $P\times Y$ by the diagonal action of $H$. It can also be defined as the fiber product
% https://q.uiver.app/#q=WzAsNCxbMCwwLCJYXFx0aW1lc15IIFkiXSxbMCwxLCJYIl0sWzEsMSwiW0hcXGJhY2tzbGFzaFxccHRdIl0sWzEsMCwiW0hcXGJhY2tzbGFzaCBZXSJdLFswLDEsIlxcbWF0aHJte3ByfV8xIiwyXSxbMSwyXSxbMCwzLCJcXG1hdGhybXtwcn1fMiJdLFszLDJdXQ==
\begin{equation}\label{eq:diagram_ttimes}
	\begin{tikzcd}
		{X\times^H Y} & {[H\backslash Y]} \\
		X & {[H\backslash\pt];}
		\arrow["{\rq}", from=1-1, to=1-2]
		\arrow["{\rp}"', from=1-1, to=2-1]
		\arrow[from=1-2, to=2-2]
		\arrow[from=2-1, to=2-2]
	\end{tikzcd}
\end{equation}
here the morphism $X\to [H\backslash \pt]$ is determined by the $H$-torsor $P\to X$.

Let $\cF\in \Coh(X),\cG\in \Coh^H(Y)=\Coh([H\backslash Y])$. 
Consider the natual projection $\mathbf{p}=\rp\times\rq:X\ttimes Y=[H\backslash(P\times Y)] \to X\times [H\backslash Y]$. Then we can define the twisted (external) product $\cF\tbox \cG\in \Coh(X\ttimes Y)$ as follows:
\[\cF\tbox \cG := \mathbf{p}^*(\cF\boxtimes \cG)=\rp^*(\cF)\otimes \rq^*(\cG).\]

\begin{remark}
	The constructions $\boxtimes,\tbox$ above can be defined for general smooth base scheme rather than $\Spec \C$, e.g. $X^I$, where $X$ is a smooth algebraic curve, $I$ is a finite set.
\end{remark}

\subsubsection{Sheaf Hom}\label{subsubsec:cHom}
Given a geometric derived stack $Y$ and $\cF\in \QCoh(Y)$, the sheaf Hom $$\cHom(\cF,-):\QCoh(Y)\to \QCoh(Y)$$ 
is defined by the adjunction
\[-\otimes \cF:\leftrightarrows : \cHom(\cF,-).\]

If $Y$ is a reasonable ind-geometric derived stack and $\cF\in \Coh(Y)$, there is still a natural functor $$\cHom(\cF,-) : \IndCoh(Y) \to \IndCoh(Y),$$ 
despite the absence of a tensor product of ind-coherent sheaves in general, see \cite[Eq. 7.38]{CW24stk}.

There are sevaral compatibilities of $\cHom$ with other functors. Let $X$ and $Y$ be admissible, ind-tamely presented ind-geometric derived stacks\footnote{These conditions guarantee that $X$, $Y$, $X\times X$, $Y\times Y$ are coherent (see the comments after \cite[Def. 5.11]{CW24mor}), which are used in this two propositions.}.
\begin{itemize}
	\item (\cite[Prop. 7.52]{CW24stk}) Let $f : X \to Y$ be an ind-proper, almost ind-finitely presented morphism. Then for any $\cF\in \Coh(X)$ and $\cG \in  \IndCoh(Y)$ the natural map 
	\begin{equation}\label{eq:f_*,f^!}
		f_* \cHom(\cF, f^!(\cG))\to
		\cHom(f_*(\cF), \cG)
	\end{equation}
	is an isomorphism.
	\item (\cite[Prop. 7.14]{CW24mor}) Let $h : X \to Y$ be a morphism with stable coherent pullback, and let $\cF \in \Coh(Y)$, $\cG \in \IndCoh(Y)$. Then the Beck–Chevalley map 
	\begin{equation}\label{eq:h^*cHom}
		h^* \cHom (\cF, \cG) \to \cHom (h^*(\cF), h^*(\cG))
	\end{equation}
	is an isomorphism.
\end{itemize}

\section{Categorified Coulomb branch}\label{sec:CW_theory}
From now on, we fix a connected reductive group $G$ over $\C$. We also fix a Borel subgroup $B$ and a maximal torus $T$ in $B$. The Lie algebra of $G$ is denoted by $\g$. Let $\Pv=X_*(T)$ (resp. $P=X^*(T)$) be the cocharacter (resp. character) lattice, and ${\check{\Phi}}$ (resp. $\Phi$) be the root set (resp. coroot set) of $G$. The quadruple $({\Pv},P,{\check{\Phi}},\Phi)$ is called the root datum of $G$. Let ${\check{G}}$ be the Langlands dual group, which is a connected reductive group over $\C$ with opposite root datum $(P,{\Pv},\Phi,{\check{\Phi}})$.

The choice of Borel $B$ determines the dominant cocharacters subset $\Pv_+\subset \Pv$ and the dominant characters subset $P_+\subset P$. It also determines a decomposition of ${\check{\Phi}}$ (resp. $\Phi$) into positive and negative coroots (resp. roots): ${\check{\Phi}}={\check{\Phi}}_+\bigsqcup{\check{\Phi}}_- $ (resp. $\Phi=\Phi_+\bigsqcup \Phi_-$). There is a partial order $\le$ on $\Pv$: $\lav\le\muv$ iff $\muv-\lav$ is a sum of positive coroots; similar partial order is also defined for $P$. 

For a dominant cocharacter $\lav\in {\Pv}_+$, let $V_{\lav}$ denote the irreducible representation of ${\check{G}}$ with highest weight $\lav$. For three dominant cocharacters $\lav,\muv,\nuv\in \Pv_+$, let 
\begin{equation}\label{def of clamunu}
	\clamunu\in \mathbb{N}
\end{equation}
denote the multiplicity of $V_{\nuv}$ in $V_{\lav}\otimes V_{\muv}$. 
\subsection{Affine Grassmannian and the space of triple}\label{subsec: def of spaces}
\subsubsection{Affine grassmannian $\Gr_G$}\label{subsubsec:Gr_G}
Let $\O$ denote the formal power series ring $\C[[t]]$, and $\K$ the formal Laurent power series ring $\C((t))$.

For a (classical) stack $Y\in \Stk_{\C}^{\cl}$, define the positive loop space $Y_{\O}\in \Stk_{\C}^{\cl}$
\[Y_{\O}:\cCAlg\to \Grpd,R\mapsto Y(R[[t]])\]
and the loop space $Y_{\K}\in \Stk_{\C}^{\cl}$
\[Y_{\K}:\cCAlg\to \Grpd,R\mapsto Y(R((t))).\]
We will concentrate on the case when $Y$ is an smooth affine scheme over $\C$. In this case, $Y_{\O}$ is represented by a scheme of infinite type over $\C$. Moreover, one can define $Y_{\O/t^n}\in \Stk_{\C}^{\cl}$
\begin{equation}\label{eq:Y_O/t^n}
	Y_{\O/t^n}:\cCAlg\to \Grpd,R\mapsto Y(R[t]/(t^n)),
\end{equation}
which are represented by smooth affine schemes over $\C$; then $Y_{\O}$ is an inverse limit of $Y_{\O/t^n}$ along flat morphisms. And $Y_{\K}$ is represented by an ind-affine scheme over $\C$, which contains $Y_{\O}$ as a closed subscheme. One may refer to \cite[Sec. 1]{KV04} for more details.

Let $G$ be a connected reductive group over $\C$. Then $G_{\O}$ is an affine group scheme, and $G_{\K}$ is a group ind-affine scheme containing $G_{\O}$ as a closed subgroup. Define $\hGO:=(G_{\O}\rtimes \Gm)\times\Gm$, where the inner $\Gm$ acts by natural loop rotation on the loop variable $t$ in $\O=\C[[t]]$. For future use, we write the inner $\Gm$ by $\Gmrot$, and the outer $\Gm$ by $\Gmdil$. Similarly, we define $\hGK:=(G_{\K}\rtimes \Gmrot)\times\Gmdil $.

Define the affine Grassmannian $\Gr_G$ to be the fpqc quotient $ [G_{\K}/G_{\O}] $. It's the same as $[\hGK/\hGO]$. The affine grassmannian $\Gr_G$ is represented by an ind-projective scheme over $\C$, with a natural $\hGO$-action by multiplication on the left. The orbits of this action can be described as follows. Let $\lav\in \Pv$ be a cocharacter, then it determines a map $\K^{\times}=\Gm(\K)\to T(\K)\hookrightarrow G(\K) $, with the image of $t$ denoted by $t^{\lav}\in G(\K)$. Its image in $\Gr_G(\C)=G(\K)/G(\O)$ is also denoted by $t^{\lav}$. For a dominant cocharacter $\lav\in \Pv_+$, define $\Gr_{\lav}$ to be the $\hGO$-orbit of $t^{\lav}$. It is a locally closed subscheme of $\Gr_G$, called a (spherical) Schubert cell; it is a smooth variety over $\C$. Let $\overline{\Gr}_{\lav}$ be its closure in $\Gr_G$, called a (spherical) Schubert variety; it is a normal projective variety over $\C$. We write $j_{ \lav}:\Gr_{\lav}\hookrightarrow \overline{\Gr}_{\lav}$ for the natural open embedding, and $i_{\le \lav}:\overline{\Gr}_{\lav}\hookrightarrow\Gr$ for the natural closed embedding. The Cartan decomposition on $G(\K)$ says
\begin{itemize}
    \item $\Gr_G(\C)=\bigsqcup_{\muv \in \Pv_+}{\Gr}_{\muv}(\C) $;
    \item $\overline{\Gr}_{\lav}(\C)=\bigsqcup_{\muv\le \lav}{\Gr}_{\muv}(\C) $; thus we often write $\overline{\Gr}_{\lav}$ as $\Gr_{\le \lav}$;
    \item The reduced locus of $\Gr_G$ is: $({\Gr}_G)_{\rred}=\colim_{\lav\in \Pv}\overline{\Gr}_{\lav}$.
\end{itemize}
The dimension of $\Gr_{\lav}$ and $\overline{\Gr}_{\lav}$ are both $d_{\lav}:=\lr{2\rho,\lav}$. Note that for any positive coroot $\check{\alpha}$, the pair $\lr{2\rho,\check{\alpha}}\in 2\Z$, so if $d_{\lav} \not\equiv d_{\muv}\pmod{2}$, then $\overline{\Gr}_{\lav}\bigcap \overline{\Gr}_{\muv}=\emptyset$. This shows $\Gr_{G}=\Gr_{G}^0\bigsqcup \Gr_{G}^{1}$, where $\Gr_G^0$ is the even component: $(\Gr_{G}^0)_{\rred}=\bigcup_{\lav\in \Pv,2\mid d_{\lav}}\Gr_{\lav}$, while $\Gr_G^1$ is the odd component: $(\Gr_{G}^1)_{\rred}=\bigcup_{\lav\in \Pv,2\nmid d_{\lav}}\Gr_{\lav}$. 

We usually view $\hGK$ as a $\hGO$-torsor over $\Gr_G$.
\begin{lemma}\label{lem:Zar_triv}
	The $\hGO$-torsor $\hGK\to \Gr_G$ is Zariski locally trivial.
\end{lemma}
\begin{proof}
	This lemma follows from \cite[Rem. 1.3.8, Lem. 2.3.5]{Zhu17}.
\end{proof}

\subsubsection{The space of triples}\label{subsec:R}
Let $N$ be a finite dimensional representation of $G$. Let $N_{\O}$ (resp. $N_{\K}$) be the positive loop space (resp. loop space) of $N$. Here $N_{\O}$ is an affine scheme isomorphic to $\A^{\infty}_{\C}$ as in \eqref{eq:Ainfty}. We consider two natural Zariski-locally trivial fiber bundles over $\Gr_G$ with fibres $N_{\O}$. The first one is the product $\Gr_G\times N_{\O}$. The second one is the twisted product (see Sec. \ref{subsubsec:twisted_prod}):
\[\cT_{G,N}:=\Gr\ttimes N_{\O}:=\hGK\times^{\hGO}N_{\O}.\]
More concretely, $\cT_{G,N}$ is the fpqc quotient of $\hGK\times N_{\O}$ by the $\hGO$-action $h\cdot(g,x)=(gh^{-1},hx)$. Here the $\hGO$-action on $N_{\O}$ is given by the natural $G_{\O}$-action on $N_{\O}$, the $\Gmrot$-loop rotation on loop variable $t$ in $\O=\C[[t]]$, and the $\Gmdil$-dilation (i.e. scalar multiplication) on $N$ (hence on $N_{\O}$, $N_{\K}$).

Let $ j_1 : \mathrm{Gr}_G \times N_{\mathcal{O}} \to \mathrm{Gr}_G \times N_{\mathcal{K}} $ denote the natural embedding and let
\begin{equation}\label{eq:j_2}
	j_2 : \mathcal{T}_{G,N} \to \mathrm{Gr}_G \times N_{\mathcal{K}}, \ [g, s] \mapsto ([g], gs).
\end{equation}
These maps realize $ \mathrm{Gr}_G \times N_{\mathcal{O}} $ and $ \mathcal{T}_{G,N} $ as closed ind-subschemes of $ \mathrm{Gr}_G \times N_{\mathcal{K}} $.

\begin{remark}\label{rmk: an isom}
    There is a natural isomorphism $\Gr_G\ttimes N_{\K}\cong \Gr_G\times N_{\K}, [g,s]\to ([g],gs)$.
\end{remark}

\begin{definition}\label{def:R}
The space of triples $ \mathcal{R}_{G,N} $ is defined to be the fiber product
\begin{equation}\label{eq: R}
    \begin{tikzcd}
    \mathcal{R}_{G,N} \arrow[r, "i_1"] \arrow[d, "i_2"] & \mathrm{Gr}_G \times N_{\mathcal{O}} \arrow[d, "j_1"] \\
    \mathcal{T}_{G,N} \arrow[r, "j_2"'] & \mathrm{Gr}_G \times N_{\mathcal{K}}
    \end{tikzcd}
\end{equation}
in the $\infty$-category of ind-derived schemes.
\end{definition}
We will also write 
\begin{equation}\label{eq:RT}
	\begin{aligned}
		\cT_{\le \lav}:=\cT_{G,N}\times_{\Gr_G}\Gr_{\le \lav}&,\cT_{ \lav}:=\cT_{G,N}\times_{\Gr_G}\Gr_{\lav};\\
		\cR_{\le \lav}:=\cR_{G,N}\times_{\Gr_G}\Gr_{\le \lav}&,\cR_{ \lav}:=\cR_{G,N}\times_{\Gr_G}\Gr_{ \lav}.
	\end{aligned}
\end{equation}

The following proposition is proved in \cite{CW23}.
\begin{proposition}\label{Prop: prop of R}\ 
    \begin{enumerate}
        \item (\cite[Prop. 4.7]{CW23}) The ind-derived scheme $\cR_{G,N}$ is ind-tamely presented, and all maps in \eqref{eq: R} are almost ind-finitely presented ind-closed immersions;
        \item (\cite[Prop. 4.11]{CW23}) The fpqc quotient $[\hGO\backslash\cR_{G,N}]$ is an admissible, ind-tamely presented ind-geometric derived stack.
    \end{enumerate}
\end{proposition}

There is a useful description of $[\hGO\backslash \cR_{G,N}]$ as follows.

\begin{proposition}[{\cite[Prop. 4.8]{CW23}}]\label{prop: Rsymmdiagram}
    There is a Cartesian diagram in the $\infty$-category of derived stacks
    % https://q.uiver.app/#q=WzAsNCxbMCwxLCJcXGhHT1xcYmFja3NsYXNoIE5fe1xcT30iXSxbMSwxLCJcXGhHS1xcYmFja3NsYXNoIE5fe1xcS30iXSxbMCwwLCJcXGhHT1xcYmFja3NsYXNoIFxcY1Jfe0csTn0iXSxbMSwwLCJcXGhHT1xcYmFja3NsYXNoIE5fe1xcT30iXSxbMiwwLCJcXHBpXzIiLDJdLFszLDFdLFsyLDMsIlxccGlfMSJdLFswLDFdXQ==
	\begin{equation}\label{eq:Rsymmdiagram}
		\begin{tikzcd}
	{[\hGO\backslash \cR_{G,N}]} & {[\hGO\backslash N_{\O}]} \\
	{[\hGO\backslash N_{\O}]} & {[\hGK\backslash N_{\K}].}
	\arrow["{\pi_1}", from=1-1, to=1-2]
	\arrow["{\pi_2}"', from=1-1, to=2-1]
	\arrow[from=1-2, to=2-2]
	\arrow[from=2-1, to=2-2]
\end{tikzcd}
	\end{equation}
Here $\pi_1$ is the $\hGO$-quotient of the composition $\wt{\pi}_1\circ i_1$, where $\wt{\pi}_1$ is the projection $ \Gr \times N_\O \to N_O$; and $\pi_2$ is the composition of $i_2$ (quotient by $\hGO$) with the morphism $\wt{\pi}_2: [\hGO\backslash\cT] \to [\hGK\backslash\cT] \cong [\hGO\backslash N_\O]$. The two morphisms $\pi_1$ and $\pi_2$ are ind-proper and almost ind-finitely presented.
\end{proposition}

\begin{definition}
	The stable $\infty$-category of coherent sheaves
	\[\Coh^{\hGO}(\cR_{G,N}):=\Coh([\hGO\backslash\cR_{G,N}])\]
	is called the categorified Coulomb branch.
\end{definition}
The $K$-group of this stable $\infty$-category is called the $K$-theoretic Coulomb branch as in \cite[Rem. 3.14]{BFNII}.

We define two degree shift functors $\{1\},\lr{1}$ on $\Coh^{\hGO}(\cR_{G,N})$ using the $\Gmrot\times\Gmdil$-equivariance as follows. 
\begin{definition}\label{shift functors}

	 Let $\mathrm{wt}_n$ denotes the weight $n$ one dimensional representation of $\Gm$. Define
	 $$\{1\}:\Coh^{\hGO}(\cR_{G,N})\to \Coh^{\hGO}(\cR_{G,N})$$
	 by tensoring with $\mathrm{wt}_{-1}\in\Rep{\Gmrot}=\Coh([\Gmrot\backslash \pt])$. Define
	 $$\lr{1}:\Coh^{\hGO}(\cR_{G,N})\to \Coh^{\hGO}(\cR_{G,N})$$
	 by tensoring with $\mathrm{wt}_{-1}\in\Rep{\Gmdil}=\Coh([\Gmdil\backslash \pt])$.\footnote{The notation $\lr{1}$ here is different from \cite{CW19} and \cite{CW23}.}
	 
	 Similarly, for any derived stack $X$ over $[\Gmrot\times \Gmdil\backslash \pt]$, we can define $\{1\}$ and $\lr{1}$ functors on $\QCoh(X)$, $\Coh(X)$, $\IndCoh(X)$ by tensoring with the corresponding $\Gmrot$ and $ \Gmdil$-representations.
\end{definition}

\begin{remark}\label{rmk:Gm_importance}
	The $\Gmrot$-equivariance plays an important role in the construction of commutativity constraints and fiber functor on $\KP_0$ (see Section \ref{subsec:comm_constraint} and $\ref{subsec:fiber_functor}$), while the $\Gmdil$-equivariance plays an important role in the construction of Koszul perverse t-structures (see Section \ref{subsec:Koszul_t-structure}).
\end{remark}

In the remaining part of this section, we will abbreviate $\Gr_G,\cT_{G,N},\cR_{G,N}$ by $\Gr,\cT,\cR$, respectively.

\subsection{Convolution}\label{Subsec: convolution}
In this subsection, we will introduce the convolution construction on $\Coh^{\hGO}(\cR)$, which gives $\Coh^{\hGO}(\cR)$ a rigid monoidal structure.
\subsubsection{} 
According to Proposition \ref{eq:Rsymmdiagram}, the quotient derived stack $[\hGO\backslash\cR]$ can be identified with the fibre product 
\begin{equation*}\label{eq: identification with fibre product}
    [\hGO\backslash N_{\O}]\underset{[\hGK\backslash N_{\K}]}{\times}[\hGO\backslash N_{\O}]. 
\end{equation*}
This identification gives a natural convolution construction on $\Coh^{\hGO}(\cR_{G,N})=\Coh([\hGO\backslash\cR_{G,N}])$. More precisely, consider the convolution diagram:
\[([\hGO\backslash\cR_{G,N}])\times ([\hGO\backslash\cR_{G,N}])\xleftarrow{p_{12}\times p_{23}}[\hGO\backslash N_{\O}]\underset{[\hGK\backslash N_{\K}]}{\times}[\hGO\backslash N_{\O}]\underset{[\hGK\backslash N_{\K}]}{\times}[\hGO\backslash N_{\O}]\xrightarrow{p_{13}} [\hGO\backslash\cR_{G,N}].\]
Here $p_{12}\times p_{23}$ has stable coherent pullback, and $p_{13}$ is ind-proper and almost ind-finitely presented. Then the convolution product can be defined as
\begin{equation}\label{eq: conv1}
    \cF\conv \cG:=p_{13*}(p_{12}\times p_{23})^*(\cF\boxtimes\cG).
\end{equation}
According to \cite[Prop. 5.5, Prop. 5.7]{CW23}, the convolution product on $\Coh^{\hGO}(\cR)$ extends to a monoidal structure, whose unit object is the pushforward $e_*(\O_{[\hGO\backslash N_{\O}]})$ for the diagonal map $e:[\hGO\backslash N_{\O}]\to [\hGO\backslash N_{\O}]\underset{[\hGK\backslash N_{\K}]}{\times}[\hGO\backslash N_{\O}]$. 

We also introduce the second method in \cite[Sec. 5.2]{CW23} defining the convolution product below, which involves unwinding the above convolution diagram before taking the quotient by $\hGO$. Such diagram appeared in \cite[Eq. (3.2)]{BFNII}.

Let $\Gr \ttimes \cR$ be the twisted product $\hGK \times^{\hGO} \cR$ as in Section \ref{subsubsec:twisted_prod}; this twisted product has a natural projection map $p:\Gr \ttimes \cR\to \Gr$ onto the first factor. The twisted product $\cT \ttimes \cR$ is the base change of $\Gr \ttimes \cR$ along the projection $ \cT \to \Gr$, and $\cR \ttimes \cR$ is the further base change of $\cT\ttimes \cR$ along $i_2: \cR \to \cT$. 

Recall in Proposition \ref{prop: Rsymmdiagram} we have a map $\pi_1:[\hGO\backslash \cR]\to [\hGO\backslash N_{\O}]$. We also write $\pi_1$ for the map before taking $\hGO$-quotient. Let $\td$ for the map $\id \ttimes (\pi_1 \times \id): \Gr \ttimes \cR \to \Gr \ttimes (N_\O \times \cR) \cong \cT \ttimes \cR$. More concretely, $\td$ is the map obtained by passing to $\hGO$-quotients from
\begin{align*}
	\id \times (\pi_1 \times \id):  \hGK \times \cR &\to \hGK \times (N_\O \times \cR),\\ (g_1,[g_2,s])&\mapsto (g_1,(g_2s,[g_2,s])).
\end{align*}
We also write $\tm: \Gr \ttimes \cR \to \cT$ for the composition
$$\Gr \ttimes \cR \xrightarrow{\id \ttimes i_2} \Gr \ttimes \cT \cong \Gr \ttimes \Gr \ttimes N_\O \xrightarrow{m_{\Gr} \ttimes \id} \Gr \ttimes N_\O \cong \cT, $$ 
Here $m_{\Gr}:\Gr\ttimes\Gr\to\Gr$ is the convolution map of $\Gr$, sending $[g,[h]]$ to $[gh]$.

\cite[Prop. 5.9]{CW23} shows that there exists an ind-derived scheme $\cS$ fitting into the following diagram
	\begin{equation}\label{eq:convS}
	% https://q.uiver.app/#q=WzAsNixbMSwwLCJcXGNTIl0sWzAsMCwiXFxjUlxcdHRpbWVzIFxcY1IiXSxbMiwwLCJcXGNSIl0sWzEsMSwiXFxHciBcXHR0aW1lcyBcXGNSIl0sWzAsMSwiXFxjVFxcdHRpbWVzIFxcY1IiXSxbMiwxLCJcXGNUIl0sWzAsMSwiZCIsMl0sWzAsMiwibSJdLFszLDQsIlxcdGQiLDJdLFswLDMsImkiXSxbMyw1LCJcXHRtIl0sWzIsNSwiaV8yIl0sWzEsNCwiaV8yXFx0dGltZXMgXFxpZCIsMl1d
	\begin{tikzcd}
		{\cR\ttimes \cR} & \cS & \cR \\
		{\cT\ttimes \cR} & {\Gr \ttimes \cR} & {\cT.}
		\arrow["{i_2\ttimes \id}"', from=1-1, to=2-1]
		\arrow["d"', from=1-2, to=1-1]
		\arrow["m", from=1-2, to=1-3]
		\arrow["i", from=1-2, to=2-2]
		\arrow["{i_2}", from=1-3, to=2-3]
		\arrow["\td"', from=2-2, to=2-1]
		\arrow["\tm", from=2-2, to=2-3]
	\end{tikzcd}
	\end{equation}
in which both squares are Cartesian and all maps are $\hGO$-equivariant. They also show that the $\td$ has stable coherent pullback, and $\tm$ is ind-proper and almost ind-finitely presented; hence $d$ and $m$ have the same properties respectively. 

Let $\cF,\cG\in \Coh^{\hGO}(\cR)$. Following Section \ref{subsubsec:twisted_prod}, one can define the twisted product $\cF\tbox \cG\in \Coh^{\hGO}(\cR\ttimes \cR)$. Then, the convolution product $\conv$ on $\Coh^{\hGO}(\cR)$ can be defined as
\[\cF\conv\cG := m_*d^*(\cF\tbox \cG).\]
In \cite[Sec. 5.2]{CW23}, it is shown that this definition is the same as \eqref{eq: conv1}. The diagonal map $e:[\hGO\backslash N_{\O}]\to [\hGO\backslash N_{\O}]\underset{[\hGK\backslash N_{\K}]}{\times}[\hGO\backslash N_{\O}]$ can be identified with the $\hGO$-quotient of the $t^0$-fiber embedding $$t^0\times \id:N_{\O}\to \cR,$$ hence the unit object is identified with $(t^0\times \id)_*\O_{N_{\O}}\in \Coh^{\hGO}(\cR)$.

\subsubsection{Rigidity}\label{subsubsec:rigidity}
In \cite[Sec. 6]{CW23}, it is shown that the monoidal structure on $\Coh^{\hGO}(\cR_{G,N})$ given by convolution product is rigid. We briefly explain their constructions below.

\textbf{Firstly}, one can construct a $(-)^*$-involution on $\Coh^{\hGO}(\cR)$. Consider the morphism 
\begin{equation}\label{eq:inv}
	\mathrm{inv}:\hGK\to \hGK
\end{equation}
given by $g\mapsto g^{-1}$. This morphism commutes two $\hGO\times \hGO$-actions on $\hGK$: one is $(g_1,g_2)\cdot h=g_1hg_2^{-1}$, and the other one is $(g_1,g_2)\cdot h=g_2hg_1^{-1}$. Define 
\begin{equation}\label{eq:iotaGr}
	\iota_{\Gr}:[\hGO\backslash\hGK/\hGO]\to [\hGO\backslash\hGK/\hGO]
\end{equation}
to be the $\hGO\times \hGO$ quotient of $\mathrm{inv}:\hGK\to \hGK$, which is an involution on $[\hGO\backslash\hGK/\hGO]=[\hGO\backslash\Gr]$. After identifying $[\hGO\backslash \Gr]$ with $$ [\hGO\backslash \pt]\underset{[\hGK\backslash \pt]}{\times}[\hGO\backslash \pt]$$ as in Proposition \ref{eq:Rsymmdiagram}, $\iota_{\Gr}$ becomes the morphism exchanging the two factors\label{identification2}.

There are two $\hGO\times \hGO$-actions on $\hGK\times N_{\O}$: one is $(g_1,g_2)\cdot (h,s)=(g_1hg_2^{-1},g_1 s)$, and the other one is $(g_1,g_2)\cdot (h,s)=(g_2hg_1^{-1},g_1 s)$. The quotient of $G_{\K}\times N_{\O}$ by the first $\hGO\times \hGO$-action is naturally identified with $[\hGO\backslash(\Gr\times N_{\O})]$, and the quotient of $G_{\K}\times N_{\O}$ by the second one is naturally identified with $[\hGO\backslash\cT]$.

Consider the morphism $\mathrm{inv}\times \id_{N_{\O}}:\hGK\times N_{\O}\to \hGK\times N_{\O}$. It intertwines the two $\hGO\times \hGO$-actions on $\hGK\times N_{\O}$ described above. We have the following Cartesian diagram:
% https://q.uiver.app/#q=WzAsNCxbMCwxLCJHX3tcXEt9Il0sWzEsMSwiR197XFxLfSJdLFswLDAsIkdfe1xcS31cXHRpbWVzIE5fe1xcT30iXSxbMSwwLCJHX3tcXEt9XFx0aW1lcyBOX3tcXE99Il0sWzIsMCwiXFxtYXRocm17cHJ9XzEiLDJdLFszLDEsIlxcbWF0aHJte3ByfV8xIl0sWzIsMywie1xcbWF0aHJte2ludn19XFx0aW1lcyBOX3tcXE99Il0sWzAsMSwie1xcbWF0aHJte2ludn19IiwyXV0=
\[\begin{tikzcd}
	{G_{\K}\times N_{\O}} & {G_{\K}\times \id_{N_{\O}}} \\
	{G_{\K}} & {G_{\K}.}
	\arrow["{{\mathrm{inv}}\times N_{\O}}", from=1-1, to=1-2]
	\arrow["{\mathrm{pr}_1}"', from=1-1, to=2-1]
	\arrow["{\mathrm{pr}_1}", from=1-2, to=2-2]
	\arrow["{{\mathrm{inv}}}", from=2-1, to=2-2]
\end{tikzcd}\]

After taking quotient by the $\hGO\times\hGO$-actions, we construct a Cartesian diagram (use \cite[Lem. 4.10]{CW23}):
% https://q.uiver.app/#q=WzAsNCxbMCwwLCJcXGhHT1xcYmFja3NsYXNoKFxcR3JcXHRpbWVzIE5fe1xcT30pIl0sWzEsMCwiXFxoR09cXGJhY2tzbGFzaFxcY1QiXSxbMCwxLCJcXGhHT1xcYmFja3NsYXNoXFxHciJdLFsxLDEsIlxcaEdPXFxiYWNrc2xhc2hcXEdyIl0sWzAsMSwiXFx3dHtcXGlvdGF9Il0sWzAsMiwicF8xIiwyXSxbMiwzLCJcXGlvdGFfe1xcR3J9Il0sWzEsMywicF8yIl1d
\begin{equation}\label{diagram1}
    \begin{tikzcd}
	{[\hGO\backslash(\Gr\times N_{\O})]} & {[\hGO\backslash\cT]} \\
	{[\hGO\backslash\Gr]} & {[\hGO\backslash\Gr].}
	\arrow["{\wt{\iota}}", from=1-1, to=1-2]
	\arrow["{p_1}"', from=1-1, to=2-1]
	\arrow["{p_2}", from=1-2, to=2-2]
	\arrow["{\iota_{\Gr}}", from=2-1, to=2-2]
\end{tikzcd}
\end{equation}

We flip the above Cartesian diagram left-right:
\begin{equation}\label{diagram2}
    \begin{tikzcd}
	{[\hGO\backslash\cT]} & {[\hGO\backslash(\Gr\times N_{\O})]} \\
	{[\hGO\backslash\Gr]} & {[\hGO\backslash\Gr].}
	\arrow["{\wt{\iota}^{-1}}", from=1-1, to=1-2]
	\arrow["{p_2}"', from=1-1, to=2-1]
	\arrow["{p_1}", from=1-2, to=2-2]
	\arrow["{\iota_{\Gr}^{-1}=\iota_{\Gr}}", from=2-1, to=2-2]
\end{tikzcd}
\end{equation}

By the same method, we obtain the following Cartesian diagram 
% https://q.uiver.app/#q=WzAsNCxbMCwwLCJcXGhHT1xcYmFja3NsYXNoXFxHclxcdGltZXMgTl97XFxLfSJdLFsxLDAsIlxcaEdPXFxiYWNrc2xhc2hcXEdyXFx0aW1lcyBOX3tcXEt9Il0sWzAsMSwiXFxoR09cXGJhY2tzbGFzaFxcR3IiXSxbMSwxLCJcXGhHT1xcYmFja3NsYXNoXFxHciJdLFswLDEsIlxcaW90YV97XFxHcn1cXHRpbWVzIFxcaWQiXSxbMiwzLCJcXGlvdGFfe1xcR3J9Il0sWzEsM10sWzAsMl1d
\begin{equation}\label{diagram3}
    \begin{tikzcd}
	{[\hGO\backslash(\Gr\times N_{\K})]} & {[\hGO\backslash(\Gr\ttimes N_{\K})]\cong [\hGO\backslash(\Gr\times N_{\K})]} \\
	{[\hGO\backslash\Gr]} & {[\hGO\backslash\Gr],}
	\arrow["{\wt{\iota}}", from=1-1, to=1-2]
	\arrow[from=1-1, to=2-1]
	\arrow[from=1-2, to=2-2]
	\arrow["{\iota_{\Gr}}", from=2-1, to=2-2]
\end{tikzcd}
\end{equation}
Here, the natural isomorphism in the top right corner is discussed in Remark \ref{rmk: an isom}.

Now consider the fiber product of diagrams \eqref{diagram1} and \eqref{diagram2} over \eqref{diagram3}, we obtain the following Cartesian diagram:

% https://q.uiver.app/#q=WzAsNCxbMCwwLCJcXGhHT1xcYmFja3NsYXNoXFxjUiJdLFsxLDAsIlxcaEdPXFxiYWNrc2xhc2hcXGNSIl0sWzAsMSwiXFxoR09cXGJhY2tzbGFzaFxcR3IiXSxbMSwxLCJcXGhHT1xcYmFja3NsYXNoXFxHciJdLFswLDEsIlxcaW90YSJdLFswLDIsInAiLDJdLFsyLDMsIlxcaW90YV97XFxHcn0iXSxbMSwzLCJwIl1d
\begin{equation}\label{eq: diagram of iota}
    \begin{tikzcd}
	{[\hGO\backslash\cR]} & {[\hGO\backslash\cR]} \\
	{[\hGO\backslash\Gr]} & {[\hGO\backslash\Gr].}
	\arrow["\iota", from=1-1, to=1-2]
	\arrow["p"', from=1-1, to=2-1]
	\arrow["p", from=1-2, to=2-2]
	\arrow["{\iota_{\Gr}}", from=2-1, to=2-2]
\end{tikzcd}
\end{equation}

Here $\iota$ is also an involution on $[\hGO\backslash \cR]$. After identifying $[\hGO\backslash \cR]$ with $$[\hGO\backslash N_{\O}]\underset{[\hGK\backslash N_{\K}]}{\times}[\hGO\backslash N_{\O}]$$ as in Proposition \ref{eq:Rsymmdiagram}, $\iota$ becomes the morphism exchanging two factors\label{identification1}.

Define 
\begin{equation}\label{eq:(-)^*}
	(-)^*:=\iota^*(-):\Coh^{\hGO}(\cR)\to \Coh^{\hGO}(\cR).
\end{equation}
The same involution is defined on $\IndCoh^{\hGO}(\cR)$ by replacing $\Coh^{\hGO}(\cR)$ with $\IndCoh^{\hGO}(\cR)$ in the construction.

\textbf{Secondly}, we construct the Grothendieck-Serre dual functor
\[\D(-):\Coh^{\hGO}(\cR)^{\op}\to \Coh^{\hGO}(\cR).\]
Let $\pi:\Gr\to \pt$ be the projection, and define $\omega_{\Gr}:=\pi^!(\O_{\pt})\in \IndCoh^{\hGO}(\Gr)$. Recall $i_1:\cR\to \Gr\times N_{\O}$ denotes the defining embedding in \eqref{eq: R}, and we write $p_1:\Gr\times N_{\O}\to \Gr$ for the projection onto the first factor. Define 
\[\omega:=i_1^!p_1^*(\omega_{\Gr})\in \IndCoh^{\hGO}(\cR).\]
and
\begin{equation}\label{def of D}
	\D(-):=\cHom(-,\omega):\Coh^{\hGO}(\cR)^{\mathrm{op}}\to \IndCoh^{\hGO}(\cR).
\end{equation}
We refer to Section \ref{subsubsec:!-pull} and \ref{subsubsec:cHom} for the discussions on $!$-pullback and $\cHom$ functors. 

The following proposition is shown in \cite[Sec. 6]{CW23}.

\begin{proposition}\label{prop. dual in Coh(R)} \ 
\begin{enumerate}
    \item The monoidal $\infty$-category $\Coh^{\hGO}(\cR)$ is rigid (\cite[Prop. 6.5]{CW23});
    \item For any $\cF\in \Coh^{\hGO}(\cR)$, the right dual $\cF^{R}$ is given by $\D(\cF)^*$, and the left dual $\cF^{L}$ is given by $\D(\cF^*)$ (\cite[Prop. 6.8]{CW23}). 
    
    In particular, $\D(\cF)\cong (\cF^{R})^*\in \Coh^{\hGO}(\cR)$, and $\D\circ\D(\cF)\cong ((\cF^{R})^{**})^L\cong\cF$.
\end{enumerate}
\end{proposition}
Take $N=0$, the functor $\D(-):\Coh^{\hGO}(\cR)^{\mathrm{op}}\to \Coh^{\hGO}(\cR)$ becomes $\D_{\Gr}(-):=\cHom(-,\omega_{\Gr}):\Coh^{\hGO}(\Gr)^{\mathrm{op}}\to \Coh^{\hGO}(\Gr)$. These functors are called Grothendieck-Serre dual functors.

\subsection{Koszul perverse t-structure}\label{subsec:Koszul_t-structure}

In this subsection, we review the Koszul perverse t-sructure on $\Coh^{\hGO}(\cR_{G,N})$ in \cite{CW23}. After a discussion on some degree modification, we will firstly introduce the perverse t-structure on $\Coh^{\hGO}(\Gr)$, and secondly introduce the Koszul perverse t-structures on $\Coh^{\hGO}(\Gr)$ and $\Coh^{\hGO}(\cR)$.

\subsubsection{A degree modification of $\Coh^{\hGO}(\Gr_G)$}\label{subsubsec: degree modification}
Before entering the main discussions, we provide a remark on a degree issue regarding $\Coh^{\hGO}(\Gr)$ and related categories (see \cite[Sec. 2.1]{CW19}). Recall $\Gr=\Gr^0\bigsqcup \Gr^1$, where $\Gr^0$ and $\Gr^1$ are the union of even and odd cells respectively. We want to modify $\Coh^{\hGO}(\Gr)$ to obtain ${\Coh}_{{ [ 1\! / 2]}}^{\hGO}(\Gr)$ (resp. ${\Coh}_{\lr{\onetwo}}^{\hGO}(\Gr)$). These new categories are essentially the same as $\Coh^{\hGO}(\Gr)$, except that formally the cohomological-degrees (resp. $\Gmdil$-degrees) of objects in $\Coh^{\hGO}(\Gr^1)$ lie in $\tfrac12 +\Z$. These modifications will be used in the discussions on (Koszul) perverse t-structures.

Let us make these modifications rigorous following \cite[Sec. 1.2.1]{Fe21}. 

\begin{definition}\label{sqrt of shift}
	Let $\cC$ be a ($\infty$-) category with an auto-equivalence $T$. Define the new category $$\mathcal{C}[\sqrt{T}]:=\mathcal{C}^+\times\mathcal{C}^-$$ to be the product of two copies of $\mathcal{C}$. This decomposition also gives $\mathcal{C}[\sqrt{T}]$ a $\mathbb{Z}/2\mathbb{Z}$-grading. If $\mathcal{F}$ is an object of $\mathcal{C}$, we write $\mathcal{F}^{\pm}$ for $\mathcal{F}$ viewed as an object of $\mathcal{C}^{\pm}\subset\mathcal{C}[\sqrt{T}]$.
	
	Define an autoequivalence $\sqrt{T}$ of $\mathcal{C}[\sqrt{T}]$ by
	$$
	\sqrt{T}(\mathcal{F}^+) = \mathcal{F}^-, \quad \sqrt{T}(\mathcal{F}^-) = T(\mathcal{F})^+.
	$$
	Note that $T \cong  (\sqrt{T})^2$. 
\end{definition}

Recall $[1]:\Coh^{\hGO}(\Gr)\to \Coh^{\hGO}(\Gr)$ stands for the cohomology grading shift, and $\lr{1}:\Coh^{\hGO}(\Gr)\to \Coh^{\hGO}(\Gr)$ stands for the $\Gmdil$-grading shift (Definition \ref{shift functors}). The decomposition $$\Coh^{\hGO}(\Gr)=\Coh^{\hGO}(\Gr^0) \times  \Coh^{\hGO}(\Gr^1)$$ 
can be viewed as a $\Z/2\Z$-grading.

\begin{definition}
    Consider $\Coh^{\hGO}(\Gr)[\sqrt{[1]}]$. There is a $\Z/2\Z\times \Z/2\Z$-grading on $\Coh^{\hGO}(\Gr)[\sqrt{[1]}]$: one $\Z/2\Z$-grading comes from the even-odd decomposition of $\Gr$, and the other one comes from the square root construction. We restrict this $\Z/2\Z\times\Z/2\Z$-grading along the diagonal embedding $\Z/2\Z\hookrightarrow \Z/2\Z\times\Z/2\Z$, and define 
    \[{\Coh}_{[\onetwo]}^{\hGO}(\Gr):=\left(\Coh^{\hGO}(\Gr)\left[\sqrt{[1]}\right]\right)_{+}\]
    to be the degree $0$ part of this $\Z/2\Z$-grading. 

    By replacing $[1]$ with $\lr{1}$, we can carry out the same constructions. Then we define
    \[{\Coh}_{\lr{\onetwo}}^{\hGO}(\Gr):=\left(\Coh^{\hGO}(\Gr)\left[\sqrt{\lr{1}}\right]\right)_{+}.\]

    Above constructions can apply to other sheaf categories (e.g., $\QCoh(\Gr), \IndCoh(\Gr)$), and other spaces with such natural even-odd decompositions as well (e.g., $\cR,\cT,\cR\ttimes \cR$, etc.). The pull-push functors defined by morphisms in \eqref{eq: R} and \eqref{eq:convS} also have natural modifications adapted to the modifications for the categories.

\end{definition}
To simplify notation, we will denote $\sqrt{[1]}$ and $\sqrt{\lr{1}}$ by $[\tfrac12]$ and $\lr{\tfrac12}$ respectively.

\begin{remark}
    We will discuss the {perverse t-structure} on ${\Coh}_{[\onetwo]}^{\hGO}(\Gr)$, and the {Koszul perverse t-structure} on ${\Coh}_{\lr{\onetwo}}^{\hGO}(\Gr)$, ${\Coh}_{\lr{\onetwo}}^{\hGO}(\cR)$ later. The Koszul perverse t-structure is the main player of this paper. So for simplicity of notation, we just use $\Coh$ (resp. $\QCoh$, $\IndCoh$) to denote $\Coh_{\lr{\onetwo}}$ (resp. $\QCoh_{\lr{\onetwo}}$, $\IndCoh_{\lr{\onetwo}}$) in those situations.
\end{remark}

\subsubsection{Perverse t-structure on $\Coh^{\hGO}_{[\onetwo]}(\Gr_G )$}
We first discuss the \textit{perverse (coherent) t-structure} on $\Coh_{[\onetwo]}^{\hGO}(\Gr)$, which was introduced and studied by \cite{AB10,BFM05,CW19}. 

Let $S \subset \Gr$ be a $\hGO$-invariant (classical) locally closed subscheme of $\Gr$. Let $a_{{S}}: {S}\to \pt$ be the projection. Let $\omega_{{S}}:=a_{{S}}^!(\O_{\pt})\in \Coh^{\hGO}({S})$
\footnote{\label{footnote}When $a_S$ is proper, we have the definition $a_S^!$ as in Section \ref{subsubsec:!-pull}. In general, let $\overline{S}$ be the closure of $S$ in $\Gr$. We can factor $a_S:S\to \pt$ as $S\xrightarrow{j_{S\hookrightarrow \overline{S}}}\overline{S}\xrightarrow{a_{\overline{S}}}\pt$, where $j_{S\hookrightarrow \overline{S}} $ is an open embedding and $a_{\overline{S}}$ is a proper morphism. Then $a_S^!$ is defined by composition $j_{S\hookrightarrow \overline{S}}^*\circ a_{\overline{S}}^!$.

We won't discuss the functorality of $!$-pullback functors for non-proper morphisms. We refer the reader to \cite{GR17} for more discussions on this aspect.

} be our fixed dualizing complex. Define the Grothendieck-Serre dual functor as follows:
\begin{equation}\label{eq:tly}
	\D_S(-):=\cHom(-,\omega_S):\Coh_{[\onetwo]}^{\hGO}(S)^{\mathrm{op}}\to \Coh_{[\onetwo]}^{\hGO}(S).
\end{equation}

\begin{definition}[\cite{AB10}]\label{def:perv t-structure}
	Define the full $\infty$-subcategories $$\Coh_{[\onetwo]}^{\hGO}(S)^{\leq 0}_p, \Coh_{[\onetwo]}^{\hGO}(S)^{\geq 0}_p$$ of $\Coh_{[\onetwo]}^{\hGO}(S)$ as follows: 
	\begin{itemize}
		\item $\cF \in \Coh_{[\onetwo]}^{\hGO}(S)^{\leq 0}_p$ if and only if $i_{\lav}^*(\cF) \in {\QCoh}_{[\onetwo]}^{\hGO}(\Gr_{\lav})^{\leq - \tfrac12 {d_{\lav}}}$ for all $\Gr_{\lav} \subset S$ (recall that $d_{\lav}=\dim \Gr_{\lav}$), 
		\item $\Coh_{[\onetwo]}^{\hGO}(S)^{\geq 0}_p:=\D\left(\Coh_{[\onetwo]}^{\hGO}(S)^{\leq 0,\mathrm{op}}_p\right)$.
	\end{itemize}
\end{definition}

\begin{theorem}\cite{AB10}
	The pair $\left(\Coh_{[\onetwo]}^{\hGO}(S)^{\leq 0}_p, \Coh_{[\onetwo]}^{\hGO}(S)^{\geq 0}_p\right)$ defines a bounded and finite length t-structure on $\Coh_{[\onetwo]}^{\hGO}(S)$, which is called the (middle) perverse t-structure.
\end{theorem}

This perverse t-structure is a coherent analogue of the perverse t-structure for constructible sheaves in \cite{BBD}, and it shares many similar properties. For example, given a $\hGO$-invariant closed subscheme $Z\xhookrightarrow{i} S$, the push-forward $i_*:\Coh_{[\onetwo]}^{\hGO}(Z)\to \Coh_{[\onetwo]}^{\hGO}(S)$ is perverse t-exact; given a $\hGO$-invariant open subscheme $U\xhookrightarrow{j} S$, the pullback $j^*:\Coh_{[\onetwo]}^{\hGO}(S)\to \Coh_{[\onetwo]}^{\hGO}(U)$ is perverse t-exact.

Recall that $\Coh_{[\onetwo]}^{\hGO}(\Gr)\cong \colim_{\alpha}\Coh_{[\onetwo]}^{\hGO}(\Gr_{\alpha})$ for any reasonable presentation $\Gr\cong \colim_{\alpha} \Gr_{\alpha}$ as in Section \ref{subsub:coh_ind}. Then the perverse t-structure on $\Coh_{[\onetwo]}^{\hGO}(\Gr_{\alpha})$ glues to the perverse t-structure on $\Coh_{[\onetwo]}^{\hGO}(\Gr)$ using \cite[Lem. 4.13]{CW23}. In other words, the perverse t-structure on $\Coh_{[\onetwo]}^\hGO(\Gr)$ $$\left(\Coh_{[\onetwo]}^\hGO(\Gr)_p^{\le 0},\Coh_{[\onetwo]}^\hGO(\Gr)_p^{\ge 0}\right)$$ 
can be defined as the unique t-structure such that for any $\hGO$-invariant closed subscheme $S$, the pushforward $i_{S*}: \Coh_{[\onetwo]}^\hGO(S) \to \Coh_{[\onetwo]}^\hGO(\Gr)$ is t-exact with respect to the perverse t-structure on $\Coh_{[\onetwo]}^\hGO(S)$.

We write $\Pcoh^{\hGO}(S)$ and $\Pcoh^{\hGO}(\Gr)$ for the hearts of the perverse t-structures. Objects in the hearts are called perverse (coherent) sheaves. The abelian categories of perverse sheaves have some nice properties, which are summarized in the following theorem.

\begin{theorem}\cite[Sec. 4]{AB10}\label{theorem:simple objects on Gr}\ 
	\begin{enumerate}
		\item The heart $\Pcoh^{\hGO}(\Gr)$ is a finite length abelian category, which is stable under the Grothendieck-Serre dual functor $\D_{\Gr}$ on $\Coh_{[\onetwo]}^{\hGO}(\Gr)$; 
		
		\item  For any Schubert cell $\Gr_{\lav}$ and the natural embeddings $\Gr_{{\lav}}\xhookrightarrow{j_{\lav}} \overline{\Gr}_{\lav} \xhookrightarrow{i_{\le \lav}} \Gr$, there is a fully faithful functor called the intermediate extension functor
		\begin{equation}\label{intermediate extension for P}
			j_{\lav,p!*}:\Pcoh^{\hGO}(\Gr_{{\lav}})\to \Pcoh^{\hGO}(\overline{\Gr}_{\lav}).
		\end{equation}
        It satisfies $j_{\lav}^*\circ j_{\lav,p!*}\cong \id$.	A perverse sheaf in the $\cF\in \Pcoh^{\hGO}(\overline{\Gr}_{\lav})$ belongs to the essential image of $j_{\lav,p!*}$ if and only if
		\begin{equation}\label{Hom=0 criterion}
			\begin{aligned}
				\Hom(\cF,\cG)=0 \text{\ for all\ } \cG\in \Pcoh^{\hGO}(\overline{\Gr}_{\lav}) \text{\ such that\ } \Supp \cG \cap \Gr_{\lambda}=\emptyset\\
				\Hom(\cG,\cF)=0 \text{\ for all\ } \cG\in \Pcoh^{\hGO}(\overline{\Gr}_{\lav}) \text{\ such that\ } \Supp \cG \cap \Gr_{\lambda}=\emptyset
			\end{aligned}
		\end{equation}
		These two conditions are respectively equivalent to
		\begin{equation}\label{eq:dim supp}
				\begin{aligned}
					\codim \Supp \cH^{k-\tfrac12 d_{\lav}}(\cF)\ge 2k+2 &\text{\ for all\ } k\ge 1,\\
					\codim \Supp \cH^{k-\tfrac12 d_{\lav}}(\D_{\overline{\Gr}_{\lav}}(\cF))\ge 2k+2 &\text{\ for all\ } k\ge 1.
				\end{aligned}
			\end{equation}
		Let $\Pcoh^{\hGO}(\Gr_{\lav})_{!*}$ be the essential image of $j_{\lav,p!*}$. So $j_{\lav,p!*}$ yields an equivalence $$\Pcoh^{\hGO}(\Gr_{\lav})\cong \Pcoh^{\hGO}(\overline{\Gr}_{\lav})_{!*}.$$
		\item Simple objects of $\Pcoh^{\hGO}(\Gr)$ are of the form $i_{\le\lav*}j_{\lav,p!*}(\cL)$ for some $\lav\in \Pv_+$ and simple object $\cL\in \Pcoh^{\hGO}({\Gr}_{{\lav}})$. 
	\end{enumerate}
\end{theorem}

Let's make the description of simple objects in $\Pcoh^{\hGO}(\Gr)$ more explicit. Write $\hPla$\label{stabilizer hPla} for the stabilizer of $t^{{\lav}}$ with respect to the $\hGO$-action on $\Gr_{{\lav}}$. The quotient stack $[\hGO\backslash \Gr_{\lav}]$ can be identified with $[\hPla\backslash\pt]$. This gives an equivalence:
\[\Coh_{[\onetwo]}^{\hGO}(\Gr_{{\lav}}){\cong} \Coh^{\hGO}(\Gr_{{\lav}})\cong  \Coh^{\hPla}(\pt)\cong \mathsf{D}^b(\Rep{\hPla}).\]
The perverse t-structure on $\Coh_{[\onetwo]}^{\hGO}(\Gr_{{\lav}})$ translates to the standard t-structure on $\mathsf{D}^b(\Rep \hPla)$. So $\Pcoh^{\hGO}(\Gr_{\lav})\cong \Rep{\hPla}$ as abelian categories. Let $L_{\lav}\subset G$ be the subgroup of $G$ generated by $T$ and the root subgroups $U_{\alpha}$ of $G$ for those roots $\alpha$
satisfying $\lr{\alpha,\lav}=0$. Then the Levi quotient of $\hPla$ is isomorphic to $L_{\lav}\times \Gmrot\times \Gmdil$. The irreducible representations of $L_{\lav}$ are labeled by the set of weights $\mu\in P$ which are dominant with respect to $L_{\lav}$. Such a pair $(\lav,\mu)\in \Pv\times P$ is called a dominant pair. The subset of dominant pairs is denoted by $(\Pv\times P)_{\dom}$; this set is bijective to $(\Pv\times P)/W$. Let $\Olamu$ be the $\hGO$-equivariant sheaf over $\Gr_{\lav}$ corresponding to the $(\mu,0,0)$-representation of $L_{\mu}\times \Gmrot\times \Gmdil$. Then we get
\begin{equation*}
	\begin{aligned}
		\Irr(\Pcoh^{\hGO}(\Gr_{\lav}))=\{\Olamu[\tfrac{1}{2}d_{\lav}]\{m\}\lr{n}:m,n\in\mathbb{Z}, \mu\in P \text{\ such that \ }(\lav,\mu)\in (\Pv\times P)_{\dom}\}.
	\end{aligned}
\end{equation*}
Here $\Irr$ stands for the simple objects of an abelian category.

Now, for a dominant pair $(\lav,\mu)$, we define a simple object in $\Pcoh^{\hGO}(\Gr)$:
\begin{equation}\label{eq:cLp}
	\overline{\cL}^p_{\lav,\mu}:=i_{\le \lav *}j_{\lav,p!*}(\Olamu[\tfrac{1}{2}d_{\lav}]).
\end{equation}
Then by Theorem \ref{theorem:simple objects on Gr}, we have
\begin{equation}\label{simple in perverse heart}
		\Irr(\Pcoh^{\hGO}(\Gr))=\{\overline{\cL}^p_{\lav,\mu}\{m\}\lr{n}\colon m,n\in\mathbb{Z}, (\lav,\mu)\in ({\Pv}\times P)_{\dom}\}.
\end{equation}

At the end of this part, we introduce an explicit formula of the intermediate extension functor, following \cite{Sch15av} and \cite{Xin25}. The formula also allows this functor to be defined in broader cases, which we will use in Section \ref{subsec: twisted prod is IC}.\footnote{More explicitly, we want to use some intermediate extension functor on $\Gr\ttimes\Gr$. Notice that the $\hGO$-orbits on $\Gr\ttimes\Gr$ don't have parity property, which means, if two $\hGO$-orbits $\O_1,\O_2\subset \Gr\ttimes\Gr$ such that $\overline{\O_1}\supset \O_2$, $\dim \O_1$ and $\dim \O_2$ may not have the same parity. Thus we couldn't define some middle perverse coherent t-structure on $\Coh^{\hGO}(\Gr\ttimes\Gr)$ to use the intermediate extension functor constructed in \cite[Thm. 4.2]{AB10}.}

Let $X$ be a variety over $\C$, equipped with an action of an affine group scheme $H$. Let $\omega_X$ be our fixed dualzing complex $a^!\O_{\pt}$ for $a:X\to \pt$.\footnote{We assume $X$ to be $H$-embeddable as in the previous footnote \ref{footnote}. This condition is always satisfied in our situations.} We write
\begin{equation}\label{eq:cD}
	\cD_X(-):=\cHom(-,\omega_X[-\dim X]):\Coh^H(X)^{\op}\to \Coh^H(X),
\end{equation}
which is a cohomological shift version of the Grothendieck-Serre dual functor.

Let $\mathrm{Refl}^H(X)\subset \Coh^H(X)^{\heartsuit}$ be the additive full subcategory consisting of $H$-equivariant reflexive coherent sheaves in the sense of \cite[\href{https://stacks.math.columbia.edu/tag/0AVU}{Tag 0AVU}]{stacks-project}.
\begin{definition}\label{def:cIC}
	\begin{enumerate}
		\item Define
		\begin{equation}\label{eq:cIC for Refl}
			\begin{aligned}
			\wt{\cIC}_{X}:\mathrm{Refl}^H(X)&\to \Coh^{H}(X)\\
			\cF  &\mapsto (\tau_{\le \ell-1}\circ \cD_X \circ \tau_{\le \ell-2}\circ \cD_X\circ\dots \circ \cD_X\circ \tau_{\le 1} \circ \cD_X)(\cF).
		\end{aligned}
		\end{equation}
		Here $\ell$ is any odd integer such that $2\ell +1\ge \dim X$. According to the proof of \cite[Prop. 22.2]{Sch15av}, for each two $\ell_1,\ell_2$ satisfying the conditions, the functors produced by the above formula can be canonically identified.
		\item Suppose $X$ is normal. Let $U$ be an $H$-invariant open subset with $\codim (X\setminus U)\ge 2$. Write $\VB^H(U)\subset \Coh^{H}(U)^{\heartsuit}$ for the additive full subcategory consisting of $H$-equivariant locally free sheaves of finite rank. According to \cite[\href{https://stacks.math.columbia.edu/tag/0AY6}{Tag 0AY6}]{stacks-project}, $\cH^0 (j_*(-))$ defines a functor $\VB^H(U)\to\mathrm{Refl}^H(X) $. Define
		\begin{equation}\label{eq:cIC for VB}
			\begin{aligned}
				\cIC_{U\subset X}:\VB^H(U)&\to \Coh^{H}(X)\\
				\cE &\mapsto \wt{\cIC}_{X}\circ \cH^0 (j_*(\cE)).
			\end{aligned}
		\end{equation}
	\end{enumerate}
\end{definition}
We will often omit the subscript $X$ and $U\subset X$ respectively when it is clear. Since the $\wt{\cIC}$ and $\cIC$ functors are by definition compatible with forgetting $H$-equivariance, we use the same notations before and after forgetting $H$-equivariance.

The following proposition is proved in \cite[Prop. 4.7]{Xin25}, which uses the same method in {\cite[Prop. 21.1]{Sch15av}
\begin{proposition}[{\cite[Prop. 21.1]{Sch15av}, \cite[Prop. 4.7]{Xin25}}]\label{prop:cIC_and_support}
    Let $\cC\in \Coh^{H}(X)$. Suppose we have following support conditions
    \begin{equation}\label{eq: supp condition}
        \codim \Supp \cH^k\cC \ge 2k+2 \ \ \ \text{and}\ \ \  \codim \Supp \cH^k\cD_X(\cC) \ge 2k+2 \ \ \ \text{for all }k\ge 1,
    \end{equation}
    and $j^*\cH^0\cC$ is locally free. Then there is a natural isomorphism $\cC\cong \cIC(j^*\cH^0\cC)$.
\end{proposition}

\begin{remark}\ 
    \begin{enumerate}
        \item In \cite{Sch15av}\cite{Xin25}, they stated this result in the non-equivariant setting, but the proof actually extends to the $H$-equivariant setting.
        \item In \cite[Prop. 4.7]{Xin25}, the author assumed $\cC\in \Coh^H(X)^{\ge 0}$. This condition can be deduced from the second support condition in \eqref{eq: supp condition}, using \cite[Lem. 17.5]{Sch15av}.
    \end{enumerate}
\end{remark}

In the case when $H=\hGO,X=\overline{\Gr}_{\lav}, U=\Gr_{\lav}$, we deduce from \eqref{eq:dim supp} that: for any $\cE\in \VB^{\hGO}(\Gr_{\lav})$, 
\[j_{\lav,p!*}(\cE[\tfrac12 d_{\lav}])\cong \cIC(\cE)[\tfrac12 d_{\lav}],\]
and
\begin{equation}\label{cIC in our case}
	\codim \Supp \cH^k\cIC(\cE) \ge 2k+2 \ \ \ \text{and}\ \ \  \codim \Supp \cH^k\cD_X(\cIC(\cE)) \ge 2k+2 \ \ \ \text{for all }k\ge 1.
\end{equation}

\subsubsection{Koszul perverse t-structure on $\Coh^{\hGO}(\Gr_G)$}
To define the so-called Koszul perverse t-structure on $\Coh^{\hGO}(\cR)$, we need to introduce the Koszul perverse t-structure on $\Coh^{\hGO}(\Gr)$, which is a grading shift of the perverse t-structure on $\Coh_{[\onetwo]}^{\hGO}(\Gr)$.

Let $S \subset \Gr^{\epsilon}$ be an $\hGO$-invariant subscheme for $\epsilon\in \{0,1\}$. Let $(\Coh_{[\onetwo]}^{\hGO}(S)^{\leq 0}_{p}, \Coh_{[\onetwo]}^{\hGO}(S)^{\geq 0}_{p})$ be the perverse t-structure on $\Coh_{[\onetwo]}^{\hGO}(S)$.

First, since the $\Gmdil$ in $\hGO=(G_\O \rtimes \Gmrot) \times \Gmdil$ acts on $S$ trivially, we obtain the following weight decomposition given by this $\Gm$-equivariance:
\[\Coh_{[\onetwo]}^{\hGO}(S)\cong \bigoplus_{n\in\mathbb{Z}}\Coh^{\hGO}(S)_{n}.\]
The grading shift $\lr{1}$ has degree $-1$.

By Definition \ref{def:perv t-structure}, the perverse t-structure is compatible with this weight decomposition, which means we have the weight decompositions
\[\Coh_{[\onetwo]}^{\hGO}(S)^{\le 0}_{p}=\bigoplus_{n\in\mathbb{Z}}\Coh_{[\onetwo]}^{\hGO}(S)^{\le 0}_{p,n}\ ,\ \Coh_{[\onetwo]}^{\hGO}(S)^{\ge 0}_{p}=\bigoplus_{n\in\mathbb{Z}}\Coh_{[\onetwo]}^{\hGO}(S)^{\ge 0}_{p,n}.\]

Next, let's consider the following grading shift. The basic idea comes from Koszul duality.
\begin{definition}
	Define $\infty$-subcategories $\Coh^{\hGO}(S)^{\leq 0}_{\Kp}, \Coh^{\hGO}(S)^{\geq 0}_{\Kp}$ of $\Coh^{\hGO}(S)$ as follows:
	\begin{equation}
		\begin{aligned}
			&\Coh^{\hGO}(S)^{\leq 0}_{\Kp}:=\bigoplus_{n\in\mathbb{Z}}\Coh_{[\onetwo]}^{\hGO}(S)^{\leq 0}_{p,n}[-n][\tfrac{\epsilon}{2}]\lr{\tfrac{\epsilon}{2}},
			\\&\Coh^{\hGO}(S)^{\geq 0}_{\Kp}:=\bigoplus_{n\in\mathbb{Z}}\Coh_{[\onetwo]}^{\hGO}(S)^{\geq 0}_{p,n}[-n][\tfrac{\epsilon}{2}]\lr{\tfrac{\epsilon}{2}};
		\end{aligned}
	\end{equation}
	here, the reason for $[\tfrac{\epsilon}{2}]\lr{\tfrac{\epsilon}{2}}$ refers to Section \ref{subsubsec: degree modification}. This will define a t-structure on $\Coh^{\hGO}(S)$, called the \textbf{Koszul perverse t-structure}. The heart of the Koszul perverse t-structure is denoted by $\KPcoh^{\hGO}(S)$. It can be decomposed as
	\begin{equation}\label{eq:KP&P}
		\KPcoh^{\hGO}(S)=\bigoplus_{n\in\mathbb{Z}}\Pcoh^{\hGO}(S)_n[-n][\tfrac{\epsilon}{2}]\lr{\tfrac{\epsilon}{2}}.
	\end{equation}

	Similarly, we also obtain the Koszul perverse t-structure on $\Coh^{\hGO}(\Gr)$.
\end{definition}

There are two t-exact grading shifts with respect to the perverse t-structure: $\{1\}$ and $\lr{1}$. Similarly, we have two t-exact grading shifts with respect to the Koszul perverse t-structure: $\{1\}$ and $[1]\lr{1}$. We will denote $[1]\lr{1}$ by $(1)_K$, and call it the Koszul shift. (We add the subscript ${}_K$ to distinguish it from the Tate twist for Hodge modules.)

The intermediate extension functors \eqref{intermediate extension for P} can be transferred to Koszul perverse sheaves
\begin{equation}\label{intermediate extension for KP}
	j_{\lav,!*}: \KPcoh^{\hGO}(\Gr_{ \lav}) \to \KPcoh^{\hGO}(\overline{\Gr}_{\lav}).
\end{equation}
by means of the identity \eqref{eq:KP&P}; more precisely, for a $\hGO$-equivariant vector bundle $\cE$ on $\Gr_{\lav}$ of $\Gmdil$-weight $n$, we have $\cE[-n+\tfrac12 d_{\lav}]\in \KPcoh^{\hGO}(\Gr_{\lav}) $, and
\begin{equation}\label{eq:!*relations}
	j_{\lav,!*}(\cE[\tfrac12 d_{\lav}-n]):=j_{\lav,p!*}(\cE[\tfrac12 d_{\lav}])[-n]=\cIC(\cE)[\tfrac12 d_{\lav}-n].
\end{equation}
The essential image of $j_{\lav,!*}$ is denoted by
\begin{equation}\label{eq:KP_!*}
	\KPcoh^{\hGO}(\overline{\Gr}_{\lav})_{!*};
\end{equation}
the criterion for a Koszul perverse sheaf to lie in this full subcategory is similar to \eqref{Hom=0 criterion}.

For a dominant pair $(\lav,\mu)\in (\Pv\times P)_{\dom}$, we define 
\begin{equation}\label{eq:cLb}
\overline{\cL}_{\lav,\mu}:=\overline{\cL}^p_{\lav,\mu}(-\tfrac12 d_{\lav})_K=i_{\le \lav *}j_{\lav,!*}(\Olamu\lr{-\tfrac{1}{2}d_{\lav}})=i_{\le \lav *}\cIC(\Olamu)\lr{-\tfrac{1}{2}d_{\lav}}.
\end{equation}

From the description \eqref{simple in perverse heart} for simple objects in the perverse heart, we obtain the following description for the set of simple objects in the Koszul perverse heart:
\begin{equation}\label{simple in Koszul perverse heart for Grlambda}
	\Irr(\KPcoh^{\hGO}(\Gr))=\{\overline{\cL}_{\lav,\mu}\{m\}( n)_K\colon m,n\in\mathbb{Z}, (\lav,\mu)\in ({\Pv}\times P)_{\dom}\}.
\end{equation}

\subsubsection{Koszul perverse t-structure on $\Coh^{\hGO}(\cR_{G,N})$}\label{subsubsec:Koszul_perverse}
Now, the Koszul perverse t-structure on $\Coh^{\hGO}(\cR)$ can be constructed, which is one of the main achievement in \cite{CW23}. We summarize the main features of this t-structure below, and refer to \textit{ibid.} for more details.

We write $i_1:\cR\to \Gr\times N_{\O},i_2:\cR\to \cT$ for the defining embeddings in \eqref{eq: R}, $p_1:\Gr\times N_{\O}\to \Gr,p_2:\cT\to \Gr$ for the projections, and $\sigma_1:\Gr\to \Gr\times N_{\O},\sigma_2:\Gr\to \cT$ for the embeddings of the zero sections. The morphisms $\sigma_1,\sigma_2$ have stable coherent pullback (see Example \ref{exam:stable_coh_pull}).

Let $S\subset \Gr$ be an $\hGO$-invariant locally closed subscheme. We write $\cR_S,\cT_S$ for the base change of $\cR,\cT$ to $S$ respectively. We abuse notations to denote the base changes of $\sigma_1,\sigma_2,i_1,i_2,p_1,p_2$ along $S\hookrightarrow \Gr$ by the same symbols.
\begin{theorem}\label{thm:koszul-perverse}
\ 

    \begin{enumerate}
        \item (\cite[Prop. 3.16]{CW23}) There exists a unique t-structure on $\Coh^{\hGO}(S\times N_{\O})$ (resp. $\Coh^{\hGO}(\cT_S)$) such that the pullback 
        \[\sigma_1^*:\Coh^{\hGO}(S\times N_{\O})\to \Coh^{\hGO}(S)\]
        \[(\text{resp.\ }\sigma_2^*:\Coh^{\hGO}(\cT_S)\to \Coh^{\hGO}(S))\]
        along the zero section is t-exact with respect to the Koszul perverse t-structure on $\Coh^{\hGO}(S)$. This t-structure on $\Coh^{\hGO}(S\times N_{\O})$ (resp. $\Coh^{\hGO}(\cT_S)$) is also called the Koszul perverse t-structure.
        
        Moreover, this t-structure is bounded and finite length, and there is a bijection between simple objects in Koszul perverse hearts given by $p_1^*$ and $\sigma_1^*$ (resp. $p_2^*$ and $\sigma_2^*$): 
        \begin{equation}\label{eq: simple objects bijection}
            % https://q.uiver.app/#q=WzAsMixbMCwwLCJcXElycihcXEtQY29oXntcXGhHT30oUykpIl0sWzEsMCwiXFxJcnIoXFxLUGNvaF57XFxoR099KFNcXHRpbWVzIE5fe1xcT30pKSJdLFswLDEsInBfMV4qIiwwLHsib2Zmc2V0IjotMX1dLFsxLDAsIlxcc2lnbWFfMV4qIiwwLHsib2Zmc2V0IjotMX1dXQ==
            \begin{tikzcd}
            	{\Irr(\KPcoh^{\hGO}(S))} & {\Irr(\KPcoh^{\hGO}(S\times N_{\O}))}
            	\arrow["{p_1^*}", shift left, from=1-1, to=1-2]
            	\arrow["{\sigma_1^*}", shift left, from=1-2, to=1-1]
            \end{tikzcd}
        \end{equation}
    	\[(\text{resp. \ }\begin{tikzcd}
    		{\Irr(\KPcoh^{\hGO}(S))} & {\Irr(\KPcoh^{\hGO}(\cT_S))}
    		\arrow["{p_2^*}", shift left, from=1-1, to=1-2]
    		\arrow["{\sigma_2^*}", shift left, from=1-2, to=1-1]
    	\end{tikzcd}).\]

        \item (\cite[Thm. 3.1, Thm. 4.1]{CW23}) There exists a unique t-structure on $\Coh^{\hGO}(\cR_S)$ such that 
        $$i_{1*}:\Coh^\hGO(\cR_S) \to \Coh^\hGO(S\times N_{\O})$$
        is t-exact with respect to the Koszul perverse t-structure on $\Coh^\hGO(S\times N_{\O})$. Symmetrically, it is the unique t-structure such that 
        $$i_{2*}:\Coh^\hGO(\cR_S) \to \Coh^\hGO(\cT_S)$$
        is t-exact with respect to the Koszul perverse t-structure on $\Coh^\hGO(\cT_{S})$. This t-structure on $\Coh^{\hGO}(\cR_S)$ is also called the Koszul perverse t-structure.

        Moreover, this t-structure is bounded and finite length, and there is a bijection between simple objects in $\KPcoh^{\hGO}(S\times N_{\O})$ and $\KPcoh^{\hGO}(\cR_S)$.
        
        \item (\cite[Prop 3.11, 3.13, 3.24]{CW23}) For a $\hGO$-invariant subscheme $Z\xhookrightarrow{i}S$, the push-forward $\Coh^{\hGO}(\cR_Z)\to \Coh^{\hGO}(\cR_S)$ is Koszul perverse t-exact. For a $\hGO$-invariant open subscheme $U\xhookrightarrow{j}S$, the pullback $j^*:\Coh^{\hGO}(\cR_S)\to \Coh^{\hGO}(\cR_U)$ is Koszul perverse t-exact.
        
        \item (\cite[Thm. 4.1]{CW23})       Since
       	\begin{align*}
       		\Coh^{\hGO}(\Gr\times N_{\O})&\cong \colim_{\alpha} \Coh^{\hGO}(\Gr_{\alpha}\times N_{\O}),\\ (\text{resp.\ } \Coh^{\hGO}(\cT)&\cong \colim_{\alpha} \Coh^{\hGO}(\cT_{\Gr_{\alpha}}),\\ \Coh^{\hGO}(\cR)&\cong \colim_{\alpha} \Coh^{\hGO}(\cR_{\Gr_{\alpha}})).
       	\end{align*}
       	for any reasonable presentation $\Gr\cong \colim_{\alpha} \Gr_{\alpha}$ as in Section \ref{subsub:coh_ind}, the Koszul perverse t-structure on $\Coh^{\hGO}(\Gr_{\alpha}\times N_{\O})$, (resp. $\Coh^{\hGO}(\cT_{\Gr_{\alpha}})$, resp. $\Coh^{\hGO}(\cR_{\Gr_{\alpha}})$) glues to the Koszul perverse t-structure on $\Coh^{\hGO}(\Gr\times N_{\O})$ (resp. $\Coh^{\hGO}(\cT)$, $\Coh^{\hGO}(\cR)$) by \cite[Lem. 4.13]{CW23}.
    \end{enumerate}
\end{theorem}

\begin{remark}
	By some categorical arguments as in \cite[Prop. 3.2]{CW23}, the Koszul perverse t-structure on $\Coh^{\hGO}(X)$ (here $X$ stands for $S,S\times N_{\O},\cR_S,\Gr,\Gr\times N_{\O},\cR$, etc.) can be extended to a t-structure on $\IndCoh^{\hGO}(X)$ (resp. $\QCoh^{\hGO}(X)$) by taking $\IndCoh^{\hGO}(X)_{\Kp}^{\le 0}$ (resp. $\QCoh^{\hGO}(X)_{\Kp}^{\le 0}$) to be the full $\infty$-subcategory generated by $\Coh^{\hGO}(X)^{\le 0}_{\Kp}$ (resp. $\Coh^{\hGO}(X)^{\le 0}_{\Kp}$) under extensions and colimits, and $\IndCoh^{\hGO}(X)_{\Kp}^{\ge 0}$ (resp. $\QCoh^{\hGO}(X)_{\Kp}^{\ge 0}$) to be the right orthogonal of $\IndCoh^{\hGO}(X)_{\Kp}^{\le 0}$ (resp. $\QCoh^{\hGO}(X)_{\Kp}^{\le 0}$). We will write the heart of these Koszul perverse t-structures by $\KPindcoh^{\hGO}(X)$ (resp. $\KPqcoh^{\hGO}(X)$). These extensions are not involved in the main content of this article, but they will be used in some arguments; for example, the description of simple objects below \eqref{def of KPlamu}.
\end{remark}

Let's make the description of simple objects in $\KPcoh^{\hGO}(\cR)$ more explicit. Let $\cl:\cR^{\cl}_{\lav}\to \cR_{\lav}$ be the classical locus embedding, and $p^{\cl}:\cR^{\cl}_{\lav}\to \Gr_{\lav}$ be the projection (which is a vector bundle of infinite dimension). Then
\begin{equation}\label{simple in Koszul perverse heart for cRlambda}
\begin{aligned}
	&\Irr{\KPcoh^{\hGO}(\cR_{\lav})}=\{\cl_*{p^{\cl*}}(\Olamu\lr{-\tfrac{1}{2}d_{\lav}})\{m\}(n):m,n\in\mathbb{Z}, \mu \text{ dominant for }L_{\lav}\}.
\end{aligned}
\end{equation}

Let $j_{\lav}:\cR_{\lav}\to \cR_{\le \lav},i_{\le \lav}:\cR_{\le \lav}\to \cR$ be the embeddings. Define 
\begin{equation}\label{def of KPlamu'}
	\cL_{\lav,\mu}':=\text{the socle of }\HK^0(j_{\lav*}\cl_* {p^{\cl*}}(\Olamu\lr{-\tfrac{1}{2}d_{\lav}})).
\end{equation}
Here, we view $j_{\lav*}\cl_*{p^{\cl*}}(\Olamu\lr{-\tfrac{1}{2}d_{\lav}})$ as an object in $\IndCoh^{\hGO}(\cR_{\le \lav})$, and $\HK^0$ denotes the cohomology functor for the Koszul perverse t-structure on $\IndCoh^{\hGO}(\cR)$. According to \cite[Prop. 3.28]{CW23}, $\cL'_{\lav,\mu}$ is a simple object in $\KPcoh^{\hGO}(\cR_{\le \lav})$ satisfying 
\begin{equation}\label{eq: j_{lav}^*simple}
    j_{\lav}^*(\cL_{\lav,\mu}')\cong \cl_*{p^{\cl*}}(\Olamu\lr{-\tfrac{1}{2}d_{\lav}}).
\end{equation}
By \cite[Prop. 3.29]{CW23}, 
\begin{equation}\label{def of KPlamu}
	\cL_{\lav,\mu}:=i_{\le \lav*}(\cL'_{\lav,\mu})
\end{equation}
is a simple object in $\KPcoh^{\hGO}(\cR)$.

According to \cite[Thm. 3.31]{CW23}, we have the following description for the set of simple objects in $\KPcoh^{\hGO}(\cR)$: 

\begin{equation}\label{simple in Koszul perverse heart}
	\Irr(\KPcoh^{\hGO}(\cR))=\{\cL_{\lav,\mu}\{m\}( n)\colon m,n\in\mathbb{Z}, (\lav,\mu)\in ({\Pv}\times P)_{\dom}\}.
\end{equation}

Later, we will pay more attention to the simple objects for $\mu=0$, so we write
\begin{equation}\label{simple object cL_lav}
	\cL_{\lav}':=\cL_{\lav,0}',\ \cL_{\lav}:=\cL_{\lav,0}.
\end{equation}

\ 

At the end of this subsection, we discuss the compatibility of the Koszul perverse t-structure with the monoidal structure on $\Coh^{\hGO}(\cR)$, which is encapsulated in the following theorem.
\begin{theorem}[{\cite[Thm. 5.1, Thm. 6.1]{CW23}}]\ 
   \begin{enumerate}
   	\item The convolution product $\conv$ is Koszul perverse t-exact, i.e. if $\cF,\cG\in\KPcoh^{\hGO}(\cR)$, then $\cF\conv \cG\in \KPcoh^{\hGO}(\cR)$. Additionally, the monoidal unit in $\Coh^{\hGO}(\cR)$ is a Koszul perverse sheaf (which is in fact $\cL_0$).
   	
   	\item Taking (right/left) dual is Koszul perverse t-exact, i.e. if $\cF\in \KPcoh^{\hGO}(\cR)$, then $\cF^R$ and $\cF^L$ are both in $\KPcoh^{\hGO}(\cR)$. Also, taking $(-)^*$ and $\D$ are both Koszul perverse t-exact.
   \end{enumerate}

   In conclusion, $\KPcoh^{\hGO}(\cR)$ is a rigid monoidal abelian category.
\end{theorem}

\subsection{Renormalized \texorpdfstring{$r$}{r}-matrices}\label{subsec: Renormalized r-matrices}
In this subsection, we first review the notion of renormalized $r$-matrices. Then we discuss the factorization structure on $\cR_{G,N}$ as presented in \cite[Sec. 7]{CW23}. Finally, following \cite[Sec. 5]{CW19}, we explain how this factorization structure induces the renormalized $r$-matrices in $\mathrm{h}(\Coh^{\hGO}(\cR_{G,N}))$ (the homotopy category of $\Coh^{\hGO}(\cR_{G,N})$). The renormalized $r$-matrices also restrict to the Koszul perverse heart $\KPcoh^{\hGO}(\cR_{G,N})\subset \mathrm{h}(\Coh^{\hGO}(\cR_{G,N}))$.

\subsubsection{Renormalized $r$-matrices}
Let $(\cC,\conv)$ be a $\C$-linear monoidal category. Assume $\cC$ has a grading shift functor, denoted by $\{1\}$.

\begin{definition}[{\cite[Def. 4.1, Def. 4.16]{CW19}}]
	A system of renormalized $r$-matrices in $\cC$ is an assignment to each pair of objects $M,N \in \cC$ an element $\Lambda(M,N) \in \Z \cup \{-\infty\}$ and a map \[\rmat{M,N}: M \conv N \to N \conv M\{\Lambda(M,N)\},\] satisfying the following properties. (Here $\Lambda(M,N)=-\infty$ means $\rmat{M,N}=0$.)
	\begin{enumerate}
		\item\label{p1} For any $M \in \cC$ the morphisms $\rmat{M,1_\cC}$ and $\rmat{1_\cC,M}$ are given by composing the unit isomorphisms of $M \conv 1_\cC$ and $1_\cC \conv M$ with $M$. 
		\item\label{p2} $\rmat{M,N}$ is nonzero if and only if $\Lambda(M,N)\neq-\infty$.
		\item\label{p3} For any $M,N_1,N_2 \in \cC$ we have 
		$$\Lambda(M, N_1 \conv N_2) \leq \Lambda(M, N_1) + \Lambda(M, N_2).$$
		If equality holds, then 
		$$\rmat{M,N_1 \conv N_2} = (\id_{N_1} \conv \rmat{M,N_2}\{\Lambda(M,N_1)\}) \circ (\rmat{M,N_1} \conv \id_{N_2}),$$
		while the right-hand composition is zero if the inequality is strict. The corresponding statement with $M$ on the other side also holds. 
		
		\item\label{p4} For $M,N \in \cC$ we have $\Lambda(M,N) + \Lambda(N,M) \ge 0$. Moreover
		\begin{equation}\label{eq:two equiv condition}
			\rmat{N,M} \circ \rmat{M,N} \ne 0 \iff \Lambda(M,N) + \Lambda(N,M) = 0.
		\end{equation}
		\item\label{p5} For any $M,N_1,N_2 \in \cC$ and morphism $f: N_1 \to N_2$, consider the diagram
		
		$$
		\begin{tikzcd}[row sep=large, column sep=3.1cm] 
			M \conv N_1 \arrow[r, "{\rmat{M, N_1}\{-\Lambda(M,N_1)\}}"] \arrow[d, "\id_M \conv f"'] & N_1 \conv M  \arrow[d, "f \conv \id_M"] \\
			M \conv N_2 \arrow[r, "{\rmat{M, N_2}\{-\Lambda(M,N_2)\}}"] & N_2 \conv M.
		\end{tikzcd}
		$$
		\begin{itemize}
			\item If $\Lambda(M,N_1) = \Lambda(M,N_2)$, the diagram commutes.
			\item If $\Lambda(M,N_1) < \Lambda(M,N_2)$, the bottom left composition is zero. 
			\item If $\Lambda(M,N_1) > \Lambda(M,N_2)$, the top right composition is zero. 
		\end{itemize}
		The corresponding statements hold when the product with $M$ is taken on the other side. 
	\end{enumerate}
\end{definition}

\subsubsection{The factorization structure on $\cR_{G,N}$}\label{subsubsec:gobal}
Let $X$ be a smooth algebraic curve, $I$ be a finite set, and $S=\Spec R$ be an (classical) affine scheme. For a tuple of maps $f_I=(f_i)_{i\in I}:S\to X^I$, write $z_I\subset S\times X$ for the union of the graphs of all the $f_i,i\in I$. Let $\O_{\wh{z_I}}$ denote the inverse limit of coordinate rings of closed subschemes in $S\times X$ with reduced locus $z_I$. Define $\wh{z_I}:=\Spec \O_{\wh{z_I}}$, which can be viewed as the scheme-theoretic formal neighborhood of $z_I$. We also have the punctured formal neighborhood $\wh{z_I}\setminus z_I$.

For a (classical) stack $Y\in \Stk_{\C}^{\cl}$, define the $I$-positive loop space $Y_{\O,X^I}$
	\[ Y_{\O,X^I}:\cCAlg\to \Grpd, R\mapsto \{f_I:\Spec R\to X^I, \rho: \wh{z_I}\to Y\}\]
and the $I$-loop space $Y_{\K,X^I}$
	\[ Y_{\K,X^I}:\cCAlg\to \Grpd, R\mapsto \{f_I:\Spec R\to X^I, \rho: \wh{z_I}\setminus z_I\to Y\}.\]

We will concentrate on the case when $Y$ is a smooth affine scheme over $\C$. In this case, $Y_{\O,X^I}$ is represented by a scheme of infinite type over $X^I$; moreover, this scheme is an inverse limit of flat morphisms between smooth schemes which are flat over $X^I$. And $Y_{\K,X^I}$ is represented by an ind-affine scheme over $X^I$, which contains $Y_{\O,X^I}$ as a closed subscheme. One may refer to \cite[Sec. 2]{KV04} for more details.

The key feature of these constructions is that the families $Y_{\O,X^I}$, $Y_{\K,X^I}$ (with $I$ running over nonempty finite sets) each form a factorization monoid in the sense of \cite[Def. 2.2.1]{KV04}: for a surjection of finite sets $f:J\twoheadrightarrow I$, let's write $\Delta_{J/I}:X^I\hookrightarrow X^J$ for the induced diagonal embedding and $j_{J/I}:U_{J/I}:=X^J\setminus X^I\hookrightarrow X^J$ for the complement; then the factorization monoid structure consists of isomorphisms
\begin{equation}\label{eq: factorization isom}
    \Delta_{J/I}^*(Y_{\O,X^J})\cong Y_{\O,X^I},\ \ \ j_{J/I}^*(\prod_{i\in I}Y_{\O,X^{f^{-1}(i)}})\cong j_{J/I}^*(Y_{\O,X^J}) 
\end{equation}
satisfying compatibilities in \textit{ibid.}. Similar statements also hold for $\left\{Y_{\K,X^I}\right\}_{I}$.

The spaces constructed in Section \ref{subsec: def of spaces} all have their factorization counterparts. The factorization counterpart of $\Gr_G$ is the Beilinson-Drinfeld grassmannian $$\Gr_{G,X^I}:=G_{\K,X^I}/G_{\O,X^I}$$ (or $\Gr_{X^I}$ for short), which is represented by an ind-projective scheme over $X^I$. The factorization counterpart of $\cT_{G,N}$ is $$\cT_{G,N,X^I}:=G_{\K,X^I}\times ^{G_{\O,X^I}} N_{\O,X^I}$$ (or $\cT_{X^I}$ for short). The natural action of $G_{\K,X^I}$ on $N_{\K,X^I}$ produces a map 
$$\cT_{X^I}\to \Gr_{X^I}\times_{X^I} N_{\K,X^I}$$
as in \eqref{eq:j_2}. 

\begin{definition}
	The factorization space of triples $\cR_{G,N,X^I}$ (or $\cR_{X^I}$ for short) is defined to be the fiber product:
	% https://q.uiver.app/#q=WzAsNCxbMCwwLCJcXGNSX3tYXkl9Il0sWzEsMCwiXFxHcl97WF5JfVxcdGltZXNfe1heSX1OX3tcXE8sWF5JfSJdLFswLDEsIlxcY1Rfe1heSX0iXSxbMSwxLCJcXEdyX3tYXkl9XFx0aW1lc197WF5JfU5fe1xcSyxYXkl9Il0sWzAsMSwiaV8xIl0sWzAsMiwiaV8yIiwyXSxbMiwzLCJqXzIiLDJdLFsxLDMsImpfMSJdXQ==
	\begin{equation}\label{eq:global R}
		\begin{tikzcd}
			{\cR_{X^I}} & {\Gr_{X^I}\times_{X^I}N_{\O,X^I}} \\
			{\cT_{X^I}} & {\Gr_{X^I}\times_{X^I}N_{\K,X^I}}
			\arrow["{i_1}", from=1-1, to=1-2]
			\arrow["{i_2}"', from=1-1, to=2-1]
			\arrow["{j_1}", from=1-2, to=2-2]
			\arrow["{j_2}", from=2-1, to=2-2]
		\end{tikzcd}
	\end{equation}
in the category of ind-derived schemes.
\end{definition}

From now on, we will fix $X$ to be $\A^1=\Spec \C[t]$. Then $X$ admits a natural rotation $\Gm$-action. As in the local case, we add two $\Gm$ actions $\Gmrot$ and $\Gmdil$ here. Define $\cGmrot{X^I}$ to be the action groupoid scheme
% https://q.uiver.app/#q=WzAsMixbMCwwLCJcXEdtcm90IFxcdGltZXMgWF5JIl0sWzEsMCwiWF5JIl0sWzAsMSwiXFxtYXRocm17cHJ9IiwyLHsib2Zmc2V0IjoxfV0sWzAsMSwiXFxtYXRocm17YWN0fSIsMCx7Im9mZnNldCI6LTF9XV0=
\[\begin{tikzcd}
	{\Gmrot \times X^I} & {X^I.}
	\arrow["{\mathrm{pr}}"', shift right, from=1-1, to=1-2]
	\arrow["{\mathrm{act}}", shift left, from=1-1, to=1-2]
\end{tikzcd}\]
Define the groupoid scheme
\[{\hGOfac{X^I}}:=(G_{\O,X^I}\rtimes_{X^I} \cGmrot{X^I})\times \Gmdil.\]
The spaces over $X^I$ appearing in \eqref{eq:global R} all have natural $\hGOX$-actions. 
\begin{remark}
	For a derived scheme $\mathcal{Y}$ over $X^I$ with a $\hGOX$-action, the fpqc quotient $[\hGOfac{X^I}\backslash \mathcal{Y}]_{X^I}$ can be identified with $$\Big[\Gmrot\big{\backslash}\big[(G_{\O,X^I}\times \Gmdil)\backslash\mathcal{Y}\big]_{X^I}\Big]_{\pt}.$$ This means: firstly, we quotient $\mathcal{Y}$ by $G_{\O,X^I}\times \Gmdil$ over $X^I$; secondly, we quotient this derived stack by $\Gmrot$ over $\pt$. Then we can view $[\hGOfac{X^I}\backslash \mathcal{Y}]_{X^I}$ as a derived stack over $[\Gmrot\backslash X^I]$.
\end{remark}

The algebro-geometric properties in Proposition \ref{Prop: prop of R} also holds for $\cR_{X^I}$ and $[\hGOX\backslash \cR_{X^I}]_{X^I}$ respectively. We can consider $\left\{\Coh^{\hGOX}(\cR_{X^I})\right\}_I$ and $\left\{\IndCoh^{\hGOX}(\cR_{X^I})\right\}_I$. All of them admit convolution product constructions similar to those for $\Coh^{\hGO}(\cR)$ (as described in Section \ref{Subsec: convolution}), via the global analogue of the convolution diagram:
% https://q.uiver.app/#q=WzAsNixbMSwwLCJcXGNTX3tYXkl9Il0sWzAsMCwiXFxjUl97WF5JfVxcdHRpbWVzIFxcY1Jfe1heSX0iXSxbMiwwLCJcXGNSX3tYXkl9Il0sWzEsMSwiXFxHcl97WF5JfSBcXHR0aW1lcyBcXGNSX3tYXkl9Il0sWzAsMSwiXFxjVF97WF5JfVxcdHRpbWVzIFxcY1Jfe1heSX0iXSxbMiwxLCJcXGNUX3tYXkl9Il0sWzAsMSwiZCIsMl0sWzAsMiwibSJdLFszLDQsIlxcdGQiLDJdLFswLDMsImkiXSxbMyw1LCJcXHRtIl0sWzIsNSwiaV8yIl0sWzEsNCwiaV8yXFx0dGltZXMgXFxpZCIsMl1d
\[\begin{tikzcd}
	{\cR_{X^I}\ttimes_{X^I} \cR_{X^I}} & {\cS_{X^I}} & {\cR_{X^I}} \\
	{\cT_{X^I}\ttimes_{X^I} \cR_{X^I}} & {\Gr_{X^I} \ttimes_{X^I} \cR_{X^I}} & {\cT_{X^I},}
	\arrow["{i_2\ttimes \id}"', from=1-1, to=2-1]
	\arrow["d"', from=1-2, to=1-1]
	\arrow["m", from=1-2, to=1-3]
	\arrow["i", from=1-2, to=2-2]
	\arrow["{i_2}", from=1-3, to=2-3]
	\arrow["\td"', from=2-2, to=2-1]
	\arrow["\tm", from=2-2, to=2-3]
\end{tikzcd}\]
and the convolution products are compatible with the factorization isomorphisms \eqref{eq: factorization isom} for $\left\{\cR_{X^I}\right\}$ and $\hGOX$. See \cite[Sec. 7.2]{CW23} for more details.

\subsubsection{Renormalized $r$-matrices in $\mathrm{h}(\Coh^{\hGO}(\cR_{G,N}))$}
The factorization monoid structure for $\left\{[\hGO\backslash\cR_{X^I}]_{X^I}\right\}_I$ gives $\left\{\IndCoh^{\hGOX}(\cR_{X^I})\right\}_{I}$ a $\Gmrot$-equivariant chiral category structure in the sense of \cite{Ras14}\footnote{The $\Gmrot$-equivariance notion follows \cite[Def. 5.4]{CW23}}. Moreover, \cite[Sec. 7.3]{CW23} shows that this chiral structure is unital. We will extract the information behind the notion of (unital) chiral category as far as we need. Then we will review how to construct a system of renormalized $r$-matrices in $\mathrm{h}(\Coh^{\hGO}(\cR))$ using these structures.

Let $\cF,\cG\in \Coh^{\hGO}(\cR)$. Since we assume $X=\A^1=\mathbb{G}_a$, the group structure on $X$ gives trivializations $\cR_X\cong \cR\times X$, $\hGOfac{X}\cong \left((G_{\O}\times X)\rtimes_{X}\cGmrot{X}\right) \times \Gmdil$. Let $\wt{\cF},\wt{\cG}\in \Coh^{\hGOfac{X}}(\cR_X)$ be the pullback of $\cF,\cG$ along the projection 
$$[\hGOfac{X}\backslash\cR_X]_{X}\cong \Big[\Gmrot\big{\backslash}\big([(G_{\O}\times \Gmdil)\backslash\cR]\times X\big)\Big]\to\Big[\Gmrot\big{\backslash}\big([(G_{\O}\times \Gmdil)\backslash\cR]\big)\Big]\cong[\hGO\backslash\cR].$$ 

A unital chiral structure contains, as part of its definition, monoidal functors
\[\eta_{I}^{I\sqcup J}:\Coh^{\hGOX}(\cR_{X^I}\times X^J)\to \Coh^{\hGOfac{X^{I\sqcup J}}}(\cR_{X^{I\sqcup J}}).\] 
These functors are compatible with the factorization isomorphisms \eqref{eq: factorization isom}. For example, after identifying $(\cR_{X^2})|_{\Delta}\cong \cR_X$ for $\Delta=\Delta_{\{12\}/\{1\}}:X\hookrightarrow X^2$, we have functorial isomorphisms
\begin{equation}\label{eq:Delta}
	(\eta^{12}_1(\mathcal{P}\boxtimes \O_X))|_{\Delta}\cong \mathcal{P},\ \ \text{for }\mathcal{P}\in \Coh^{\hGOfac{X}}(\cR_X).
\end{equation}
More precisely, this means for the embedding 
\[\iota_{\Delta}:[\hGOfac{X}\backslash \cR_{X}]_X\to [\hGOfac{X^2}\backslash \cR_{X^2}]_{X^2}\]
over the diagonal embedding
\[\Delta:[\Gmrot\backslash X]\rightarrow [\Gmrot\backslash X^2],\]
we have 
\[\iota_{\Delta}^*(\eta^{12}_1(\mathcal{P}\boxtimes \O_X))\cong \mathcal{P}.\]
After identifying $ {\cR_{X^2}}|_{U}\cong (\cR_{X}\times \cR_X )|_{U} $ for $U=X^2\setminus\Delta$, we have functorial isomorphisms
\begin{equation}\label{eq:UU}
	\eta^{12}_1(\mathcal{P}\boxtimes \O_X)|_{U}\cong (\mathcal{P}\boxtimes \O_X)|_{U},\ \ \text{for }\mathcal{P}\in \Coh^{\hGOfac{X}}(\cR_X).
\end{equation} 
More precisely, this means when we identify the restrictions of 
\[[\hGOfac{X^2}\backslash \cR_{X^2}]_{X^2} \text{ \ and \ }[\hGOfac{X}\backslash \cR_{X}]_{X}\times_{[\Gmrot\backslash \pt]}[\hGOfac{X}\backslash \cR_{X}]_{X}\]
over the open embedding
\[[\Gmrot\backslash U]\rightarrow [\Gmrot\backslash X^2],\]
the restrictions of 
\[\eta^{12}_1(\mathcal{P}\boxtimes \O_X)\text{ \ and \ } \mathcal{P}\boxtimes \O_X\]
over their corresponding opens are naturally isomorphic.
\begin{remark}
	We will abbreviate the notations as above \eqref{eq:Delta} \eqref{eq:UU}, thus we just write the restrictions as $(\ )|_{\Delta},(\ )|_U$. The actual meaning is similar to the explanations as above. 
\end{remark}

We will abbreviate $\eta_1$ for $\eta^{12}_1:\Coh^{\hGOfac{X}}(\cR_X\times X)\to \Coh^{\hGOfac{X^2}}(\cR_{X^2})$ and $\eta_2$ for $\eta^{12}_2:\Coh^{\hGOfac{X}}(X\times \cR_X)\to \Coh^{\hGOfac{X^2}}(\cR_{X^2})$. 

Define 
\begin{equation}\label{eq:rmat1}
	\begin{aligned}
	 C_{\cF,\cG}&:=\eta_1(\wt{\cF} \boxtimes \O_X) \conv \eta_2(\O_X \boxtimes \wt{\cG})\in \Coh^{\hGOfac{X^2}}(\cR_{X^2}).
	\end{aligned}
\end{equation}

According to \eqref{eq:Delta} \eqref{eq:UU}, we have 
\begin{equation}\label{eq: isom of CFG}
	\begin{gathered}
		(C_{\cF,\cG})|_{\Delta}\cong \wt{\cF}\conv \wt{\cG}\cong \wt{\cF\conv \cG},\ \ 
		(C_{\cF,\cG})|_{U}\cong (\wt{\cF}\boxtimes \wt{\cG})|_{U}.
\end{gathered}
\end{equation}

Define $ \wh{\Hom}(C_{\cF,\cG},C_{\cG,\cF}):=\bigoplus_{m\in\Z}\Hom(C_{\cF,\cG},C_{\cG,\cF}\{m\})$\footnote{Here $\Hom$ means the DG complex associated to the mapping space of objects in $\C$-linear $\infty$-categories.}. It can be viewed as a $\Gmrot$-equivariant DG-module of $\O_X(X)=\C[t_1,t_2]$. We can similarly define $\wh{\Hom}(C_{\cF,\cG}|_{U},C_{\cG,\cF}|_{U})$. By the fact that $C_{\cF,\cG}$ is a compact object, it is shown in \cite[Sec. 5.2]{CW19} that 
\begin{equation}\label{eq: Hom isom}
 \wh{\Hom}(C_{\cF,\cG}|_{U},C_{\cG,\cF}|_{U})\cong \wh{\Hom}(C_{\cF,\cG},C_{\cG,\cF})\otimes_{\C[t_1,t_2]}\C[t_1,t_2,(t_1-t_2)^{-1}].
\end{equation}

Since \eqref{eq: isom of CFG} gives isomorphisms $C_{\cF,\cG}|_{U}\cong (\wt{\cF}\boxtimes \wt{\cG})|_{U}$, $C_{\cG,\cF}|_{U}\cong (\wt{\cG}\boxtimes \wt{\cF})|_{U}$, swapping two factors gives a morphism 
$$\swap_{\cF,\cG}: C_{\cF,\cG}|_{U}\xrightarrow{\cong}C_{\cG,\cF}^{\sw}|_{U},$$
where $C_{\cG,\cF}^{\sw}:=\sw^*C_{\cG,\cF}$ for the swapping two factors map $\sw:X^2\to X^2,(a_1,a_2)\mapsto (a_2,a_1)$.\footnote{We implicitly use the isomorphism $\sw^*\cR_{X^2}\cong \cR_{X^2}$, which is one of the factorization isomorphisms} This produces a $\Gmrot$-equivariant section
\[\C[t_1,t_2,(t_1-t_2)^{-1}]\to \wh{\Hom}(C_{\cF,\cG}|_{U},C_{\cG,\cF}^{\sw}|_{U}).\]
Combining \eqref{eq: Hom isom}, we get a $ \Gmrot $-equivariant section
\begin{equation}\label{eq: alpha section}
\alpha:\C[t_1,t_2]\hookrightarrow \C[t_1,t_2,(t_1-t_2)^{-1}]\to \wh{\Hom}(C_{\cF,\cG},C_{\cG,\cF}^{\sw})\otimes_{\C[t_1,t_2]}\C[t_1,t_2,(t_1-t_2)^{-1}].
\end{equation}

The $\C[t_1,t_2]$-module $\C[t_1,t_2,(t_1-t_2)^{-1}]$ has a natural increasing filtration $F_k:=(t_1-t_2)^{-k}\C[t_1,t_2]$, which gives a natural increasing filtration on the right hand side of the above equation \eqref{eq: alpha section}. Let $\Lambda(\cF,\cG)\in \Z\bigcup \{-\infty\}$ be the minimal value of $k$ such that the section $\alpha$ factors as $$\alpha:\C[t_1,t_2]\to \wh{\Hom}(C_{\cF,\cG},C_{\cG,\cF}^{\sw})\bigotimes_{\C[t_1,t_2]} F_k.$$ 
\begin{itemize}
	\item If $\Lambda(\cF,\cG)=-\infty$, $\rmat{\cF,\cG}$ is defined to be $0$. 
	\item If $\Lambda(\cF,\cG)\neq -\infty$. We can identify $F_{\Lambda(\cF,\cG)}$ with $\C[t_1,t_2]\{\Lambda(\cF,\cG)\}$ as $\Gmrot$-equivariant $\C[t_1,t_2]$-modules. Then $\alpha$ becomes
	\[\C[t_1,t_2]\to \wh{\Hom}(C_{\cF,\cG},C_{\cG,\cF}^{\sw})\{\Lambda(\cF,\cG)\}\]
	and produces an object in $$\alpha(1)\in \Hom(C_{\cF,\cG},C_{\cG,\cF}^{\sw}\{\Lambda(\cF,\cG)\}).$$
	By restricting it to $\{0\}\subset\Delta\subset X^2$, and using the first isomorphism in \eqref{eq: isom of CFG}, we get an object $\overline{\alpha(1)}\in \cH^0\big(\Hom({\cF\conv \cG},{\cG\conv \cF}\{\Lambda(\cF,\cG)\})\big)$, i.e. a morphism 
	\[{\rmat{\cF,\cG}}:{\cF\conv \cG}\to{\cG\conv \cF}\{\Lambda(\cF,\cG)\}\]
	in $\mathrm{h}(\Coh^{\hGO}(\cR))$. For more details, see \cite[Sec. 5.2]{CW19}.
\end{itemize}

In particular, when $\Lambda(\cF,\cG)=0$, the morphism $\swap_{\cF,\cG}: C_{\cF,\cG}|_{U}\xrightarrow{\cong} C_{\cG,\cF}^{\sw}|_{U}$ given by swapping two factors can extend (uniquely in $\mathrm{h}(\Coh^{\hGOfac{X^2}}(\cR_{X^2}))$) to a morphism $$\overline{\swap}_{\cF,\cG}:C_{\cF,\cG}\to C_{\cG,\cF}^{\sw},$$
and $\rmat{\cF,\cG}:\cF\conv\cG\to \cG\conv\cF$ is the restriction of $\overline{\swap}_{\cF,\cG}$ over $\{0\}\in X^2$.

We end this subsection by stating \cite[Thm. 5.10]{CW19}.
\begin{theorem}[{\cite[Thm. 5.10]{CW19}}]
	The morphisms $\rmat{\cF,\cG}$ and elements $\Lambda(\cF,\cG)\in \Z\bigcup\{-\infty\}$ defined above form a system of renormalized $r$-matrices in $\mathrm{h}(\Coh^{\hGO}(\cR))$. 

\end{theorem}
Then restricting to the Koszul perverse heart, we obtain a system of renormalized $r$-matrices in $\KPcoh^{\hGO}(\cR)$.

\begin{remark}\label{rmk: rmat neq 0}
	According to \cite{CW19}[Cor. 4.7], the rigidity of the monoidal abelian category $\KPcoh^{\hGO}(\cR)$ ensures that $\Lambda(\cF,\cG)\neq -\infty$ (thus $\rmat{\cF,\cG}\neq 0$) for all nonzero $\cF,\cG\in \KPcoh^{\hGO}(\cR)$.
\end{remark}

\section{The full subcategory \texorpdfstring{$\KP_0\subset \KPcoh^{\hGO}(\cR_{G,\g})$}{KP0}}\label{sec:KP_0}
In the following sections, we will concentrate on the case when $N=\g$, the adjoint representation of $G$ (except Section \ref{subsec:ring_hom} \ref{subsec:sigi}). 

We introduce the following subcategory.
\begin{definition}\label{def: def of KP_0}
    Let $\KP_0$ be the full subcategory of $\KPcoh^{\hGO}(\cR_{G,\g})$ generated by simple objects $\cL_{\lav}$ (defined in (\ref{simple object cL_lav})) for all $\lav\in \Pv_+$ under taking direct sums, convolutions, subquotients and right and left duals.
\end{definition}
By definition, $\KP_0$ is a rigid monoidal abelian full subcategory of $\KPcoh^{\hGO}(\cR_{G,\g})$.

\subsection{Computation of duals in \texorpdfstring{$\KP_0$}{KP0}}
The remaining task of this section is to compute right and left duals in $\KP_0$ (Proposition \ref{Prop for calculation of dual}), and then take the initial step toward describing $\KP_0$ (Corollary \ref{cor after calculation of dual}). 

\begin{proposition}\label{Prop for calculation of dual}
    For $\lav\in \Pv_+$ and the simple object $\cL_{\lav}\in \KPcoh^{\hGO}(\cR_{G,\g})$, we have 
    $$\cL^R_{\lav}\cong \cL^L_{\lav}\cong \cL_{{\lav}^{\dagger}}.$$
    Here ${\lav}^{\dagger}:=-w_0(\lav) $.
\end{proposition}

\begin{proof}

According to Proposition \ref{prop. dual in Coh(R)}, the right and left duals of $\cL_{\lav}$ can be computed as
\[\cL_{\lav}^R\cong \D(\cF)^*, \cL^L_{\lav}\cong \D(\cF^*).\]
We refer the readers to Section \ref{subsubsec:rigidity} for constructions of the functors $(-)^*$ and $\D(-)$. 

So we divide the proof of this proposition into two steps, which are summarized as two lemmas as follows.
\begin{lemma}\label{lem1 calculate star}
	For $\lav\in \Pv_+$, we have $(\cL_{\lav})^*\cong \cL_{{\lav}^{\dagger}}$.
\end{lemma}
\begin{proof}
	
	Recall that using the morphism $\mathrm{inv}:\hGK\to\hGK,g\mapsto g^{-1}$ as in \eqref{eq:inv}, we construct two involutions
	\[\iota_{\Gr}:[\hGO\backslash \Gr]\to [\hGO\backslash \Gr]\]
	\[\iota:[\hGO\backslash \cR]\to [\hGO\backslash \cR]\]
	as in \eqref{eq:iotaGr} and \eqref{eq: diagram of iota}.
	
	Notice that $\mathrm{inv}(t^{\muv})=t^{-\muv}=\dot{w_0}t^{{\muv}^{\dagger}}\dot{w_0}^{-1}$; here $\dot{w_0}\in N(T)$ is a representative of the longest element $w_0$ in the Weyl group $W=N(T)/T$. Then we have $\iota_{\Gr}([\hGO\backslash\Gr_{\le \lav}])=[\hGO\backslash\Gr_{\le {\lav}^{\dagger}}]$, i.e. we have the following Cartesian diagram
   % https://q.uiver.app/#q=WzAsNCxbMCwwLCJcXGhHT1xcYmFja3NsYXNoXFxHcl97XFxsZXtcXGxhdn1eKn0iXSxbMSwwLCJcXGhHT1xcYmFja3NsYXNoXFxHcl97XFxsZXtcXGxhdn19Il0sWzAsMSwiXFxoR09cXGJhY2tzbGFzaCBcXEdyIl0sWzEsMSwiXFxoR09cXGJhY2tzbGFzaCBcXEdyIl0sWzAsMSwiXFxpb3RhX3tcXEdyfSJdLFswLDIsImlfe1xcbGUge1xcbGF2fV4qfSIsMl0sWzEsMywiaV97XFxsZSBcXGxhdn0iXSxbMiwzLCJcXGlvdGFfe1xcR3J9IiwyXV0=
\[\begin{tikzcd}
	{[\hGO\backslash\Gr_{\le{\lav}^{\dagger}}]} & {[\hGO\backslash\Gr_{\le{\lav}}]} \\
	{[\hGO\backslash \Gr]} & {[\hGO\backslash \Gr].}
	\arrow["{\iota_{\Gr}}", from=1-1, to=1-2]
	\arrow["{i_{\le {\lav}^{\dagger}}}"', from=1-1, to=2-1]
	\arrow["{i_{\le \lav}}", from=1-2, to=2-2]
	\arrow["{\iota_{\Gr}}", from=2-1, to=2-2]
\end{tikzcd}\]
By doing base change along the Cartesian diagram \eqref{eq: diagram of iota}, we get the following Cartesian diagram
% https://q.uiver.app/#q=WzAsNCxbMCwwLCJcXGhHT1xcYmFja3NsYXNoXFxjUl97XFxsZXtcXGxhdn1eKn0iXSxbMSwwLCJcXGhHT1xcYmFja3NsYXNoXFxjUl97XFxsZXtcXGxhdn19Il0sWzAsMSwiXFxoR09cXGJhY2tzbGFzaCBcXGNSIl0sWzEsMSwiXFxoR09cXGJhY2tzbGFzaCBcXGNSIl0sWzAsMSwiXFxpb3RhX3tcXGNSfSJdLFswLDIsImlfe1xcbGUge1xcbGF2fV4qfSIsMl0sWzEsMywiaV97XFxsZSBcXGxhdn0iXSxbMiwzLCJcXGlvdGFfe1xcY1J9IiwyXV0=
\[\begin{tikzcd}
	{[\hGO\backslash\cR_{\le{\lav}^{\dagger}}]} & {[\hGO\backslash\cR_{\le{\lav}}]} \\
	{[\hGO\backslash \cR]} & {[\hGO\backslash \cR].}
	\arrow["{\iota}", from=1-1, to=1-2]
	\arrow["{i_{\le {\lav}^{\dagger}}}"', from=1-1, to=2-1]
	\arrow["{i_{\le \lav}}", from=1-2, to=2-2]
	\arrow["{\iota}", from=2-1, to=2-2]
\end{tikzcd}\]

It follows that $(\cL_{\lav})^*\cong\iota^*(i_{\le \lav })_*(\cL'_{\lav})\cong (i_{\le {\lav}^{\dagger}})_*\iota^*(\cL_{\lav}')$.

Restricting $\iota^*(\cL_{\lav}')$ to $j_{{\lav}^{\dagger}}:\cR_{{\lav}^{\dagger}}\hookrightarrow \cR_{\le {\lav}^{\dagger}}$, we get $$j_{{\lav}^{\dagger}}^*\iota^*(\cL_{\lav}')\cong \iota^* j_{\lav}^*(\cL_{\lav}')\overset{\eqref{eq: j_{lav}^*simple}}{\cong} \iota^*\left(\cl_*{p^{\cl*}}(\O_{\Gr_{\lav}}\lr{-\tfrac12 d_{\lav}})\right).$$
Using the Cartesian diagram
% https://q.uiver.app/#q=WzAsNCxbMCwxLCJcXGhHT1xcYmFja3NsYXNoXFxjUl97e1xcbGF2fV4qfSJdLFsxLDEsIlxcaEdPXFxiYWNrc2xhc2hcXGNSX3t7XFxsYXZ9fSJdLFswLDAsIlxcaEdPXFxiYWNrc2xhc2hcXGNSX3t7XFxsYXZ9Xip9XntcXGNsfSJdLFsxLDAsIlxcaEdPXFxiYWNrc2xhc2hcXGNSX3t7XFxsYXZ9fV57XFxjbH0iXSxbMCwxLCJcXGlvdGEiLDJdLFsyLDAsIlxcY2wiXSxbMywxLCJcXGNsIiwyXSxbMiwzLCJcXGlvdGFee1xcY2x9Il1d
\[\begin{tikzcd}
	{[\hGO\backslash\cR_{{\lav}^{\dagger}}^{\cl}]} & {[\hGO\backslash\cR_{{\lav}}^{\cl}]} \\
	{[\hGO\backslash\cR_{{\lav}^{\dagger}}]} & {[\hGO\backslash\cR_{{\lav}}],}
	\arrow["{\iota^{\cl}}", from=1-1, to=1-2]
	\arrow["\cl"', from=1-1, to=2-1]
	\arrow["\cl", from=1-2, to=2-2]
	\arrow["\iota", from=2-1, to=2-2]
\end{tikzcd}\]

we deduce $$\iota^*\cl_*{p^{\cl*}}(\O_{\Gr_{\lav}}\lr{-\tfrac12 d_{\lav}})\cong \cl_*{{\iota^{\cl}}^*}{p^{\cl*}}(\O_{\Gr_{\lav}}\lr{-\tfrac12 d_{\lav}})\cong \cl_*{p^{\cl*}}(\O_{\Gr_{{\lav}^{\dagger}}}\lr{-\tfrac12 d_{\lav^{\dagger}}}).$$. 

In conclusion, 
\begin{equation}\label{eq:()}
	(\cL_{\lav})^*\cong (i_{\le {\lav}^{\dagger}})_*\iota^*(\cL_{\lav}') \text{\ \ and \ \ } j_{{\lav}^{\dagger}}^*\iota^*(\cL_{\lav}')\cong j_{{\lav}^{\dagger}}^*(\cL_{{\lav}^{\dagger}}').
\end{equation}
Since $(-)^*$ is an involution on $\KPcoh^{\hGO}(\cR)$, it sends simple objects to simple objects, so $(\cL_{\lav})^*$ has the form $\cL_{{\lav}',\mu'}$ for some $(\lav',\mu')\in (\Pv\times P)_{\dom}$. By \eqref{eq:()}, $({\lav}',\mu')$ can only be $({\lav}^{\dagger},0)$, thus $(\cL_{\lav})^*\cong \cL_{{\lav}^{\dagger}}$.
\end{proof}

The second lemma is more crucial, which is specific to our situation $N=\g$, especially the computations in \eqref{eq: normal bundle calculation} below.
\begin{lemma}\label{lem2 calculation of D}
    For $\lav\in \Pv_+$, we have $\D(\cL_{\lav})\cong \cL_{\lav}$.
\end{lemma}
\begin{proof}
	Firstly, we perform the following computations: 
    \begin{equation*}
        \begin{aligned}
        \D(\cL_{\lav})=\D(i_{\le \lav*}(\cL_{\lav}'))&{=}\cHom(i_{\le \lav*}\cL_{\lav}',i_1^!p_{1}^*\omega_{\Gr}) \ \ \ \ \ \ \text{use \eqref{def of D}}\\
        &{\cong} i_{\le {\lav}*}\cHom(\cL_{\lav}',i_{\le {\lav}}^!i_1^!p_{1}^*\omega_{\Gr}) \ \ \text{use \eqref{eq:f_*,f^!}}\\
        &\cong i_{\le {\lav}*}\cHom(\cL_{\lav}',i_1^!i_{\le {\lav}}^!p_{1}^*\omega_{\Gr});
        \end{aligned}
    \end{equation*}
	here, we abuse notation: $i_1$ represents both $\cR_{\lav}\hookrightarrow \Gr_{\lav}\times \g_{\O}$,  $\cR\hookrightarrow \Gr\times \g_{\O}$, and $i_{\le \lav}$ represents both $\cR_{\lav}\hookrightarrow \cR$ and $\Gr_{\lav}\times\g_{\O}\hookrightarrow \Gr\times\g_{\O}$.
    Applying Beck-Chevalley isomorphism \eqref{eq:!and*} to the Cartesian diagram 
    % https://q.uiver.app/#q=WzAsNCxbMCwxLCJcXEdyX3tcXGxlXFxsYXZ9Il0sWzEsMSwiXFxHciJdLFswLDAsIlxcR3Jfe1xcbGVcXGxhdn1cXHRpbWVzIE5fe1xcT30iXSxbMSwwLCJcXEdyXFx0aW1lcyBOX3tcXE99Il0sWzAsMSwiaV97XFxsZSBcXGxhdn0iLDJdLFsyLDAsInBfMSJdLFszLDEsInBfMSIsMl0sWzIsMywiaV97XFxsZSBcXGxhdn0iXV0=
\[\begin{tikzcd}
	{\Gr_{\le\lav}\times \g_{\O}} & {\Gr\times \g_{\O}} \\
	{\Gr_{\le\lav}} & {\Gr,}
	\arrow["{i_{\le \lav}}", from=1-1, to=1-2]
	\arrow["{p_1}"', from=1-1, to=2-1]
	\arrow["{p_1}", from=1-2, to=2-2]
	\arrow["{i_{\le \lav}}", from=2-1, to=2-2]
\end{tikzcd}\]
we get
\begin{equation}\label{eq: D(L_{lav})}
    \D(\cL_{\lav})\cong i_{\le {\lav}*}\cHom(\cL_{\lav}',i_1^!p_{1}^*i_{\le {\lav}}^!(\omega_{\Gr}))  \cong i_{\le \lav*}{\cHom(\cL_{\lav}',i_{1}^!p_1^*(\omega_{\Gr_{\le \lav}}))}=i_{\le \lav*}\D_{\cR_{\le \lav}}(\cL_{\lav}');
\end{equation}
here, we define 
\begin{equation}\label{eq:D_R}
	\D_{\cR_{\le \lav}}(-):=\cHom(-,i_1^!p_1^*(\omega_{\Gr_{\le \lav}})):\Coh^{\hGO}(\cR_{\le\lav})^{\op}\to\Coh^{\hGO}(\cR_{\le\lav}).
\end{equation}

Next, we compute $j_{\lav}^*\D_{\cR_{\le \lav}}(\cL_{\lav}')$ as follows:
\begin{equation}\label{eq: j^*Hom}
    \begin{aligned}
    & j_{\lav}^*{\cHom(\cL_{\lav}',i_{1}^!p_1^*(\omega_{\Gr_{\le \lav}}))}\\
    \overset{\eqref{eq:h^*cHom}}{\cong}& \cHom(j_{\lav}^*\cL_{\lav}',j_{\lav}^*i_{1}^!p_1^*(\omega_{\Gr_{\le \lav}}))\\
    \overset{\eqref{eq: j_{lav}^*simple}}{\cong} &\cHom(\cl_*{p^{\cl*}}(\O_{\Gr_{\lav}}\lr{-\tfrac12 d_{\lav}}),j_{\lav}^*i_{1}^!p_1^*(\omega_{\Gr_{\le \lav}}))\\
    \overset{\eqref{eq:f_*,f^!}}{\cong} & \cl_*\cHom(\O_{\cR_{\lav}^{\cl}},\cl^!i_1^!p_1^*(\omega_{\Gr_{\lav}}))\lr{\tfrac12 d_{\lav}}\\
    \overset{ \ }{\cong} \ \ &\cl_*{i_1^{\cl }}^!p_1^*(\omega_{\Gr_{\lav}})\lr{\tfrac12 d_{\lav}}.
\end{aligned}
\end{equation}
Here, in the last term, the morphisms are $\cR_{\lav}^{\cl}\xhookrightarrow{i_1^{\cl}}\Gr_{\lav}\times \g_{\O}\xrightarrow{p_1}\Gr_{\lav}$. 

Recall that we denote the stabilizer of $t^{\lav}\in \Gr_{\lav}$ with respect to the $\hGO$-action by $\hPla$. We also restrict the $\hGO$-action to $G_{\O}$, and denote the stabilizer of $t^{\lav}$ by $\Pla$. Let $\pla$ be the Lie algebra of $\Pla$, which can be viewed as a subspace in $\g_{\O}$ with finite codimension. Now we make the following identification of vector bundles over $\Gr_{\lav}\cong \hGO/\hPla\cong G_{\O}/\Pla$:
% https://q.uiver.app/#q=WzAsNixbMCwxLCJHX3tcXE99XFx0aW1lc157XFxoUGxhfVxcaHBsYSAiXSxbMSwxLCJHX3tcXE99XFx0aW1lc157XFxoUGxhfVxcZ197XFxPfSJdLFsyLDEsIkdfe1xcT30vXFxoUGxhIl0sWzAsMCwiXFxjUl97XFxsYXZ9Il0sWzEsMCwiXFxHcl97XFxsYW1iZGFee1xcdmVlfX1cXHRpbWVzXFxnX3tcXE99Il0sWzIsMCwiXFxHcl97XFxsYW1iZGFee1xcdmVlfX0iXSxbMCwxLCJpXzEiXSxbMiwxLCJcXHNpZ21hXzEiLDJdLFszLDQsImlfMSJdLFs1LDQsIlxcc2lnbWFfMSIsMl0sWzEsNCwiXFx0ZXh0e1xcZGluZ3sxNzN9fSJdLFsyLDUsIlxcdGV4dHtcXGRpbmd7MTc0fX0iXSxbMCwzLCJcXHRleHR7XFxkaW5nezE3Mn19Il0sWzAsMywiXFxjb25nIiwyXSxbMSw0LCJcXGNvbmciLDJdLFsyLDUsIlxcY29uZyIsMl1d
	\begin{equation}\label{identification of cRlav}
\begin{tikzcd}
	{\cR_{\lav}^{\cl}} & {\Gr_{{\lav}}\times\g_{\O}} & {\Gr_{{\lav}}} \\
	{\hGO\times^{\hPla}\pla } & {\hGO\times^{\hPla}\g_{\O}} & {\hGO/\hPla.}
	\arrow["{i_1^\cl}", from=1-1, to=1-2]
	\arrow["{\sigma_1}"', from=1-3, to=1-2]
	\arrow["{\text{\ding{172}}}", from=2-1, to=1-1]
	\arrow["\cong"', from=2-1, to=1-1]
	\arrow["{i_1^{\cl}}", from=2-1, to=2-2]
	\arrow["{\text{\ding{173}}}", from=2-2, to=1-2]
	\arrow["\cong"', from=2-2, to=1-2]
	\arrow["{\text{\ding{174}}}", from=2-3, to=1-3]
	\arrow["\cong"', from=2-3, to=1-3]
	\arrow["{\sigma_1}"', from=2-3, to=2-2]
\end{tikzcd}
	\end{equation}
	
	Here, the three column maps are 
	\begin{equation*}
		\begin{aligned}
			&\text{\ding{172}}:[g,x]\mapsto [gt^{\lav},\mathrm{Ad}_{t^{-{\lav}}}x],\\
			&\text{\ding{173}}:[g,x]\mapsto ([gt^{{\lav}}],\mathrm{Ad}_{g}x),\\
			&\text{\ding{174}}:[g]\mapsto [gt^{{\lav}}].
		\end{aligned}
	\end{equation*}

	Using \eqref{identification of cRlav}, we identify the normal bundle of inclusion $i_1^{\cl}:\cR_{\lav}^{\cl}\hookrightarrow \Gr_{\lav}\times \g_{\O}$ with $p_1^*\left(\hGO\times^{\hPla}(\g_{\O}/\pla)\right)$. So 
	\begin{equation}\label{eq: normal bundle calculation}
		 {i_1^{\cl}}^! p_1^*(\omega_{\Gr_{\lav}})\cong {i_1^{\cl}}^* p_1^*\left({\omega_{\hGO/\hPla}}\otimes {\det(\hGO\times^{\hPla}(\g_{\O}/\pla))}\right)[-\dim(\g_{\O}/\pla)].
	\end{equation}

	Notice that $\omega_{\hGO/\hPla}$ can also be identified with $ \det(\hGO\times^{\hPla}((\g_{\O}/\pla)^*))[\dim(\g_{\O}/\pla)]$, but here $(\g_{\O}/\pla)^*$ has \textbf{trivial} $\Gmdil$-action. 
	Then the right hand side of \ref{eq: normal bundle calculation} becomes 
	\begin{equation}\label{eq: calculation of omega}
		{i_1^{\cl}}^*p_1^*(\O_{\Gr_{\lav}}\lr{-\dim (\g_{\O}/\pla)})\cong \O_{\cR_{\lav}^{\cl}}\lr{-d_{\lav}}.
	\end{equation}

    Combining equations \eqref{eq: j^*Hom}, \eqref{eq: normal bundle calculation}, \eqref{eq: calculation of omega}, we conclude that
    \begin{equation}\label{eq: 1}
        j_{\lav}^*\D_{\cR_{\le \lav}}(\cL_{\lav}')\cong\cl_* \O_{\cR_{\lav}^{\cl}}\lr{-\tfrac12 d_{\lav}}\cong j_{\lav}^*\cL_{\lav}'.
    \end{equation}

    Then we use the argument in previous lemma \ref{lem1 calculate star} again. Since $\D$ is an involution on $\KPcoh^{\hGO}(\cR)$, $\D(\cL_{\lav})$ is also a simple object of the form $\cL_{{\lav}',\mu'}$ for some $({\lav}',\mu')\in (\Pv\times P)_{\dom}$. By \eqref{eq: D(L_{lav})} and \eqref{eq: 1}, $({\lav}',\mu')$ can only be $(\lav,0)$, thus $\D(\cL_{\lav})\cong \cL_{\lav}$.
\end{proof}

After these two lemmas, we complete the proof of Proposition \ref{Prop for calculation of dual}.
\end{proof}

We conclude this section by the following corollary.
\begin{corollary}\label{cor after calculation of dual}
	The full subcategory $\KP_0\subset \KP_{G,\g}$ defined in Definition \ref{def: def of KP_0} is generated by simple objects $\cL_{\lav}$ for all $ \lav\in \Pv_+$ under direct sums, convolutions, subquotients. This means we don't need to additionally add duals as in Definition \ref{def: def of KP_0}.
\end{corollary}
\begin{proof}
	Let $\cC$ be the full subcategory of $\KP_0$ generated by $\cL_{\lav}$ for all $ \lav\in \Pv_+$ under direct sums, convolutions, subquotients. This is a monoidal abelian full subcategory, and we need to show it's rigid.
	
	First, by the Proposition \ref{Prop for calculation of dual} above, we know all the left and right duals of $\cL_{\lav}$ lie in $\cC$. Second, we use the following properties for (left/right) dual: if $\cF,\cG\in \cC$ are (left/right) dualizable in $\cC$, then
	\begin{itemize}
		\item $(\cF\oplus\cG)^L\cong \cF^L\oplus \cG^L$, $(\cF\oplus\cG)^R\cong \cF^R\oplus \cG^R$; 
		\item the (left/right) dual of a subquotient of $\cF$ is a subquotient of the (left/right) dual of $\cF$;
		\item $(\cF\conv\cG)^L\cong \cG^L\conv \cF^L$, $(\cF\conv\cG)^R\cong \cG^R\conv \cF^R$.
	\end{itemize}
	Using these observations and the definition of $\cC$, we know $\cC$ is rigid. From the definition of $\KP_0$, we conclude that $\cC=\KP_0$.
\end{proof}

\section{Determine \texorpdfstring{$\KP_0$}{KP0} at the \texorpdfstring{$K$}{K}-group level}\label{sec:K_group}
In this section, we describe $\KP_0$ at the $K$-group level. We will assume $G$ is \textbf{of type A} from now on, although this assumption is only crucially used in Section \ref{subsec: Tools coming from Hodge module}.

The main theorem of this section is as follows.
\begin{theorem}\label{thm: determine KP_0 on K-group}
	In the $K$-group $K(\KP_0)\subset K^{\hGO}(\cR_{G,\g})$, we have the following identity:
	\begin{equation*}
		[\cL_{\lav}]\conv [\cL_{\muv}]=\sum_{\nuv\in \Pv_+}\clamunu [\cL_{\nuv}]. 
	\end{equation*}
    Here, the notation $\clamunu$ is defined in \eqref{def of clamunu}.

    In particular, simple objects in $\KP_0$ are precisely those $\cL_{\lav}$ for $\lav\in \Pv_+$.
\end{theorem}

The idea of the proof is as follows: we try to relate $\cL_{\lav}$ to Hodge modules, and then use the knowledge on geometric Satake equivalence to deduce the above identity.

\subsection{The ring homomorphism \texorpdfstring{$\Xi:=\sigi:K^{\hGO}(\cR_{G,N})\to K^{\hGO}(\Gr_G)$}{Ξ}}\label{subsec:ring_hom}
In this subsection, we deal with general $N$.

The convolution product on $\KPcoh^{\hGO}(\cR_{G,N})$ is hard to compute. Our strategy is to apply the functor $$\Xi:=\sigi:\Coh^{\hGO}(\cR)\to \Coh^{\hGO}(\Gr)$$ 
for morphisms $\cR\xhookrightarrow{i_1}\Gr\times \g_{\O}\xleftarrow{\sigma_1}\Gr$, and perform all computations on the affine grassmannian $\Gr$. Although $\sigi$ is not monoidal, it truly induces a ring homomorphism at the $K$-group level.

\begin{proposition}\label{prop:sigi is ring hom}
    The maps $\sigi,\sjgj:K^{\hGO}(\cR_{G,N})\to K^{\hGO}(\Gr_G)$ are two ring homomorphisms.
\end{proposition}
This proposition essentially follows from the proof in \cite[Sec. 5.3]{CW23}\footnote{See the last paragraph of that section.}. We present this proof below for the purpose of completeness. We remark that the same proposition with $K$-group replaced by Borel-Moore homology also appears in \cite[Lem. 5.11]{BFNII}. 

\begin{proof} We first prove that $\sjgj$ is a ring homomorphism at the $K$-group level, and subsequently transfer​ this result to $\sigi$ via the involution $(-)^*$ as in \eqref{eq:(-)^*}.
    
\begin{enumerate}[label=\textbf{Step \arabic*}]
    \item ($\sjgj$ is a ring homomorphism) Let $\cF,\cG\in \KPcoh^{\hGO}(\cR)$. Following the notations in diagram \eqref{eq:convS}, we computations as follows
    \begin{equation*}
        \begin{aligned}
            \sjgj(\cF\conv \cG)&\cong \sjgj {m}_*d^*(\cF\tbox \cG)\\
            &\cong \sigma_2^*\wt{m}_*i_*d^*(\cF\tbox \cG)\\
            &\cong \sigma_2^*\wt{m}_*\wt{d}^*(i_{2*}(\cF)\tbox \cG);
        \end{aligned}
    \end{equation*}
	here, the first isomorphism is by definition of convolution product, and the third isomorphism is by the base change isomorphism for stable coherent pullback and almost ind-finitely presented ind-proper pushforward. 
	
	According to Theorem \ref{thm:koszul-perverse}, $i_{2*}(\cF)$ belongs to the Koszul perverse heart $\KPcoh^{\hGO}(\cT)$, which is a finite length abelian category. Let $\left\{\wt{\cF_1},\dots, \wt{\cF_n}\right\}$ be the simple composition factors of $i_{2*}(\cF)$. Then $\sigma_2^*\wt{m}_*\wt{d}^*(i_{2*}(\cF)\tbox \cG)$ is an iterated extension of $\left\{\sigma_2^*\wt{m}_*\wt{d}^*(\wt{\cF}_1\tbox \cG),\dots , \sigma_2^*\wt{m}_*\wt{d}^*(\wt{\cF}_n\tbox \cG)\right\}$.
    
    The description for simple objects of $\KPcoh^{\hGO}(\cT)$ in \eqref{eq: simple objects bijection} implies that each $\wt{\cF_i}$ is isomorphic to $p_2^*(\cF_i)$ for the simple object $\cF_i:=\sigma_2^*(\wt{\cF_i})\in \KPcoh^{\hGO}(\Gr)$. Thus we have:
    \[\sigma_2^*\wt{m}_*\wt{d}^*(\wt{\cF}_i\tbox \cG)\cong \sigma_2^*\wt{m}_*\wt{d}^*(p_2^*(\cF_i)\tbox \cG)\cong \sigma_{2}^*\wt{m}_*(\cF_i\tbox \cG);\]
    the reason for the last isomorphism is that the composition $(p_2\ttimes \id)\circ \wt{d}$ is the identity map.

    Now one observes that $\tm$ factors as the composition along the bottom row of the diagram below
	% https://q.uiver.app/#q=WzAsNSxbMSwxLCJcXEdyXFx0dGltZXMgXFxjVCJdLFsyLDEsIlxcY1QiXSxbMSwwLCJcXEdyXFx0dGltZXMgXFxHciJdLFsyLDAsIlxcR3IiXSxbMCwxLCJcXEdyXFx0dGltZXMgXFxjUiJdLFswLDEsIm1fe1xcR3J9XFx0dGltZXMgXFxpZF97Tl97XFxPfX0iLDJdLFsyLDAsIlxcaWRcXHR0aW1lcyBcXHNpZ21hXzIiLDJdLFszLDEsIlxcc2lnbWFfMiJdLFsyLDMsIm1fe1xcR3J9Il0sWzQsMCwiXFxpZFxcdHRpbWVzIGlfMiIsMl1d
\[\begin{tikzcd}
	& {\Gr\ttimes \Gr} & \Gr \\
	{\Gr\ttimes \cR} & {\Gr\ttimes \cT} & \cT
	\arrow["{m_{\Gr}}", from=1-2, to=1-3]
	\arrow["{\id\ttimes \sigma_2}"', from=1-2, to=2-2]
	\arrow["{\sigma_2}", from=1-3, to=2-3]
	\arrow["{\id\ttimes i_2}"', from=2-1, to=2-2]
	\arrow["{m_{\Gr}\ttimes \id_{N_{\O}}}"', from=2-2, to=2-3]
\end{tikzcd}\]
	where $m_\Gr$ is the convolution map of $\Gr$. The right square is Cartesian, hence 
	\begin{align*}
		\sigma_2^* \tm_* (\cF_i \tbox \cG) 
		&\cong \sigma_2^* (m_\Gr \ttimes \id_{N_\O})_* (id \ttimes i_2)_* (\cF_i \tbox \cG) \\
		&\cong m_{\Gr*} (\id \ttimes \sigma_2)^* (id \ttimes i_2)_* (\cF_i \tbox \cG) \\
		&\cong m_{\Gr*} (\cF_i \tbox \sigma_2^*i_{2*}(\cG)) \\
		&\cong \cF_i \conv \sigma_2^*i_{2*}(\cG).
	\end{align*}
    In conclusion, $\sjgj(\cF\conv \cG)$ is an iterated extension of $\left\{\cF_i \conv \sigma_2^*i_{2*}(\cG)\colon i=1,\dots,n\right\}$. Since $i_{2*}(\cF)$ is an iterated extension of $\left\{
    \wt{\cF_i}\colon i=1,\dots,n\right\}$, $\sjgj(\cF)$ is an iterated extension of $\left\{
    \cF_i=\sigma_2^*(\wt{\cF_i})\colon i=1,\dots,n\right\}$. So in the $K$-group $K^{\hGO}(\Gr)$, we have
    \begin{align*}
        \sjgj([\cF]\conv [\cG])&=\sum_{i=1}^n [\cF_i]\conv\sjgj([\cG]),\\ \sjgj([\cF])&=\sum_{i=1}^n [\cF_i].
    \end{align*}
    Combining these two equations, we deduce $\sjgj([\cF]\conv [\cG])=\sjgj([\cF])\conv\sjgj([\cG])$, i.e. $\sjgj$ is a ring homomorphism.

    \item ($\sigi$ is a ring homomorphism)
    
    \textit{Claim:} $\sigi\cong \iota_{\Gr}^*\circ \sjgj \circ \iota_*: \Coh^{\hGO}(\cR)\to \Coh^{\hGO}(\Gr)$; here the morphisms $\iota_{\Gr}$ and $\iota$ are defined in \eqref{eq:iotaGr} and \eqref{eq: diagram of iota}.

    Via the construction of $ \iota $ in diagrams \eqref{diagram1} \eqref{diagram2} \eqref{diagram3}, we obtain the following two Cartesian diagrams:
    % https://q.uiver.app/#q=WzAsOCxbMCwxLCJcXGhHT1xcYmFja3NsYXNoKFxcR3JcXHRpbWVzIE5fe1xcT30pIl0sWzEsMSwiXFxoR09cXGJhY2tzbGFzaFxcY1QiXSxbMCwwLCJcXGhHT1xcYmFja3NsYXNoIFxcY1IiXSxbMSwwLCJcXGhHT1xcYmFja3NsYXNoIFxcY1IiXSxbMywwLCJcXGhHT1xcYmFja3NsYXNoKFxcR3JcXHRpbWVzIE5fe1xcT30pIl0sWzMsMSwiXFxoR09cXGJhY2tzbGFzaCBcXEdyIl0sWzQsMSwiXFxoR09cXGJhY2tzbGFzaCBcXEdyIl0sWzQsMCwiXFxoR09cXGJhY2tzbGFzaFxcY1QiXSxbMCwxLCJcXHd0e1xcaW90YX0iLDJdLFsyLDAsImlfMSIsMl0sWzMsMSwiaV8yIl0sWzIsMywiXFxpb3RhIl0sWzUsNCwiXFxzaWdtYV8xIl0sWzUsNiwiXFxpb3RhX3tcXEdyfSIsMl0sWzYsNywiXFxzaWdtYV8yIiwyXSxbNCw3LCJcXHd0e31cXGlvdGEiXV0=
\[\begin{tikzcd}
	{[\hGO\backslash \cR]} & {[\hGO\backslash \cR]} && {[\hGO\backslash(\Gr\times N_{\O})]} & {[\hGO\backslash\cT]} \\
	{[\hGO\backslash(\Gr\times N_{\O})]} & {[\hGO\backslash\cT];} && {[\hGO\backslash \Gr]} & {[\hGO\backslash \Gr].}
	\arrow["\iota", from=1-1, to=1-2]
	\arrow["{i_1}"', from=1-1, to=2-1]
	\arrow["{i_2}", from=1-2, to=2-2]
	\arrow["{\wt{\iota}}", from=1-4, to=1-5]
	\arrow["{\wt{\iota}}", from=2-1, to=2-2]
	\arrow["{\sigma_1}", from=2-4, to=1-4]
	\arrow["{\iota_{\Gr}}", from=2-4, to=2-5]
	\arrow["{\sigma_2}"', from=2-5, to=1-5]
\end{tikzcd}\]
    Then we can deduce our claim:
    \[\sigi\cong \sigma_1^*\wt{\iota}^*\wt{\iota}_*i_{1*}\cong (\wt{\iota}\sigma_1)^*(\wt{\iota}i_1)_*\cong (\sigma_2\iota_{\Gr})^*(i_2\iota)_*\cong \iota_{\Gr}^*\sigma_2^*i_{2*}\iota_*.\]

    After identifying $[\hGO\backslash\cR]$ with $$ [\hGO\backslash N_{\O}]\underset{[\hGK\backslash N_{\K}]}{\times}[\hGO\backslash N_{\O}]$$
    as in Proposition \ref{prop: Rsymmdiagram}, $\iota$ becomes the morphism exchanging two factors, then the involution $\iota_*$ is anti-monoidal by the definition of convolution product \eqref{eq: conv1}, i.e. $\iota_*(\cF\conv \cG)\cong \iota_*(\cG)\conv \iota_*(\cF)$. Similarly, $\iota_\Gr^*$ is also anti-monoidal. Combing the claim above and Step 1, we deduce that $\sigi:K^{\hGO}(\cR)\to K^{\hGO}(\Gr)$ is a ring homomorphism.

\end{enumerate}
    
\end{proof}

\subsection{Computation of \texorpdfstring{$\Xi(\cL_{\lav,\mu})$}{Ξ(L)}}\label{subsec:sigi}
In this subsection, we deal with general $N$. The goal of this subsection is to compute $\Xi(\cL_{\lav,\mu})$ for the simple object $\cL_{\lav,\mu}\in \KPcoh^{\hGO}(\cR)$.

Define $N_{\lav}$ to be the classical intersection $N_{\O}\overset{\cl}{\cap} t^{\lav} N_{\O} $ in $N_{\K}$; it is a classical scheme (of infinite type) and can be viewed as a vector space with finite codimension in $N_{\O}$. Recall that we identify $\Gr_{\lav}$ with $\hGO/\hPla$ as $\hGO$-homogeneous spaces.

\begin{proposition}\label{prop: sigi of L}
	 For $(\lav,\mu)\in (\Pv\times P)_{\dom}$, consider simple objects $\cL_{\lav,\mu}\in \KPcoh^{\hGO}(\cR)$, $\cL'_{\lav,\mu}\in \KPcoh^{\hGO}(\cR_{\le \lav})$. We have isomorphisms
	\begin{equation}\label{calculation of sigi}
		\begin{aligned}
			&\Xi(\cL_{\lav,\mu}')\cong j_{\lav,!*}\left( \bigoplus_{0\le k\le \dim (N_{\O}/N_{\lav})} \O_{\lav,\mu} \lr{-\tfrac12 d_{\lav}} \otimes \bigwedge^k\left(\hGO\times^{\hPla} (N_{\O}/N_{\lav})^*\right)[k]\right);\\
			&\Xi(\cL_{\lav,\mu})\cong i_{\le \lav *}j_{\lav,!*}\left( \bigoplus_{0\le k\le \dim (N_{\O}/N_{\lav})} \O_{\lav,\mu} \lr{-\tfrac12 d_{\lav}} \otimes \bigwedge^k\left(\hGO\times^{\hPla} (N_{\O}/N_{\lav})^*\right)[k]\right).
		\end{aligned}
	\end{equation}
	Here $j_{\lav,!*}$ is the intermediate extension functor for $j_{\lav}:\Gr_{\lav}\hookrightarrow \Gr_{\le \lav}$ in \eqref{intermediate extension for KP}.
\end{proposition}

\begin{proof} 
	There are many symbols in the proof; we refer the reader to the list Figure \ref{notation} for a quick review.
	
    We divide the proof into two steps.
	\begin{enumerate}[label=\textbf{Step \arabic*}]
		\item Recall that $\cL_{\lav,\mu}=i_{\le \lav *}\cL_{\lav,\mu}'$ as in (\ref{def of KPlamu}). So 
		\[\Xi(\cL_{\lav,\mu})=\sigi(i_{\le \lav *}\cL_{\lav,\mu}')\cong i_{\le \lav *}\sigi\cL_{\lav,\mu}'=i_{\le \lav *}\Xi(\cL_{\lav,\mu}').\]
		The goal of the first step is to show $\Xi(\cL_{\lav,\mu}')$ lies in the full subcategory $\KPcoh^{\hGO}(\overline{\Gr}_{\lav})_{!*}$ (see \eqref{eq:KP_!*} for the definition). To do this, we use Theorem \ref{theorem:simple objects on Gr} (2), and show that $\Xi(\cL_{\lav,\mu}')$ meets the criterion \eqref{Hom=0 criterion}.

		By definition, $\cL_{\lav,\mu}'$ is the socle of $\HK^0(j_{\lav*}\cl_* {p^{\cl*}}(\Olamu\lr{-\tfrac{1}{2}d_{\lav}}))\in \KPindcoh^{\hGO}(\cR_{\le \lav})$ (see \eqref{def of KPlamu'}). So $i_{1*}(\cL'_{\lav,\mu})$ is a subobject of
		\begin{equation*}
			i_{1*}\left(\HK^0(j_{\lav*}\cl_* {p^{\cl*}}(\Olamu\lr{-\tfrac{1}{2}d_{\lav}}))\right)=\HK^0(i_{1*}j_{\lav*}\cl_* {p^{\cl*}}(\Olamu\lr{-\tfrac{1}{2}d_{\lav}})).
		\end{equation*}
		The composition $\cR_{\lav}\xrightarrow{j_{\lav}} \cR_{\le\lav}\xrightarrow{i_1}\overline{\Gr}_{\lav}\times N_{\O}$ is isomorphic to the composition $\cR_{\lav}\xrightarrow{i_1}\Gr_{ \lav}\times N_{\O} \xrightarrow{{j}_{\lav}} \overline{\Gr}_{\lav} \times N_{\O}$. So the right hand side of the above equation becomes
		\begin{equation*}
			\HK^0({j}_{\lav*}i_{1*}\cl_* {p^{\cl*}}(\Olamu\lr{-\tfrac{1}{2}d_{\lav}})).
		\end{equation*}
		Let $\mathcal{P}$ denote $i_{1*}\cl_* {p^{\cl*}}(\Olamu\lr{-\tfrac{1}{2}d_{\lav}})\in \KPcoh^{\hGO}(\Gr_{\lav}\times N_{\O})$. We further apply $\sigma_1^*$ to $i_{1*}(\cL'_{\lav,\mu})$, and deduce $\Xi(\cL_{\lav,\mu}')=\sigi(\cL_{\lav,\mu}')$ is a subobject of $$\sigma_1^*\HK^0({j}_{\lav*}\mathcal{P})\cong \HK^0(\sigma_1^*{j}_{\lav*}\mathcal{P}).$$.
        \begin{lemma}
        	We have an isomorphism $\sigma_1^*{j}_{\lav*}\mathcal{P}\cong {j}_{\lav*}\sigma_1^*\mathcal{P}$ in $\IndCoh^{\hGO}(\Gr_{\le \lav})$.
        \end{lemma}
    	\begin{proof}

        According to \cite[Lem. 4.14]{CW23}, $\mathcal{P}\cong p_1^{(k)*}\mathcal{P}_k$ for some $k\in \mathbb{N}$, where $p_1^{(k)}:\Gr_{\lav}\times N_{\O}\to \Gr_{\lav}\times N_{\O/t^k}$ is the projection and $\mathcal{P}_k\in \Coh^{\hGO}(\Gr_{\lav}\times N_{\O/t^k})$; here $N_{\O/t^k}$ is a finite dimensional vector space defined as in \eqref{eq:Y_O/t^n}; we also denote the projection $\Gr_{\le\lav}\times N_{\O}\to \Gr_{\le\lav}\times N_{\O/t^k}$ by $p_1^{(k)}$. Then we have
        \[ \sigma_1^*{j}_{\lav*}\mathcal{P}\cong \sigma_1^*{j}_{\lav*}p_1^{(k)*}\mathcal{P}_k {\cong} \sigma_1^*p_1^{(k)*}{j}_{\lav*}\mathcal{P}_k.\]
        The last isomorphism comes from the base change isomorphism; it holds because $j_{\lav}$ is of finite cohomological dimension, and $p_1^{(k)}$ is flat; see Section \ref{subsubsec:indcoh}.

        Let $\sigma_1^{(k)}$ denote the composition $\Gr_{\le \lav}\xhookrightarrow{\sigma_1}\Gr_{\le \lav}\times N_{\O}\xrightarrow{p_1^{(k)}}\Gr_{\le \lav}\times N_{\O/t^k}$, which is of finite Tor-dimension; we also write its restriction to $\Gr_{\lav}$ by $\sigma_1^{(k)}$. Then we continue the computation:
        \[\sigma_1^*p_1^{(k)*}{j}_{\lav*}\mathcal{P}_k\cong  \sigma_1^{(k)*}{j}_{\lav*}\mathcal{P}_k
        {\cong} {j}_{\lav*}\sigma_1^{(k)*}\mathcal{P}_k\cong {j}_{\lav*}\sigma_1^*p_1^{(k)*}\mathcal{P}_k
        \cong
        {j}_{\lav*}\sigma_1^*\mathcal{P},\]
        where the second isomorphism comes from the base change isomorphism; it holds because $j_{\lav}$ is of finite cohomological dimension and $\sigma_1^{(k)}$ is of finite Tor-dimension.
    	\end{proof}

        In conclusion, $\Xi({\cL_{\lav,\mu}'})$ is a subobject of $\HK^0({j}_{\lav*}\sigma_1^*\mathcal{P})$ in $\KPindcoh^{\hGO}(\Gr_{\le \lav})$. Then for any $\cG\in \KPcoh^{\hGO}(\Gr_{\le \lav})$ such that $\Supp \cG\cap \Gr_{\lav}=\emptyset$, we have 
        \begin{equation}\label{eq: 11}
            \Hom_{\KP}(\cG,\Xi({\cL_{\lav,\mu}'}))\subset \Hom_{\KP}(\cG,\HK^0({j}_{\lav*}\sigma_1^*\mathcal{P}))=\footnotemark\Hom_{\KP}({j}_{\lav}^*\cG,\sigma_1^*\mathcal{P})=0.
        \end{equation}\footnotetext{Here, $(j_{\lav}^*(-),\HK^0(j_{\lav*}(-)))$ is an adjoint pair for Koszul perverse sheaves because $(j_{\lav}^*(-),j_{\lav*}(-))$ is an ajoint pair and $j_{\lav}^*$ is Koszul perverse t-exact.}
        Thus the second condition in the criterion \eqref{Hom=0 criterion} is verified.
		
		Next, we use the Grothendieck-Serre dual functors to verify the first condition in the criterion \eqref{Hom=0 criterion}. Recall that we define $\D_{\cR_{\le \lav}}(-)=\cHom(-,i_1^!p_1^*(\omega_{\Gr_{\le\lav}}))$ as in \eqref{eq:D_R}, and we have $\D(\cL_{\lav,\mu})\cong i_{\le \lav}\D_{\cR_{\le \lav}}(\cL_{\lav,\mu}')$ as shown in \eqref{eq: D(L_{lav})}\footnote{The setting there assumed $N=\g$, but the computation in \eqref{eq: D(L_{lav})} is valid for general $N$.}. Since $\D(\cL_{\lav,\mu})$ is a simple object, we know $\D_{\cR_{\le \lav}}(\cL_{\lav,\mu}')$ is also a simple object. Apply \eqref{eq: 11} to this simple object, we deduce that for any $\cG\in \KPcoh^{\hGO}(\Gr_{\le \lav})$ such that $\Supp \cG\cap \Gr_{\lav}=\emptyset$,  
        $$\Hom_{\KP}(\D_{\Gr_{\le \lav}}(\cG),\Xi\circ{\D_{\cR_{\le \lav}}(\cL_{\lav,\mu}')})=0.$$
        
        Notice that
        \begin{align*}
        	\Xi\circ \D_{\cR_{\le \lav}}(\cL_{\lav,\mu}')&=\sigi\cHom(\cL_{\lav,\mu}',i_1^!p_1^*(\omega_{\Gr_{\le\lav}}))\\
        	&\overset{\eqref{eq:f_*,f^!}}{\cong} \sigma_1^*\cHom(i_{1*}\cL_{\lav,\mu}',p_1^*(\omega_{\Gr_{\le \lav}}))\\
        	&\overset{\eqref{eq:h^*cHom}}{\cong}\cHom(\sigma_1^*{i_1}_*\cL_{\lav,\mu}',\sigma_1^*p_1^*(\omega_{\Gr_{\le \lav}}))\\
        	&\overset{p_1\circ\sigma_1=\id}{\cong} \cHom(\sigma_1^*{i_1}_*\cL_{\lav,\mu}',\omega_{\Gr_{\le \lav}})=\D_{\Gr_{\le \lav}}\circ \Xi(\cL_{\lav,\mu}').
        \end{align*}
        Then 
        \begin{align*}
        	\Hom_{\KP}(\Xi(\cL_{\lav,\mu}'),\cG)&\cong\Hom_{\KP}(\D_{\Gr_{\le \lav}}(\cG),\D_{\Gr_{\le \lav}}\circ\Xi(\cL_{\lav,\mu}'))\\
        	&\cong \Hom_{\KP}(\D_{\Gr_{\le \lav}}(\cG),\Xi\circ{\D_{\cR_{\le \lav}}(\cL_{\lav,\mu}')})=0.
        \end{align*}
        Thus the first condition in the criterion \eqref{Hom=0 criterion} is verified. In conclusion, $\Xi(\cL_{\lav,\mu}')\in \KPcoh^{\hGO}(\Gr_{\le \lav})_{!*}$

		\item After completing \textbf{Step 1}, it suffices to compute $j_{\lav}^*\Xi(\cL_{\lav,\mu}')$. 
		
		Firstly, 
		\begin{align*}
			j_{\lav}^*\Xi(\cL_{\lav,\mu}')&\cong \Xi\circ j_{\lav}^*(\cL_{\lav,\mu}')\\
			&\cong \Xi (\cl_*{p^{\cl*}}(\O_{\lav,\mu}\lr{-\tfrac12 d_{\lav}}))\\
			&=\sigma_1^* i_{1*}^{\cl}{p^{\cl*}}(\O_{\lav,\mu}\lr{-\tfrac12 d_{\lav}}).
		\end{align*} 
		We do similar identifications as diagram \eqref{identification of cRlav}:
		% https://q.uiver.app/#q=WzAsOCxbMSwxLCJHX3tcXE99XFx0aW1lc157XFxoUGxhfU5fe1xcbGF2fSJdLFsyLDEsIkdfe1xcT31cXHRpbWVzXntcXGhQbGF9Tl97XFxPfSJdLFszLDEsIkdfe1xcT30vXFxoUGxhIl0sWzEsMCwiXFxjUl97XFxsYXZ9Il0sWzIsMCwiXFxHcl97XFxsYW1iZGFee1xcdmVlfX1cXHRpbWVzXFxnX3tcXE99Il0sWzMsMCwiXFxHcl97XFxsYW1iZGFee1xcdmVlfX0iXSxbMCwxLCJHX3tcXE99L1xcaFBsYSJdLFswLDAsIlxcR3Jfe1xcbGFtYmRhXntcXHZlZX19Il0sWzAsMSwiaV8xXntcXGNsfSJdLFsyLDEsIlxcc2lnbWFfMSIsMl0sWzMsNCwiaV8xXntcXGNsfSJdLFs1LDQsIlxcc2lnbWFfMSIsMl0sWzEsNCwiXFx0ZXh0e1xcZGluZ3sxNzN9fSJdLFsyLDUsIlxcdGV4dHtcXGRpbmd7MTc0fX0iXSxbMCwzLCJcXHRleHR7XFxkaW5nezE3Mn19Il0sWzAsMywiXFxjb25nIiwyXSxbMSw0LCJcXGNvbmciLDJdLFsyLDUsIlxcY29uZyIsMl0sWzAsNiwicF8xXntcXGNsfSIsMl0sWzMsNywicF8xXntcXGNsfSIsMl0sWzYsNywiXFxjb25nIiwyXSxbNiw3LCJcXGRpbmd7MTc0fSJdXQ==
		\[\begin{tikzcd}
			{\Gr_{\lambda^{\vee}}} & {\cR_{\lav}^{\cl}} & {\Gr_{\lambda^{\vee}}\times N_{\O}} & {\Gr_{\lambda^{\vee}}} \\
			{\hGO/\hPla} & {\hGO\times^{\hPla}N_{\lav}} & {\hGO\times^{\hPla}N_{\O}} & {\hGO/\hPla}
			\arrow["{p^{\cl}}"', from=1-2, to=1-1]
			\arrow["{i_1^{\cl}}", from=1-2, to=1-3]
			\arrow["{\sigma_1}"', from=1-4, to=1-3]
			\arrow["\cong"', from=2-1, to=1-1]
			\arrow["{\ding{174}}", from=2-1, to=1-1]
			\arrow["{\text{\ding{172}}}", from=2-2, to=1-2]
			\arrow["\cong"', from=2-2, to=1-2]
			\arrow["{p^{\cl}}"', from=2-2, to=2-1]
			\arrow["{i_1^{\cl}}", from=2-2, to=2-3]
			\arrow["{\text{\ding{173}}}", from=2-3, to=1-3]
			\arrow["\cong"', from=2-3, to=1-3]
			\arrow["{\text{\ding{174}}}", from=2-4, to=1-4]
			\arrow["\cong"', from=2-4, to=1-4]
			\arrow["{\sigma_1}"', from=2-4, to=2-3]
		\end{tikzcd}\]
		where the three column maps are 
		\begin{equation*}
			\begin{aligned}
				&\text{\ding{172}}:[g,x]\mapsto [gt^{\lav},{t^{-{\lav}}}x],\\
				&\text{\ding{173}}:[g,x]\mapsto ([gt^{{\lav}}],gx),\\
				&\text{\ding{174}}:[g]\mapsto [gt^{{\lav}}].
			\end{aligned}
		\end{equation*}
		
		By the standard Koszul resolution argument, we have:
		\begin{equation*}
			\begin{aligned}
				\sigma_1^* i_{1*}^{\cl}{p^{\cl*}}(\O_{\lav,\mu}\lr{-\tfrac12 d_{\lav}}) \cong \bigoplus_{k} \O_{\lav,\mu} \lr{-\tfrac12 d_{\lav}} \otimes \bigwedge^k\left(\hGO\times^{\hPla} (N_{\O}/N_{\lav})^*\right)[k].
			\end{aligned}
		\end{equation*}
		Combining \textbf{Step 1}, we complete the proof of this proposition.	
	\end{enumerate}
\end{proof}

We end this subsection with two corollaries of the above Proposition \ref{prop: sigi of L}.
\begin{corollary}\label{injectivity of sigi}
	The map between $K$-groups
	\begin{equation}
		\Xi:K^{\hGO}(\cR)\to K^{\hGO}(\Gr)
	\end{equation}
	is injective.\footnote{We remark that the following proof essentially only uses the easier part \textbf{Step 2} in the proof of Proposition \ref{prop: sigi of L}.
		
	We also remark that this corollary also holds for Borel-Moore homology version Coulomb branch; both of these two versions can also be proved by the localization theorem.}
\end{corollary}
\begin{proof}
	The $K$-group $K^{\hGO}(\cR)$ admits a filtration with terms  $\left\{K^{\hGO}(\cR_{\le \lav})\right\}_{\lav\in \Pv_+}.$ The associated graded of this filtration is identified with 
	$\bigoplus_{\lav\in\Pv_+}K^{\hGO}(\cR_{ \lav}).$ 
	Similarly, $K^{\hGO}(\Gr)$ admits a filtration with terms
	$\left\{K^{\hGO}(\Gr_{\le \lav})\right\}_{\lav\in \Pv_+},$ 
	whose associated graded is described analogously. 
	
	By definition, $\Xi=\sigi$ preserves the filtrations, thus we need to show that this map is injective for each associated graded pieces
	\[\Xi=\sigi:K^{\hGO}(\cR_{\lav})\to K^{\hGO}(\Gr_{\lav}).\]

	After the following identifications
	\begin{align*}
		& K^{\hGO}(\Gr_{ \lav})\cong K^{\hPla}(\pt)\\
		& K^{\hGO}(\cR_{\lav})=K^{\hGO}(\cR^{\cl}_{\lav}) \xrightarrow[\cong]{\sigma^*} K^{\hGO}(\Gr_{\lav})\cong K^{\hPla}(\pt),
	\end{align*}
	Proposition \ref{prop: sigi of L} implies that the map $\Xi$ is identified with the map $K^{\hPla}(\pt)\to K^{\hPla}(\pt)$ given by tensoring with $\sum_k (-1)^k [\bigwedge^k(N_{\O}/N_{\lav})^*]$, which is a nonzero element. Since $K^{\hPla}(\pt)$ is an integral domain with respect to tensor product as its multiplication, the map $\Xi$ is injective.
\end{proof}

Another corollary is for the special case $N=\g$.
\begin{corollary}\label{sigi of KPla}
	Let $N$ be the adjoint representation $\g$. Then 
	\begin{equation}\label{eq: sigi of KPla}
		\Xi(\cL_{\lav})\cong \bigoplus_{0\le k\le d_{\lav}}i_{\le \lav*}j_{\lav,!*}\left(\Omega_{\Gr_{{\lav}}}^k\lr{k-\tfrac12 d_{\lav}}[k]\right) \overset{\eqref{eq:!*relations}}{\cong} \bigoplus_{0\le k\le d_{\lav}} i_{\le \lav*}\cIC(\Omega_{\Gr_{{\lav}}}^k)\lr{k-\tfrac12 d_{\lav}}[k].
	\end{equation} 
	Here $\Omega_{\Gr_{\lav}}^k=\bigwedge^k\Omega_{\Gr_{\lav}}^1$ is the $k$-th differential sheaf on $\Gr_{\lav}$ relative to $\pt=\Spec \C$ equipped with the natural $\hGO$-equivariance.
\end{corollary}

\begin{proof}
	For $N=\g$, we can identify $N_{\lav}=\pla$, and \eqref{calculation of sigi} becomes
	\begin{equation}\label{eq: 111}
		\Xi(\cL_{\lav})\cong\bigoplus_{0\le k\le d_{\lav}}i_{\le \lav*} j_{\lav,!*}\left(\bigwedge^k(\hGO\times^{\hPla} (\g_{\O}/\pla)^*)\lr{-\tfrac12 d_{\lav}}[k]\right).
	\end{equation}
	We can identify $\hGO\times^{\hPla} (\g_{\O}/\pla)^*$ with the cotangent bundle of $\hGO/\hPla$, but with the $\Gmdil$-acting by weight $-1$ scaling. Then we can identify $\hGO\times^{\hPla} (\g_{\O}/\pla)^*$ with $\Omega_{\Gr_{\lav}}^1\lr{1}$. Substituting it into the above equation \eqref{eq: 111} completes the proof of this corollary.
\end{proof}

\subsection{Computation of \texorpdfstring{$\Xi[\cL_{\lav}]\tbox \Xi[\cL_{\muv}]$}{Ξ[L] tbox Ξ[L]}} \label{subsec: twisted prod is IC}
From now on, we assume $N=\g$.

In the previous subsection, we identify $\Xi(\cL_{\lav})$ with $\bigoplus_{k}i_{\le \lav*}\cIC(\Omega_{\Gr_{{\lav}}}^k)\lr{k-\tfrac12 d_{\lav}}[k]$. In this subsection, we will give a description of the twisted product $\Xi[\cL_{\lav}']\tbox \Xi[\cL_{\muv}']$ in Proposition \ref{prop:calculation of tbox}.

We write $\Gr_{\lav,\muv}:=\Gr_{\lav}\ttimes \Gr_{\muv}\subset \Gr\ttimes \Gr$; its closure in $\Gr\ttimes \Gr$ is $\overline{\Gr}_{\lav}\ttimes \overline{\Gr}_{\muv}$, which we write as $\overline{\Gr}_{\lav,\muv}$. We also write $d_{\lav,\muv}:=\dim\Gr_{\lav,\muv}= d_{\lav}+d_{\muv}$.

\begin{proposition}\label{prop:calculation of tbox}	
    In $K^{\hGO}(\overline{\Gr}_{\lav,\muv})$, we have the following identification
    \begin{equation*}
    	\Xi[\cL_{\lav}']\tbox \Xi[\cL_{\muv}']=\sum_{0\le m\le d_{\lav,\muv}}\left[\cIC(\Omega_{{\Gr}_{\lav,\muv}}^m)\lr{m-\tfrac12 d_{\lav,\muv}}[m]\right].
    \end{equation*}
\end{proposition}

\begin{proof}
    Using the identification \eqref{eq: sigi of KPla}, we have
    \begin{equation*}
        \begin{aligned}
            \Xi[\cL_{\lav}']\tbox \Xi[\cL_{\muv}']=\sum_{k,l}\left[\cIC(\Omega_{\Gr_{{\lav}}}^k)\tbox \cIC(\Omega_{\Gr_{\muv}}^l)\lr{k+l-\tfrac12 d_{\lav,\muv}}[k+l]\right].
        \end{aligned}
    \end{equation*}
	Thus, it remains to show
	\begin{equation}\label{eq:ICbox}
		\left[\cIC(\Omega_{\Gr_{\lav,\muv}}^m)\right]=\sum_{k+l=m}\left[\cIC(\Omega_{\Gr_{\lav}}^k)\tbox\cIC(\Omega_{\Gr_{\muv}}^l)\right].
	\end{equation}

	To simplify notations, in this proof we write $H=\hGO$, $X=\overline{\Gr}_{\lav}$, $Y=\overline{\Gr}_{\muv}$. We can restrict the $H$-torsor $\hGK\to \Gr$ to $X=\overline{\Gr}_{\lav}\hookrightarrow \Gr$, and denote the resulting $H$-torsor by $P(:=\hGK\times_{\Gr}\overline{\Gr}_{\lav})\to X$, which is Zariski locally trivial (Lemma \ref{lem:Zar_triv}). By definition, the twisted product $X\ttimes Y=\overline{\Gr}_{\lav}\ttimes \overline{\Gr}_{\muv}$ is the associated bundle $P\times^H Y$. The open subsets $\Gr_{\lav}\subset \overline{\Gr}_{\lav}$, $\Gr_{\muv}\subset \overline{\Gr}_{\muv}$ are denoted by $\Xo$, $\Yo$ respectively. We write $j_{X}:\Xo\hookrightarrow X$, $j_{Y}:\Yo\hookrightarrow Y$, $j:\Xo\ttimes\Yo\hookrightarrow X\ttimes Y$ for the open embeddings, and write $\rp:X\ttimes Y\to X$, $\rq:X\ttimes Y\to [H\backslash Y]$, $\rpc:\Xo\ttimes \Yo\to \Xo$, $\rqc:\Xo\ttimes \Yo\to [H\backslash \Yo]$, for the natural projections (see Section \ref{subsubsec:twisted_prod}).

\begin{enumerate}[label=\textbf{Step \arabic*}]
	\item To illustrate the main idea of \eqref{eq:ICbox}, we first explain the equality when restricted to the open subscheme $\Xo\ttimes \Yo \subset X\ttimes Y$. 
	
	For the fiber bundle $\Xo\ttimes \Yo\to \Xo$, there is a short exact sequence in $\Coh^H(\Xo\ttimes \Yo)^{\heartsuit}$:
	\begin{equation}\label{eq:SESo}
		0\to \rpc^*\Omega_{\Xo}^1\to \Omega^1_{\Xo \ttimes \Yo}\to \Omega^1_{\Xo\ttimes \Yo/\Xo}\to 0.
	\end{equation}
	Note that
	\begin{equation}\label{eq:step11}
		\rpc^*\Omega_{\Xo}^1=\Omega_{\Xo}^1\tbox \O_{\Yo}.
	\end{equation}
	According to the Cartesian diagram \eqref{eq:diagram_ttimes}, we also have
	\begin{equation}\label{eq:step12}
		\Omega^1_{\Xo\ttimes \Yo/\Xo}\cong \rqc^* \Omega^1_{\Yo}= \O_{\Xo}\tbox \Omega^1_{\Yo}.
	\end{equation}
	The short exact sequence \eqref{eq:SESo} produces a filtration for each $\Omega^m_{\Xo \ttimes \Yo}=\bigwedge^m \Omega^1_{\Xo \ttimes \Yo}$ with associated graded pieces
	\[\bigwedge^k\left(\Omega_{\Xo}^1\tbox \O_{\Yo}\right) \otimes \bigwedge^l\left(\O_{\Xo}\tbox \Omega^1_{\Yo} \right)=\Omega_{\Xo}^k\tbox \Omega_{\Yo}^l,\text{ for }k+l=m.\]
	Thus we have
	\[\left[\Omega_{\Xo\ttimes \Yo}^m\right]=\sum_{k+l=m}\left[\Omega_{\Xo}^k\tbox\Omega_{\Yo}^l\right]\]
	in the $K$-group $K^{H}(\Xo\ttimes\Yo)$.

    \item We show that $$\cIC(\Omega_{\Xo}^k)\tbox \cIC(\Omega_{\Yo}^l)\cong \cIC(\Omega_{\Xo}^k\tbox \Omega_{\Yo}^l)$$ using Proposition \ref{prop:cIC_and_support}. 
    
    Note that similar to the exterior product, the twisted product also satisfies the following K\"unneth formula:
    \[\cH^{r}\left(\cIC(\Omega_{\Xo}^k)\tbox \cIC(\Omega_{\Yo}^l)\right)\cong \bigoplus_{i+j=r}\cH^i\left(\cIC(\Omega_{\Xo}^k)\right)\tbox\cH^j\left(\cIC(\Omega_{\Yo}^l)\right).\]
    So $$\Supp \cH^{r}\left(\cIC(\Omega_{\Xo}^k)\tbox \cIC(\Omega_{\Yo}^l)\right)=\bigcup_{i+j=r}\Supp\cH^i\left(\cIC(\Omega_{\Xo}^k)\right)\ttimes\Supp\cH^j\left(\cIC(\Omega_{\Yo}^l)\right).$$
    According to \eqref{cIC in our case}, $\cIC(\Omega_{\Xo}^k)$ and $\cIC(\Omega_{\Yo}^l)$ satisfy the support conditions \eqref{eq: supp condition}. Let $i+j=r\ge1$. If $i\ge 1$ and $ j\ge 1$, we have
    $$\codim \Big(\Supp\cH^i\left(\cIC(\Omega_{\Xo}^k)\right)\ttimes\Supp\cH^j\left(\cIC(\Omega_{\Yo}^l)\right)\Big)\ge 2i+2+2j+2> 2r+2.$$
    If one of $i,j$ is $0$, without loss of generality, let $i=0$, $j=r\ge 1$; then
    $$\codim \Big(\Supp\cH^0\left(\cIC(\Omega_{\Xo}^k)\right)\ttimes\Supp\cH^r\left(\cIC(\Omega_{\Yo}^l)\right)\Big)\ge \codim \Supp\cH^r\left(\cIC(\Omega_{\Yo}^l)\right)\ge 2r+2.$$
    So $\cIC(\Omega_{\Xo}^k)\tbox \cIC(\Omega_{\Yo}^l)$ satisfies the first support condition in \eqref{eq: supp condition}. Since $$\omega_{X\ttimes Y}\cong \rp^!\omega_X \cong \rp^*\omega_X\otimes \omega_{X\ttimes Y/X}\cong \rp^*\omega_X\otimes \rq^*\omega_{Y}= \omega_{X}\tbox \omega_{Y},$$ we have $$\cD_{X\ttimes Y}\left(\cIC(\Omega_{\Xo}^k)\tbox \cIC(\Omega_{\Yo}^l)\right)\cong \cD_{X}\left(\cIC(\Omega_{\Xo}^k)\right)\tbox\cD_{Y}\left( \cIC(\Omega_{\Yo}^l)\right);$$ here $\cD$ is the shifted Grothendieck-Serre dual functor as in \eqref{eq:cD}. Then applying the arguments above again, we deduce $\cIC(\Omega_{\Xo}^k)\tbox \cIC(\Omega_{\Yo}^l)$ satisfies the second support condition in \eqref{eq: supp condition}. Combined with the equality $${j}_{\lav,\muv}^*(\cIC(\Omega_{\Xo}^k)\tbox \cIC(\Omega_{\Yo}^l))\cong \Omega_{\Xo}^k\tbox \Omega_{\Yo}^l,$$ we conclude that $$\cIC(\Omega_{\Xo}^k)\tbox \cIC(\Omega_{\Yo}^l)\cong \cIC(\Omega_{\Xo}^k\tbox \Omega_{\Yo}^l).$$
	
	\item 	We show that there is a fiber-cofiber sequence
	\begin{equation}\label{eq:cofiber_sequence}
		\cIC(\Omega^1_{\Xo}\tbox \O_{\Yo})\to \cIC(\Omega^1_{\Xo}\tbox \Omega^1_{\Yo})\to \cIC(\O_{\Xo}\tbox \Omega^1_{\Yo})
	\end{equation}
	in the stable $\infty$-category $\Coh^{H}(X\ttimes Y)$; in other words, $\cIC(\Omega^1_{\Xo}\tbox \Omega^1_{\Yo})$ admits a filtration in the stable $\infty$-category with associated graded pieces $\cIC(\Omega^1_{\Xo}\tbox \O_{\Yo})$ and $\cIC(\O_{\Xo}\tbox \Omega^1_{\Yo})$. Then in the $K$-group $K^{H}(X\ttimes Y)$, we obtain the identity
	\[[\cIC(\Omega^1_{\Xo}\tbox \Omega^1_{\Yo})]=[\cIC(\Omega^1_{\Xo}\tbox \O_{\Yo})]+ [\cIC(\O_{\Xo}\tbox \Omega^1_{\Yo})].\]
	
	Indeed, for the fiber bundle $X\ttimes Y\to X$, we have a short exact sequence in $\Coh^H(X\ttimes Y)^{\heartsuit}$:
	\[	0\to \rp^*\Omega_{X}^1\to \Omega^1_{X \ttimes Y}\to \Omega^1_{X\ttimes Y/X}\to 0.\]
	After similar computations as in \textbf{Step} 1 \eqref{eq:step11} \eqref{eq:step12}, this short exact sequence becomes
	\begin{equation}\label{eq:SES}
			0\to \Omega_{X}^1\tbox \O_{Y}\to \Omega^1_{X \ttimes Y}\to \O_X\tbox \Omega^1_{Y}\to 0.
	\end{equation}
	We write $\Omega_X^{[k]}:=(\Omega_{X}^k)^{\vee\vee}\in\Coh^H(X)^{\heartsuit}$, where $(-)^{\vee}:=\cH^0(\cHom(-))$. This is the reflexive hull of $\Omega_X^{k}$ as in \cite[\href{https://stacks.math.columbia.edu/tag/0AVT}{Tag 0AVT}]{stacks-project}; we have $ \Omega_X^{[k]}=\cH^0(j_{X*}(\Omega_{\Xo}^k))$ by \cite[\href{https://stacks.math.columbia.edu/tag/0AY6}{Tag 0AY6}]{stacks-project}. Similar notations are used for other K\"ahler differential sheaves. 
	
	After forgetting the $H$-equivariance, the short exact sequence \eqref{eq:SES} splits Zariski locally since the fiber bundle $X\ttimes Y\to X$ is Zariski locally trivial. More explicitly, we choose a Zariski open cover $\phi:U\to X$, such that $\phi^*(X\ttimes Y)=U\ttimes Y\cong U\times Y$ over $U$. The pullback of \eqref{eq:SES} along $\phi\ttimes \id$ becomes
	\[0\to \Omega_{U}^1\tbox \O_{Y}\to \Omega^1_{U \ttimes Y}\to \O_U\tbox \Omega^1_{Y}\to 0,\]
	which splits by a choice of trivialization $U\ttimes Y\cong U\times Y$. Applying the additive functor $(-)^{\vee\vee}$, we obtain a split short exact sequences for reflexive hulls:
	\begin{equation}\label{eq:SESICU}
		0\to \Omega_{U}^{[1]}\tbox \O_{Y}\to \Omega^{[1]}_{U \ttimes Y}\to \O_U\tbox \Omega^{[1]}_{Y}\to 0;
	\end{equation}
	this sequence is the pullback along $\phi\ttimes \id$ of the following sequence whose terms are reflexive hulls of sheaves in \eqref{eq:SES}:
	\begin{equation}\label{eq:SESrefl}
		0\to \Omega_{X}^{[1]}\tbox \O_{Y}\to \Omega^{[1]}_{X \ttimes Y}\to \O_X\tbox \Omega^{[1]}_{Y}\to 0.
	\end{equation}
	In conclusion, the sequence \eqref{eq:SESrefl} is a short exact sequence, which splits Zariski locally.
	
	Now we apply the functor \eqref{eq:cIC for Refl}
	\[\wt{\cIC}:\mathrm{Refl}^H(X\ttimes Y)\to \Coh^H(X\ttimes Y)\]
	to \eqref{eq:SESrefl}, and obtain a sequence of morphisms:
	\begin{equation}\label{eq:SESIC}
		\wt{\cIC}(\Omega_{X}^{[1]}\tbox \O_{Y})\xrightarrow{} \wt{\cIC}(\Omega^{[1]}_{U \ttimes Y})\xrightarrow{} \wt{\cIC}(\O_U\tbox \Omega^{[1]}_{Y}),
	\end{equation}
	where each term is the same as \eqref{eq:cofiber_sequence} respectively by Definition \ref{def:cIC}. To show this sequence can be made into a fiber-cofiber sequence, we only need to show its pullback along the faithfully flat cover $X\ttimes Y\to [H\backslash X\ttimes Y]$ (i.e. forgetting the $H$-equivariance) can be made into a fiber-cofiber sequence using the fpqc descent for $\QCoh$. We further pullback the sequence to the Zariski cover ${\phi}\ttimes \id_Y:U\ttimes Y\to X\ttimes Y$ and show that it can be made into a fiber-cofiber sequence. Note that $\wt{\cIC}$-functors commute with pullback along $\wt{\phi}\ttimes \id_Y$:
	\[({\phi}\ttimes \id_Y)^*\circ \wt{\cIC}\cong \wt{\cIC}\circ ({\phi}\ttimes \id_Y)^*\]
	since our chosen dualizing complexes satisfy $\omega_{U\ttimes Y}=({\phi}\ttimes \id_Y)^*\omega_{X\ttimes Y} $. Thus the pullback of the sequence \eqref{eq:SESIC} to $U\ttimes Y$ becomes
	\[\wt{\cIC}(\Omega_U^{[1]}\tbox \O_{Y})\to \wt{\cIC}(\Omega_{U\ttimes Y}^{[1]})\to \wt{\cIC}(\O_U\tbox\Omega_{Y}^{[1]}),\]
	which can be made into a split fiber-cofiber sequence, since \eqref{eq:SESICU} is a split short exact sequence. In conclusion, \eqref{eq:SESIC} can be made into a fiber-cofiber sequence.
	
	 \item Using the Zariski locally split short exact sequence \eqref{eq:SES}, we obtain a filtration for each $\Omega_{X\ttimes Y}^{m}=\bigwedge^m\Omega_{X\ttimes Y}^1$ with associated graded pieces
	 \[\bigwedge^k\left(\Omega_{X}^1\tbox \O_{Y}\right) \otimes \bigwedge^l\left(\O_{X}\tbox \Omega^1_{Y} \right)=\Omega_{X}^k\tbox \Omega_{Y}^l,\text{ for }k+l=m.\]
	 After the same arguments as in \textbf{Step 2}, the filtration for $\Omega_{X\ttimes Y}^{m}$ produces a filtration for the reflexive hull $\Omega_{X\ttimes Y}^{[m]}$ in the abelian category $\Coh^H(X\ttimes Y)^{\heartsuit}$ with associated graded pieces
	 \[\Omega_{X}^{[k]}\tbox \Omega_{Y}^{[l]},\text{ for }k+l=m,\]
	 and it further produces a filtration for $\wt{\cIC}(\Omega_{X\ttimes Y}^{[m]})$ in the stable $\infty$-category $\Coh^H(X\ttimes Y)$ with associated graded pieces
	 \[\wt{\cIC}(\Omega_{X}^{[k]})\tbox \wt{\cIC}(\Omega_{Y}^{[l]}),\text{ for }k+l=m.\]
	 Thus, we obtain the equation \eqref{eq:ICbox} in the $K$-group $K^{H}(X\ttimes Y)$.

\end{enumerate}
\end{proof}

\subsection{Tools coming from Hodge module theory}\label{subsec: Tools coming from Hodge module}
After the proposition \ref{prop:calculation of tbox}, we just need to compute the class
\[\left[{m}_{\Gr,*}\left(\cIC(\Omega_{\Gr_{ \lav,\muv}}^k)\right)\right]\]
in the $K$-group $K^{\hGO}(\Gr)$ for each $k\in \Z, 0\le k\le d_{\lav,\muv}$, where ${m}_{\Gr}:\overline{\Gr}_{\lav,\muv}=\overline{\Gr}_{\lav}\ttimes\overline{\Gr}_{\muv}\to \overline{\Gr}_{\lav+\muv}$ is the convolution map; we abbreviate $m_{\Gr}$ as $m$ in the remaining section.

The tools come from Hodge module theory. We briefly review some basic facts on Hodge module theory in the appendix \ref{sec:appendix}. 

The key to connecting coherent $\cIC$-sheaves with Hodge modules lies in a result from \cite{Xin25}. We first review some knowledge on symplectic varieties.
\begin{definition}[\cite{Bea00}]
	A variety $X$ has symplectic singularity if it admits a symplectic form $\omega$ on the regular locus $X_{\mathrm{reg}}$ of $X$, such that for some (equivalently, any) resolution of singularities $\pi:\tilde{X}\to X$, the pullback $\pi^*\omega$ to $\pi^{-1}(X_{\mathrm{reg}})$ extends to a holomorphic 2-form $\tilde{\omega}$ on $\tilde{X}$.
	
	If moreover the form $\tilde{\omega}$ is nondegenerate everywhere, $\pi$ is said to be a symplectic resolution.
\end{definition}

\begin{theorem}[{\cite[Thm. 5.4]{Xin25}}]\label{thm: Xin's result}
	Let $X$ be a variety with symplectic singularity of dimension $2n$, and suppose it admits a \textbf{symplectic resolution}. Let $\ICHg_X$ be the IC-Hodge module on $X$, which is the pure Hodge module of weight $2n$ such that $\Rat(\ICHg_X)=\rmIC_X^{\Q}$. Then the following inequality holds
	\[\codim \Supp \cH^{i}\left(\gr_{-k}\DR(\ICHg_{X})[k- 2n]\right)\ge 2i+2 \text{ \ for any $i\ge 1$}.\]
	Here, $\gr_{-k}\DR(-)$ is the graded de Rham functor reviewed in the appendix \ref{subsec:grDR}.
\end{theorem}

To apply this result, we need to introduce the affine grassmannian slices $\bW^{\muv}_{\lav}$, which capture the local singularities of the (spherical) Schubert variety $\overline{\Gr}_{\lav}$. Let $G$ be a reductive group over $\C$. Consider the negative (polynomial) loop group
\[L_{\pol}^-G:\cCAlg\to \Set, \ R\to G(R[t^{-1}]).\]
It can be viewed as a subgroup of $G_{\K}$. The group $L_{\pol}^-G$ has an evaluation map $ev:L_{\pol}^-G\to G$. The kernel is denoted by $L_{\pol}^{<0}G$. Then the orbit $L_{\pol}^{<0}G\cdot t^{\muv}$ is a transversal slice of $\Gr_{\muv}$ in $\Gr$; see e.g. \cite[Prop. 2.3.9]{Zhu17} for more details.

\begin{definition}
	For two dominant cocharacters $\muv \le \lav$, define the affine grassmannian slice to be
	\begin{equation}\label{eq:slice}
		\bW^{\muv}_{\lav}:=L_{\pol}^{<0}G\cdot t^{\muv}\cap \overline{\Gr}_{\lav}.
	\end{equation}
	This is a transversal slice of $\Gr_{\muv}$ in $\overline{\Gr}_{\lav}$, which means there is an open neighborhood of $t^{\muv}$ in $\overline{\Gr}_{\lav}$ isomorphic to $\bW^{\muv}_{\lav}\times \Gr_{\muv}$.
\end{definition}

\begin{proposition}[{\cite[Thm. 2.7]{KWWY}}]
	Let $G$ be a connected reductive group. 
	\begin{enumerate}
		\item All the affine grassmannian slices $\bW_{\lav}^{\muv}$ have symplectic singularities;
		\item \textbf{If $G$ is of type A}, then all the affine grassmannian slices $\bW_{\lav}^{\muv}$ have symplectic resolutions.
	\end{enumerate}
\end{proposition} 

Now we can apply Theorem \ref{thm: Xin's result} to $\overline{\Gr}_{\lav}$. We will abbreviate $\ICHg_{\overline{\Gr}_{\lav}}$ as $\ICHg_{\lav}$ and $\ICHg_{\overline{\Gr}_{\lav,\muv}}$ as $\ICHg_{\lav,\muv}$ for simplicity.

\begin{proposition}\label{2}
	Let $\lav,\muv\in \Pv_+$. The following inequality holds:
	\begin{equation}\label{1}
		\begin{aligned}
				\codim \Supp \cH^{i}\left(\gr_{-k}\DR(\ICHg_{\lav})[k- d_{\lav}]\right)&\ge 2i+2 \text{ \ for any $i\ge 1$},\\
					\codim \Supp \cH^{i}\left(\gr_{-k}\DR(\ICHg_{\lav,\muv})[k- d_{\lav}]\right)&\ge 2i+2 \text{ \ for any $i\ge 1$}.
		\end{aligned}
	\end{equation}
\end{proposition}
\begin{proof}
	We analyze the supports of these modules locally. 
	
	For each $t^{\nuv}\in \overline{\Gr}_{\lav}$, there exists an open neighborhood $U_{\nuv}$ isomorphic to $\bW_{\lav}^{\nuv}\times \Gr_{\nuv}$. Since $\Gr_{\nuv}$ is smooth, using the similar computations as \cite[Lemma 5.3]{Xin25}(or see the \textbf{Step 2} in the proof of Proposition \ref{prop:calculation of tbox}), we deduce that the IC Hodge module $\ICHg_{U_{\nuv}}$ on $U_{\nuv}\cong \bW_{\lav}^{\nuv}\times \Gr_{\nuv}$ satisfy the inequality \eqref{1}. 
	
	Since the union of $U_{\nuv} $ for all $\nuv\le \lav$ covers $\overline{\Gr}_{\lav}$, we finish the proof of the first part of \eqref{1}.
	
	The second part is similar. According to Lemma \ref{lem:Zar_triv}, $\overline{\Gr}_{\lav,\muv}$ is a Zariski locally trivial $\overline{\Gr}_{\muv}$-bundle over $\overline{\Gr}_{\lav}$. Thus Zariski locally $\overline{\Gr}_{\lav,\muv}$ is a product of open subsets of $\overline{\Gr}_{\lav}$ and $\overline{\Gr}_{\muv}$. Using the similar computations as \cite[Lemma 5.3]{Xin25} again, we obtain the second part of \eqref{1}.
\end{proof}

In the appendix \ref{subsec:equivariant}, we reviewed the equivariant Hodge module theory, and extended the graded de Rham functor to a weak equivariant version. For $Y=\overline{\Gr}_{\lav}$ or $\overline{\Gr}_{\lav,\muv}$, let the action of $\hGO$ on $Y$ factors through some finite dimensional quotient $\hGO\twoheadrightarrow G_i$. We can define the following weak $G_i$-equivariant de Rham functor
\[\gr_p\DR^{G_i}:\MHM^{G_i}(Y)\to \mathrm{h}(\Coh^{G_i}(Y)),\ \ p\in \Z.\]
By definition, $\MHM^{\hGO}(Y)=\MHM^{G_i}(Y)$ (see Section \ref{subsec: Satake equivalence for Hodge modules}). Composing with the natural forgetful functor $\Coh^{G_i}(Y)\to \Coh^{\hGO}(Y)$ along $\hGO\twoheadrightarrow G_i$, we obtain the following weak $\hGO$-equivariant de Rham functor
\[\gr_p\DR^{\hGO}:\MHM^{\hGO}(Y)\to \mathrm{h}(\Coh^{\hGO}(Y)),\ \ p\in \Z.\]
According to the compatibility with forgetful functors \eqref{eq:compatible}, this $\gr_p\DR^{\hGO}$ functor doesn't depend on the choice of finite dimensional quotients $\hGO\twoheadrightarrow G_i$, and after forgetting the $\hGO$-equivariance we get the $\gr_p\DR$ functor.

Using Proposition \ref{2}, we can deduce the following corollary.
\begin{corollary}\label{cor: Hodge-de Rham desciption for IC lav}
	Let $\lav,\muv\in \Pv_+$, we have isomorphisms
	\[\cIC(\Omega^k_{\Gr_{\lav}}) \cong \gr_{-k}\DR^{\hGO}(\ICHg_{\lav})[k-d_{\lav}]\in \Coh^{\hGO}(\overline{\Gr}_{\lav}),\]
	\[\cIC(\Omega^k_{\Gr_{\lav,\muv}}) \cong \gr_{-k}\DR^{\hGO}(\ICHg_{\lav,\muv})[k-d_{\lav,\muv}]\in \Coh^{\hGO}(\overline{\Gr}_{\lav,\muv}).\]
	Combined with the Corollary \ref{sigi of KPla}, we obtain the following formula:
	\begin{equation}\label{eq: Hodge formula sigi of KPla}
		\Xi(\cL_{\lav})\cong \bigoplus_{0\le k\le d_{\lav}} \gr_{-k}\DR^{\hGO}(\ICHg_{\lav}) \lr{k-\tfrac12 d_{\lav}} [ 2k- d_{\lav}].
	\end{equation}
\end{corollary}
\begin{proof}
	Using the Proposition \ref{prop:cIC_and_support}, we just need to verify the support conditions \eqref{eq: supp condition} after forgetting the $\hGO$-equivariance. Let $Y=\overline{\Gr}_{\lav}$ or $\overline{\Gr}_{\lav,\muv}$. By definition \eqref{eq:cD}, $\cD_Y(-)=\cHom(-,\omega_Y[-\dim Y ])=\D_Y(-)[-d_Y]$, then using \eqref{eq:D(grDR)}, there is an isomorphism
	\begin{align*}
		\cD_Y\left(\gr_{-k}\DR(\ICHg_{Y})[k-d_Y]\right)=\D_Y\left(\gr_{-k}\DR(\ICHg_{Y})\right)[-k]\cong \gr_{k-d_Y}\DR(\ICHg_{Y})[-k].
	\end{align*}
	Thus the support conditions \eqref{eq: supp condition} can be deduced from Proposition \ref{2}.
\end{proof}

Then we apply the Satake equivalence for Hodge modules to our computation.
\begin{corollary}\label{cor:m*(cIC)}
	Let $\lav,\muv\in \Pv_+$. We have an isomorphism in $\Coh^{\hGO}(\overline{\Gr}_{\lav+\muv})$:
	\[m_*\left(\cIC(\Omega^k_{\Gr_{\lav,\muv}})\right)\cong \bigoplus_{\nuv\in \Pv_+}\cIC(\Omega_{\Gr_{\nuv}}^{k+\tfrac12 d_{\nuv}-\tfrac12 d_{\lav,\muv}})^{\oplus\clamunu}[\tfrac12 d_{\nuv}-\tfrac12 d_{\lav,\muv}].\]
\end{corollary}
\begin{proof}
	According to above Corollary \ref{cor: Hodge-de Rham desciption for IC lav}, we have an isomorphism
	\[m_*\left(\cIC(\Omega^n_{\Gr_{\lav,\muv}})\right)\cong m_*\left(\gr_{-k}\DR^{\hGO}(\ICHg_{\lav,\muv})[k-d_{\lav,\muv}]\right).\]
	Next, we apply the Proposition \ref{prop:equi_version}, which gives a commutativity of weak equivariant graded de Rham functors and proper pushforward functors. Let's take a finite dimensional quotient $\hGO\twoheadrightarrow G_i$ such that the action of $\hGO$ on $\overline{\Gr}_{\lav,\muv}$ and $\overline{\Gr}_{\lav+\muv}$ factor through $G_i$. Proposition \ref{prop:equi_version} implies 
	\begin{equation*}\label{eq:isom}
		m_*\left(\gr_{-k}\DR^{G_i}(\ICHg_{\lav,\muv})\right)\cong \gr_{-k}\DR^{G_i}\left(m_{*,\Hg}(\ICHg_{\lav,\muv})\right)
	\end{equation*}
	in $\Coh^{G_i}(\overline{\Gr}_{\lav,\muv})$, which further implies $$m_*\left(\gr_{-k}\DR^{\hGO}(\ICHg_{\lav,\muv})\right)\cong \gr_{-k}\DR^{\hGO}\left(m_{*,\Hg}(\ICHg_{\lav,\muv})\right)$$ in $\Coh^{\hGO}(\overline{\Gr}_{\lav,\muv})$.

	Combined with the pushfoward computation for Hodge modules \ref{cor:convolution of Hodge modules}, we have:
	\begin{align*}
		\gr_{-k}\DR^{\hGO}\left(m_{*,\Hg}(\ICHg_{\lav,\muv})\right)[k-d_{\lav,\muv}]&\cong \gr_{-k}\DR^{\hGO}\left(\bigoplus_{\nuv \in \Pv_+}(\ICHg_{\nuv})^{\oplus \clamunu}(\tfrac12 d_{\nuv}-\tfrac12 d_{\lav,\muv})\right)[k-d_{\lav,\muv}]
		\\&\cong \bigoplus_{\nuv \in \Pv_+}\gr_{-k-\tfrac12 d_{\nuv}+\tfrac12 d_{\lav,\muv}}\DR^{\hGO}(\ICHg_{\nuv})^{\oplus \clamunu}[k-d_{\lav,\muv}]\\
		&\cong \bigoplus_{\nuv\in \Pv_+}\cIC(\Omega_{\Gr_{\nuv}}^{k+\tfrac12 d_{\nuv}-\tfrac12 d_{\lav,\muv}})^{\oplus\clamunu}[\tfrac12 d_{\nuv}-\tfrac12 d_{\lav,\muv}].
	\end{align*}
	
\end{proof}

\subsection{The end of the proof of Theorem \ref{thm: determine KP_0 on K-group}}
\begin{proof}[Proof of Theorem \ref{thm: determine KP_0 on K-group}]
	By Corollary \ref{injectivity of sigi}, it suffices to show
	\[\Xi([\cL_{\lav}]\conv [\cL_{\muv}])=\sum_{\nuv}\clamunu\Xi[\cL_{\nuv}]. \]
	According to Proposition \ref{prop:sigi is ring hom}, the left hand side is $\Xi[\cL_{\lav}]\conv \Xi[\cL_{\muv}]$, which is $m_*(\Xi[\cL_{\lav}]\tbox \Xi[\cL_{\muv}])$ by definition. Using Proposition \ref{prop:calculation of tbox}, we have
	\[\Xi[\cL_{\lav}]\tbox \Xi[\cL_{\muv}]=\sum_{0\le n\le d_{\lav,\muv}}\left[\cIC(\Omega_{{\Gr}_{\lav,\muv}}^n)\lr{n-\tfrac12 d_{\lav,\muv}}[n]\right].\]
	Then using Corollary \ref{cor:m*(cIC)}, we have
	\begin{align*}
		m_*(\Xi([\cL_{\lav}]\conv [\cL_{\muv}]))&=\sum_{0\le n\le d_{\lav,\muv}}m_*\left[\cIC(\Omega_{{\Gr}_{\lav,\muv}}^n)\lr{n-\tfrac12 d_{\lav,\muv}}[n]\right]\\
		&=\sum_{0\le n\le d_{\lav,\muv}}\sum_{\nuv\in \Pv_+}\clamunu\left[\cIC(\Omega_{\Gr_{\nuv}}^{n+\tfrac12 d_{\nuv}-\tfrac12 d_{\lav,\muv}})\lr{n-\tfrac12 d_{\lav,\muv}}[n+\tfrac12 d_{\nuv}-\tfrac12 d_{\lav,\muv}]\right]\\
		&=\sum_{\nuv\in \Pv_+}\sum_{n'\in \Z}\clamunu\left[\cIC(\Omega_{\Gr_{\nuv}}^{n'})\lr{n'-\tfrac12 d_{\nuv}}[n']\right],
	\end{align*}
	which is $\sum_{\nuv\in\Pv_+}\clamunu\Xi([\cL_{\nuv}])$ by \eqref{sigi of KPla}.
\end{proof}

\section{Tannakian structure on \texorpdfstring{$\KP_0$}{KP0}}\label{sec:Tannakian}
{In this section, we still assume the reductive group $G$ is of \textbf{type A}}. 
\begin{theorem}\label{main thm}
	The category $\KP_0$ admits a neutral Tannakian structure (in the sense of \cite[Def. 2.19]{DM82}), and is equivalent to $\Rep{{\check{G}}}$ as neutral Tannakian categories over $\C$.
\end{theorem}
The main tools are the $K$-group level results established in previous section and the factorization structures discussed in Section \ref{subsec: Renormalized r-matrices}. 
\subsection{Commutativity constraint of \texorpdfstring{$\KP_0$}{KP0}}\label{subsec:comm_constraint}
According to Section \ref{subsec: Renormalized r-matrices}, there is a system of renormalized $r$-matrices in $\KPcoh^{\hGO}(\cR_{G,\g})$. 
\begin{lemma}\label{lem: Lambda=0}
	For any nonzero $\cF,\cG\in \KP_0$, we have $\Lambda(\cF,\cG)=0$.
\end{lemma}
\begin{proof}
	By the Theorem \ref{thm: determine KP_0 on K-group}, the composition factors of $\cF \conv \cG$ only involve simple objects of the form $\cL_{\lav}$ for $\lav\in {\Pv}_+$, while the composition factors of $\cG \conv \cF \{\Lambda(\cF,\cG)\}$ only involve simple objects of the form $\cL_{\lav} \{\Lambda(\cF,\cG)\}$ for $\lav\in {\Pv}_+$. If $\Lambda(\cF,\cG)\ne 0$, the composition factors of $\cF \conv \cG $ and $ \cG \conv \cF \{\Lambda(\cF,\cG)\}$ don't have intersection, thus $\rmat{\cF,\cG}$ must be zero. This contradicts \cite[Cor. 4.7]{CW19} (see Remark \ref{rmk: rmat neq 0}). So $\Lambda(\cF,\cG)=0$. 
\end{proof}

\begin{corollary}\label{cor:comm constraint}
	The system of renormalized $r$-matrices in $\KP_0$ defines a system of commutativity constraints in $\KP_0$.
\end{corollary}
\begin{proof}
	We only need to show $\rmat{\cG,\cF}\circ \rmat{\cF,\cG}=\id$ for any nonzero $\cF,\cG\in \KP_0$.\footnote{The proof is essentially the same as the last paragraph on page 750 of \cite{CW19}.} Other properties of commutativity constraints follow from the definition of renormalized $r$-matrices and the previous Lemma \ref{lem: Lambda=0}. 
	
	Recall that in the construction of the renormalized $r$-matrices, we first define $$C_{\cF,\cG}:=\eta_1(\wt{\cF} \boxtimes \O_X) \conv \eta_2(\O_X \boxtimes \wt{\cG})\in \Coh^{\hGOfac{X^2}}(\cR_{X^2}),$$ 
	and then consider the swapping two factors morphism $$\swap_{\cF,\cG}:C_{\cF,\cG}|_{U}\to C_{\cG,\cF}^{\sw}|_{U}.$$ 
	The number $\Lambda(\cF,\cG)$ measures how this morphism extends over the whole $X^2$. The previous lemma \ref{lem: Lambda=0} shows that $\Lambda(\cF,\cG)=0$, which implies $\swap_{\cF,\cG}$ can extend uniquely to a morphism $$\overline{\swap}_{\cF,\cG}:C_{\cF,\cG}\to C_{\cG,\cF}^{\sw}.$$
	in $\mathrm{h}(\Coh^{\hGOfac{X^2}}(\cR_{X^2}))$, and $\rmat{\cF,\cG}:\cF\conv \cG\to \cG\conv \cF$ is the restriction of $\overline{\swap}$ over $\{0\}\in X^2$. Since $\sw^*({\swap_{\cG,\cF}})\circ\swap_{\cF,\cG}=\id$ in $\mathrm{h}(\Coh^{\hGOfac{X^2}|_{U}})(\cR_{X^2}|_{U})$, the uniqueness of extension implies $\sw^*({\overline{\swap}_{\cG,\cF}})\circ\overline{\swap}_{\cF,\cG}=\id$ in $\mathrm{h}(\Coh^{\hGOfac{X^2}}(\cR_{X^2}))$. Restricting to $\{0\}\subset X^2$, we deduce that $\rmat{\cG,\cF}\circ \rmat{\cF,\cG}=\id$.

\end{proof}

\subsection{Construction of a fiber functor \texorpdfstring{$\mathbf{F}:\KP_0\to \Vect_{\C}$}{F:KP0 → Vect}}\label{subsec:fiber_functor}
In this subsection, we construct a fiber functor $\mathbf{F}:\KP_0\to \Vect_{\C}$, making $\KP_0$ a neutral Tannakian category over $\C$. 

Let $\pi:\Gr\to \pt$ be the projection, which is an ind-projective morphism. Consider the hypercohomology functors
\begin{equation}
	\mathbb{H}^i:=\cH^i\circ \pi_*:\KPcoh^{\hGO}(\Gr)\subset \Coh^{\hGO}(\Gr)\to \Coh^{\hGO}(\pt)^{\heartsuit}[\sqrt{\lr{1}}]=\Rep{\hGO}[\sqrt{\lr{1}}],i\in\Z.
\end{equation}
Here, adding the square root of $\lr{1}$ on the right hand side accounts for the degree modification of $\Coh^{\hGO}(\Gr)$; see Section \ref{subsubsec: degree modification}.

\begin{proposition}\label{calculation of bH*}
	For $\lav\in \Pv_+$, we have 
	\begin{enumerate}
		\item $\bH^i\circ \Xi (\cL_{\lav})=0$ for $i\ne 0$;
		\item\label{grading} $\bH^0\circ \Xi (\cL_{\lav})\cong \bigoplus_{p\in \Z} H^p(\overline{\Gr}_{\lav}, \rmIC_{\lav}^{\C})\lr{\tfrac12 p}$ in $\Rep{\hGO}[\sqrt{\lr{1}}]$. Here $H^*(\overline{\Gr}_{\lav},\rmIC_{\lav}^{\C})$ denotes the cohomology of constructible intersection complex on $\overline{\Gr}_{\lav}$ with coefficient in $\C$. Note that since $\hGO$ is connected, the natural $\hGO$-action on $H^p(\overline{\Gr}_{\lav}, \rmIC_{\lav}^{\C})$ is trivial.
	\end{enumerate}
\end{proposition}
We remark that geometric Satake equivalence implies the intersection cohomology $$IH^*(\overline{\Gr}_{\lav}):=\bigoplus_{p\in \Z} H^p(\overline{\Gr}_{\lav}, \rmIC_{\lav}^{\C})$$ is isomorphic to $V_{\lav}$ as vector spaces, where $V_{\lav}$ is the irreducible representation of $\check{G}$ with highest weight $\lav$.

\begin{proof}
	According to the formula \eqref{eq: Hodge formula sigi of KPla}, we have
	\begin{equation}\label{eq:wzc}
		\bH^i\circ \Xi (\cL_{\lav})\cong \bigoplus_{0\le k\le d_{\lav}}\bH^i\left( \gr_{-k}\DR(\ICHg_{\lav}) \lr{k-\tfrac12 d_{\lav}} [ 2k- d_{\lav}]\right).
	\end{equation}
	Using the commutativity of proper pushforward and graded de Rham functors (Proposition \ref{prop:push}), we have
	\begin{equation}\label{eq:jbz}
		\bH^n\left( \gr_{-k}\DR(\ICHg_{\lav})
		\right)\cong \cH^n\left( \gr_{-k}\DR\left(\pi_{*,\Hg}(\ICHg_{\lav})\right)\right).
	\end{equation}
	Tautologically, we can identify $\gr =\gr \DR$ for Hodge modules on a point, therefore
	\begin{equation}\label{eq:fj}
		\cH^n\left(\gr_{-k}\DR\left(\pi_{*,\Hg}(\ICHg_{\lav})\right)\right) = \cH^n\left(\gr_{-k}\left(\pi_{*,\Hg}(\ICHg_{\lav})\right)\right).
	\end{equation}

	Since the projection $\pi:\overline{\Gr}_{\lav}\to \pt$ is projective and $\ICHg_{\lav}$ is a pure Hodge module of weight $d_{\lav}$, the Hodge structure
	\[\cH^n\left(\pi_{*,\Hg}(\ICHg_{\lav})\right)\in \HM(\pt)\]
	is pure of weight $n+d_{\lav}$ with underlining $\C$-vector space $H^n(\overline{\Gr}_{\lav},\rmIC_{\lav}^{\C})$. According to \cite[Prop. 4.3 (i)]{Fe21}, this is a direct sum of Tate twists of the trivial Hodge structure. Therefore we have 
	\[\cH^n\left(\pi_{*,\Hg}(\ICHg_{\lav})\right)=0, \text{\ \ \ when }2\nmid n+d_{\lav}\]
	and when $2\mid n+d_{\lav}$, the Hodge filtration on $\cH^n\left(\pi_{*,\Hg}(\ICHg_{\lav})\right)$ is 
	\begin{align*}
		F_{-k-1}\cH^n\left(\pi_{*,\Hg}(\ICHg_{\lav})\right)&=0,\\  F_{-k}\cH^n\left(\pi_{*,\Hg}(\ICHg_{\lav})\right)&=H^n(\overline{\Gr}_{\lav},\rmIC_{\lav}^{\C})
	\end{align*}
	for $k=(n+d_{\lav})/2$.
	Equivalently, we have
	\begin{equation*}
		\cH^n\left(\gr_{-k}\left(\pi_{*,\Hg}(\ICHg_{\lav})\right)\right)\cong\gr_{-k}\left(\cH^n\left(\pi_{*,\Hg}(\ICHg_{\lav})\right)\right)=\begin{cases}
			0, \text{\ \ \ if \ } 2k\neq n+d_{\lav}, \\
			H^n(\overline{\Gr}_{\lav},\rmIC_{\lav}^{\C}), \text{\ \ \ if \ } 2k= n+d_{\lav};
		\end{cases}
	\end{equation*}
	here, the first equality follows from the exactness of $\gr_{-k}:\rD^b\HM(\pt)\to \rD^b\Vect_{\C}$. 
	
	Combining \eqref{eq:jbz} \eqref{eq:fj}, we can continue computing \eqref{eq:wzc}:
	\begin{align*}
		\bH^i\circ \Xi(\cL_{\lav})\cong & \bigoplus_{0\le k\le d_{\lav}}\bH^i\left( \gr_{-k}\DR(\ICHg_{\lav}) \lr{k-\tfrac12 d_{\lav}} [ 2k- d_{\lav}]\right) \\
		\cong&  \bigoplus_{0\le k\le d_{\lav}} \cH^{i+2k- d_{\lav}}\left(\gr_{-k}\left(\pi_{*,\Hg}(\ICHg_{\lav})\right)\right) \lr{k-\tfrac12 d_{\lav}} \\
		\cong& \begin{cases}
			0, \text{\ \ \ if \ } i\neq 0, \\
			\bigoplus_{0\le k\le d_{\lav}} H^{2k-d_{\lav}}(\overline{\Gr}_{\lav},\rmIC_{\lav}^{\C}) \lr{k-\tfrac12 d_{\lav}}, \text{\ \ \ if \ } i=0;
		\end{cases}
	\end{align*}
	
	Moreover, this isomorphism can be lifted to $\hGO$-equivariant version: we instead use Proposition \ref{prop:equi_version} in \eqref{eq:jbz}. So
	\[\bH^0\circ \Xi (\cL_{\lav})\cong \bigoplus_{0\le k\le d_{\lav}} H^{2k-d_{\lav}}(\overline{\Gr}_{\lav},\rmIC_{\lav}^{\C}) \lr{k-\tfrac12 d_{\lav}}\]
	as $\hGO$-representations. 
	
	Finally, it's known that $H^{p}(\overline{\Gr}_{\lav},\rmIC_{\lav}^{\C})=0$ if $2\nmid p-d_{\lav}$; thus we conclude the proposition.
\end{proof}

The above proposition has the following two Corollaries \ref{cor:fiber functor} \ref{cor:dimension}.

\begin{corollary}\label{cor:fiber functor}\ 
	\begin{enumerate}
		\item The composition of functors
		\begin{equation*}\label{key}
			\bH^i \circ \Xi : \KP_0 \subset \KPcoh^{\hGO}(\cR_{G,\g}) \xrightarrow{\Xi} \KPcoh^{\hGO}(\cR_{G,0}) \xrightarrow{\bH^i} \Rep{\hGO}[\sqrt{\lr{1}}]
		\end{equation*}
		is zero if $i\ne 0$. And $\bH^0 \circ \Xi(\cF)$ has trivial $\Gmrot$-action for all $\cF\in \KP_0$;
		
		\item Define
		\begin{equation*}
			\mathbf{F}:= \For \circ \bH^0 \circ \Xi : \KP_0 \subset \KPcoh^{\hGO}(\cR_{G,\g}) \xrightarrow{\Xi} \KPcoh^{\hGO}(\cR_{G,0}) \xrightarrow{\bH^0} \Rep{\hGO}[\sqrt{\lr{1}}] \xrightarrow{\For} \Vect_{\C}.
		\end{equation*}
		Then $\mathbf{F}$ is an exact faithful functor.
	\end{enumerate}
\end{corollary}

\begin{proof}
		For any short exact sequence $0\to \cF'\to \cF \to \cF'' \to 0$ in $\KP_0$, we have a long exact sequence 
		\begin{equation}\label{eq:hc}
			\dots \to \bH^i\circ \Xi (\cF') \to \bH^i\circ \Xi (\cF) \to \bH^i\circ \Xi (\cF'') \to \bH^{i+1}\circ \Xi (\cF') \dots.
		\end{equation}
		
		Theorem \ref{thm: determine KP_0 on K-group} shows that every object in $\KP_0$ is an iterated extension of simple objects of the form $\cL_{\lav},\lav\in \Pv_+$. Then Proposition \ref{calculation of bH*} implies
		\[\bH^i(\cF)=0,\forall \cF\in \KP_0, \text{\ \ when\ }i\neq 0.\]
		Therefore, the long exact sequence \eqref{eq:hc} becomes a short exact sequence
		\begin{equation}\label{eq:lyz}
			0\to \bH^0\circ \Xi (\cF') \to \bH^0\circ \Xi (\cF) \to \bH^0\circ \Xi (\cF'')\to 0
		\end{equation}
		for any short exact sequence $ 0\to \cF'\to \cF\to \cF''\to 0$ in $\KP_0$, so $\mathbf{F}$ is exact.

		According to Proposition \ref{calculation of bH*}, $\bH^0\circ \Xi(\cL_{\lav})$ has trivial $\Gmrot$-action for all $\lav\in \Pv_+$. Since $\Gmrot$ is reductive, we can deduce that $\bH^0\circ \Xi (\cF)$ has trivial $\Gmrot$-action for all $\cF\in\KP_0$ by the exactness of $\bH^0\circ \Xi$.
		
		Proposition \ref{calculation of bH*} also implies that $\dim \mathbf{F}(\cL_{\lav})=\dim IH^*(\overline{\Gr}_{\lav})=\dim V_{\lav}\neq 0$ for all $\lav \in {\Pv}_+$, then the faithfulness of $\mathbf{F}$ follows.
\end{proof}

\begin{remark}
	After we finish the construction of neutral Tannakian structure, we can deduce that $\KP_0$ is semisimple; see Lemma \ref{lem:G'} and Remark \ref{rmk:semisimple}. Then Proposition \ref{calculation of bH*} implies $\bH^0 \circ \Xi(\cF)$ has trivial $G_{\O}\rtimes \Gmrot$-action for all $\cF\in \KP_0$.
\end{remark}

The second corollary is a numerical evidence of a tensor structure on $\mathbf{F}$.
\begin{corollary}\label{cor:dimension}
	For any $\cF,\cG\in \KP_0$, we have
	\begin{equation}\label{eq:rjy}
		\dim(\mathbf{F}(\cF\conv \cG))=\dim(\mathbf{F}(\cF))\cdot\dim(\mathbf{F}(\cG)).
	\end{equation}
\end{corollary} 

\begin{proof}
	By Theorem \ref{thm: determine KP_0 on K-group} and Proposition \ref{calculation of bH*}, we have
	\begin{equation*}
		\begin{aligned}
			\dim \mathbf{F}(\cL_{\lav} \conv \cL_{\muv})=&\sum_{\nuv\in {\Pv}_+}\clamunu  \dim \mathbf{F}(\cL_{\nuv})\\ =&\sum_{\nuv\in {\Pv}_+} \clamunu \dim IH^*(\overline{\Gr}_{\nuv})\\ =&\sum_{\nuv\in {\Pv}_+} \clamunu \dim V_{\nuv}\\ = & \dim V_{\lav} \cdot \dim V_{\muv}\\ =& \dim \mathbf{F}(\cL_{\lav}) \cdot \dim \mathbf{F}(\cL_{\muv}).
		\end{aligned}
	\end{equation*}
	Then by the exactness of $\mathbf{F}$, \eqref{eq:rjy} holds for all $\cF,\cG\in \KP_0$.
\end{proof}

\begin{remark}
	Another feature of the above Corollary \ref{cor:dimension} is the fact that $\bH^*\circ \Xi(\cF)$ all have trivial $G_{\O}$-action. It's not hard to show that, for $\overline{\cF},\overline{\cG}\in \Coh^{\hGO}(\Gr)$, if $m_*(\overline{\cG})$ lies in the essential image of the forgetful functor $\Coh^{\Gmrot\times \Gmdil}(\pt)\to \Coh^{\hGO}(\pt)$ along the quotient $\hGO\to \Gmrot\times \Gmdil$, then we have a (noncanonical) isomorphism
	\[m_*(\overline{\cF}\conv\overline{\cG})\cong m_*(\overline{\cF})\otimes m_*(\overline{\cG}).\]
\end{remark}

The remaining part of this subsection is devoted to lifting the numerical result above to a tensor structure on $\mathbf{F}$, which consists of functorial isomorphisms 
\begin{equation*}\label{eq: isom}
	\mathbf{F}(\cF\conv \cG)\cong \mathbf{F}(\cF)\otimes \mathbf{F}(\cG)
\end{equation*}
satisfying certain compatibilities. The precise definition refers to \cite[Def. 1.8]{DM82}.

\begin{theorem}\label{thm: fibre functor}
		The exact faithful $\C$-linear functor $\mathbf{F}$ admits a tensor structure. Thus, we construct a fibre functor on $\KP_0$ and make $\KP_0$ a neutral Tannakian category.
\end{theorem}

\begin{proof}
	We will use the machinary coming from the factorization structure on $\cR_{G,\g}$ as in Section \ref{subsec: Renormalized r-matrices}. 
	Let $\cF,\cG\in \KP_0$. 
	
	Recall in \eqref{eq:rmat1} we have defined an object
	\[C_{\cF,\cG}:=\eta_1(\wt{\cF} \boxtimes \O_X) \conv \eta_2(\O_X \boxtimes \wt{\cG})\in \Coh^{\hGOfac{X^2}}(\cR_{X^2})\]
	with isomorphisms \eqref{eq: isom of CFG}.
	
	Now we apply the global version of the functor $\Xi=\sigi:\Coh^{\hGO}(\cR)\to \Coh^{\hGO}(\Gr)$. By abuse of notation, we still write $i_1:\cR_{X^I}\to \Gr_{X^I}\times_{X^I} N_{\O,X^I}$, $\sigma_1:\Gr_{X^I}\to \Gr_{X^I}\times_{X^I} N_{\O,X^I},\pi:\Gr_{X^I}\to X^I$ for the global version morphisms, and write
	$$ \Xi=\sigi:\Coh^{\hGOfac{X^I}}(\cR_{X^I})\to \Coh^{\hGOfac{X^I}}(\Gr_{X^I}).$$
	
	Then we consider
	\[\cE_{\cF,\cG}:=\pi_*\Xi (C_{\cF,\cG})\in \Coh^{\hGOfac{X^2}}(X^2).\]
	The isomorphisms \eqref{eq: isom of CFG} implies
	\begin{equation}\label{eq:restriction}
		\begin{aligned}
		\cE_{\cF,\cG}|_{\Delta}&\cong \pi_*\Xi(\wt{\cF\conv \cG}),\ \ \ 
		\cE_{\cF,\cG}|_{U}&\cong \left(\pi_*\Xi(\wt{\cF})\boxtimes \pi_*\Xi(\wt{\cG})\right)|_{U}.
	\end{aligned}
	\end{equation}
	If we forget the $\hGOfac{X^2}$-equivariance of $\cE_{\cF,\cG}$ and view it as an object in $\Coh(X^2)$, then \eqref{eq:restriction} implies its restriction to a closed point $x\in \Delta$ is $\pi_*\Xi(\cF\conv \cG)$, while its restriction to a closed point $x\in U$ is $\pi_*\Xi(\cF)\boxtimes \pi_*\Xi(\cG)$. According to Corollary \ref{cor:fiber functor}, $\pi_*\Xi(\cF\conv \cG)$ and $\pi_*\Xi(\cF)\boxtimes \pi_*\Xi(\cG)$ both concentrate at cohomology degree $0$, and by definition they are vector spaces $\mathbf{F}(\cF\conv\cG)$ and $\mathbf{F}(\cF)\otimes\mathbf{F}(\cG)$ respectively. Write 
	\[r:=\dim(\mathbf{F}(\cF\conv\cG))=\dim(\mathbf{F}(\cF)\otimes\mathbf{F}(\cG)).\]
	We claim that the above information implies $\cE_{\cF,\cG}$ is a locally free ordinary coherent sheaf of rank $r$ concentrated at cohomology degree $0$; in other words, $\cE_{\cF,\cG}$ is a vector bundle of rank $r$. We summarize this claim as Lemma \ref{lem:1} below.
	
	Now we add back the $\Gmrot$-equivariance, thus view $\cE_{\cF,\cG}$ as a $\Gmrot$-equivariant vector bundle over $X^2$. The next claim is that the $\Gmrot$-equivariance gives canonical identifications between each fiber of $\cE_{\cF,\cG}$ as vector spaces. We write $R:=\C[X^2]$, $\mathfrak{m}:=$ the maximal ideal of $0\in X^2$, $M:=$ the global section of $\cE_{\cF,\cG}$; we view $R=\bigoplus_{d\in \mathbb{N}}R_d$ as a graded $\C$-algebra and $M$ as a finitely generated graded projective $R$-module. Let $\overline{M}:=M/\mathfrak{m}M=$ the fiber of $\cE_{\cF,\cG}$ at $0\in X^2$; it can be identified with $\bH^0\circ\Xi(\cF\conv \cG)$, thus has trivial $\Gmrot$-action by Corollary \ref{cor:dimension}. The Lemma \ref{lem:2} below implies $M\cong R\otimes_{\C}\overline{M}$ as graded $R$-modules. Let
	\[\mathrm{Isom}_{R\rtimes\Gmrot}(M, R\otimes_{\C}\overline{M})\]
	denote the set of isomorphisms between $M$ and $R\otimes_\C \overline{M}$ as graded $R$-modules. It has a transitively free action of the automorphism group
	\[\mathrm{Aut}_{R\rtimes\Gmrot}(R\otimes_{\C}\overline{M}).\]
	Note that
	\[\mathrm{Aut}_{R\rtimes\Gmrot}(R\otimes_{\C}\overline{M})\subset\mathrm{Hom}_{R\rtimes\Gmrot} (R\otimes_{\C}\overline{M},R\otimes_{\C}\overline{M})= \mathrm{Hom}_{\Gmrot}(\overline{M},R\otimes_{\C}\overline{M}),\]
	and $\mathrm{Hom}_{\Gmrot}(\overline{M},R\otimes_{\C}\overline{M})=\mathrm{Hom}(\overline{M},R_0\otimes_{\C}\overline{M})=\mathrm{Hom}(\overline{M},\overline{M})$ because $\overline{M}$ has trivial $\Gmrot$-action. Thus
	\[\mathrm{Aut}_{R\rtimes\Gmrot}(R\otimes_{\C}\overline{M})=\mathrm{GL}(\overline{M}),\]
	i.e. the graded $R$-module automorphisms of $R\otimes_{\C}\overline{M}$ can only be constant automorphisms. 
	
	In summary, we have a transitively free action of $\mathrm{GL}(\overline{M})$ on $$\mathrm{Isom}_{R\rtimes\Gmrot}(M, R\otimes_{\C}\overline{M})=\mathrm{Isom}_{X^2\rtimes \Gmrot}\left(\cE_{\cF,\cG},\O_{X^2}\otimes_{\C} (\cE_{\cF,\cG}|_{0})\right),$$ which really means we can canonically identify each fiber of $\cE_{\cF,\cG}$ as vector spaces. This procedure gives us a canonical isomorphism
	\[\alpha_{\cF,\cG}:\mathbf{F}(\cF)\otimes F(\cG)\cong \mathbf{F}(\cF \conv \cG).\]
	
	\ 
	
	Next we verify that this class of isomorphisms gives a tensor structure on $\mathbf{F}$.
	
	\begin{enumerate}
		\item For $\cF_1,\cF_2,\cF_3\in \KP_0$, consider the following product
		\[C_{\cF_1,\cF_2,\cF_3}:=\eta_1^{123}(\wt{\cF}_1\boxtimes \O_X \boxtimes \O_X)\conv \eta_2^{123}(\O_X\boxtimes \wt{\cF}_2\boxtimes \O_X)\conv \eta_3^{123}(\O_X\boxtimes \O_X \boxtimes \wt{\cF}_3)\in \Coh^{\hGOfac{X^3}}(\cR_{X^3}),\]
		Write 
		\begin{gather*}
			\Delta_{123}:=\left\{(x_1,x_2,x_3)\in X^3 | x_1=x_2=x_3\right\},\\ \Delta_{ij}:=\left\{(x_1,x_2,x_3)\in X^3 | x_i=x_j\right\}\setminus \Delta_{123}, \text{ for } \{i,j\}\subset \{1,2,3\},\\
			U_{123}:=X^3\setminus \Delta_{12}\cup \Delta_{13} \cup \Delta_{23}.
		\end{gather*}
		Similar to the arguments before, we can consider
		\[\cE_{\cF_1,\cF_2,\cF_3}:=\pi_*\Xi(C_{\cF_1,\cF_2,\cF_3}),\]
		as a $\Gmrot$-equivariant vector bundle over $X^3$ of rank $\dim \mathbf{F}(\cF_1) \cdot \dim \mathbf{F}(\cF_2) \cdot \dim \mathbf{F}(\cF_3)$. Let $x\in X^3$ be a closed point of $X^3$; the fiber of $\cE_{\cF_1,\cF_2,\cF_3}$ at $x$ is
		\[\begin{cases}
			\mathbf{F}(\cF_1\conv\cF_2\conv \cF_3), \text{\ \ if\ }x\in \Delta_{123};\\
			(\mathbf{F}(\cF_i\conv\cF_j)\otimes \cF_k)_{i,j,k}, \text{\ \ if\ }x\in \Delta_{ij};\\
			\mathbf{F}(\cF_1)\otimes \mathbf{F}(\cF_2)\otimes \mathbf{F}(\cF_3),\text{\ \ if\ }x\in U_{123};
		\end{cases}\]
		here the notation in the second line means we put $\cF_i,\cF_j,\cF_k$ at $i,j,k$-th positions. The $\Gmrot$-equivariant vector bundle $\cE_{\cF_1,\cF_2,\cF_3}$ produces natural identifications among the fibers; by the factorization equivalences, these identifications can also be expressed in terms of $\alpha_{-,-}$ in natual ways. This means the first diagram in \cite[Def. 1.8]{DM82} commutes.
		
		\item Recall in Corollary \ref{cor:comm constraint}, we have an isomorphism $$\overline{
		\swap}:C_{\cF,\cG}\to C_{\cG,\cF}^{\sw},$$ such that its restriction to $U\subset X^2$ is the swapping two factors isomorphism $\swap:(\cF\boxtimes \cG)|_U\to \sw^*(\cG\boxtimes \cF)|_{U}$, while its restriction to $\{0\}\subset X^2$ is $\rmat{\cF,\cG}:\cF\conv\cG\to \cG\conv\cF$. Applying $\pi_*\Xi$, we have an isomorphism of two $\Gmrot$-equivariant vector bundles
		\[\pi_* \Xi(\overline{\swap}):\cE_{\cF,\cG}\to \sw^*(\cE_{\cG,\cF}),\]
		which gives the commutativity of the diagram:
		% https://q.uiver.app/#q=WzAsNCxbMCwwLCJcXGNFX3tcXGNGLFxcY0d9fF97eF9VfSJdLFsxLDAsIlxcY0Vfe1xcY0csXFxjRn18X3t4X3tVfX0iXSxbMCwxLCJcXGNFX3tcXGNGLFxcY0d9fF97eF97XFxEZWx0YX19Il0sWzEsMSwiXFxjRV97XFxjRyxcXGNGfXxfe3hfe1xcRGVsdGF9fSJdLFswLDFdLFswLDIsIlxcYWxwaGFfe1xcY0YsXFxjR30iLDJdLFsyLDNdLFsxLDMsIlxcYWxwaGFfe1xcY0csXFxjRn0iXV0=
		\[\begin{tikzcd}
			{\cE_{\cF,\cG}|_{x_U}} & {\sw^*(\cE_{\cG,\cF})|_{x_{U}}} \\
			{\cE_{\cF,\cG}|_{x_{\Delta}}} & {\sw^*(\cE_{\cG,\cF})|_{x_{\Delta}}}
			\arrow[from=1-1, to=1-2]
			\arrow["{\alpha_{\cF,\cG}}"', from=1-1, to=2-1]
			\arrow["{\alpha_{\cG,\cF}}", from=1-2, to=2-2]
			\arrow[from=2-1, to=2-2]
		\end{tikzcd}\]
		for any closed points $x_U\in U$, $x_{\Delta}\in \Delta$. This means the diagram % https://q.uiver.app/#q=WzAsNCxbMCwwLCJcXG9tZWdhKFxcY0YpXFxvdGltZXMgXFxvbWVnYShcXGNHKSJdLFsxLDAsIlxcb21lZ2EoXFxjRylcXG90aW1lcyBcXG9tZWdhKFxcY0YpIl0sWzAsMSwiXFxvbWVnYShcXGNGXFxjb252XFxjRykiXSxbMSwxLCJcXG9tZWdhKFxcY0ZcXGNvbnZcXGNHKSJdLFswLDFdLFswLDIsIlxcYWxwaGFfe1xcY0YsXFxjR30iLDJdLFsyLDNdLFsxLDMsIlxcYWxwaGFfe1xcY0csXFxjRn0iXV0=
		\[\begin{tikzcd}
			{\mathbf{F}(\cF)\otimes \mathbf{F}(\cG)} & {\mathbf{F}(\cG)\otimes \mathbf{F}(\cF)} \\
			{\mathbf{F}(\cF\conv\cG)} & {\mathbf{F}(\cG\conv\cF)}
			\arrow[from=1-1, to=1-2]
			\arrow["{\alpha_{\cF,\cG}}"', from=1-1, to=2-1]
			\arrow["{\alpha_{\cG,\cF}}", from=1-2, to=2-2]
			\arrow[from=2-1, to=2-2]
		\end{tikzcd}\]
		commutes.
	\end{enumerate}
	
\end{proof}

\begin{lemma}\label{lem:1}
	Let $Y$ be a smooth variety over $\C$, and $\cC\in\Coh(Y)$. Assume for any $k\in \Z$ and any closed point $y\in Y$, the dimension of the non-derived fiber $ \cH^k(i_y^*\cC)$ is equal to some $r_k\in \mathbb{N}$ which only depends on $k$. Then for any $k\in \Z$, $\cH^k(\cC)$ is a locally free sheaf of rank $r_k$.
\end{lemma}
\begin{proof}
	We will use a standard sublemma. Let $\cM$ be an ordinary coherent sheaf on $Y$, such that for any closed point $y\in Y$, the dimension of the non-derived fiber $\cH^0(i_y^*\cM)$ is equal to some fixed $n\in \mathbb{N}$. Then $M$ is a locally free sheaf of rank $n$ on $Y$.\footnote{For the reader's convenience, we present the proof here. For a closed point $y\in Y$, let $\O_{Y,y}$ be the stalk of $\O_{Y}$ at $y$, and $\mathfrak{m}_y$ be its maximal ideal. Let $\bar{s}_1,\dots,\bar{s}_n$ be a basis of $\cM/\mathfrak{m}_y \cM$; by Nakayama's lemma, we can lift them to a set of generators $s_1,\dots , s_n$ of the stalk $\cM_y$ as an $\O_{Y,y}$-module. Let $V$ be an affine open neighborhood of $y$ such that $s_1,\dots, s_n$ lives in $\O_{Y}(V)$. Let $t_1,\dots,t_l$ be a set of generators of the $\O_{Y}(V)$-module $\cM(V)$. At the stalk $\cM_x$, $t_i=\sum_{j}a_{ij}s_j$ for some $a_{ij}\in \O_{Y,y}$. Let $V'\subset V$ be an affine open neighborhood of $y$ such that $a_{ij}$ all lives in $\O_Y(V')$. Then the morphism $\phi: \O_{V'}^{\oplus n}\to \cM|_{V'}$ defined by the sections $s_1,\dots,s_n$ is surjective. Since $\cM/\mathfrak{m}_{y'}\cM$ has dimension $n$ for all closed point $y'\in V'$, the morphism $\phi$ on each non-derived fiber $\phi|_{y'}:(\O_{Y}/\mathfrak{m}_{y'}\O_{Y})^{\oplus n}\to \cM/\mathfrak{m}_{y'}\cM$ is isomorphism. So for any $(b_1,\dots,b_n)\in \O_{Y}^{\oplus n}(V')$ such that $\phi(b_1,\dots,b_n)=0$, we know $$b_1,\dots ,b_n \in \bigcap_{\text{closed pt }y'\in V'}\mathfrak{m}_{y'} .$$
	By Hilbert's Nullstellensatz, $\bigcap_{\text{closed pt }y'\in V'}\mathfrak{m}_{y'}$ is the nilpotent radical of $\O_{Y}(V')$, which is $0$. Therefore $\phi$ is isomorphim. We conclude that $\cM$ is locally free of rank $n$.}

	This sublemma implies another standard sublemma. Let $0\to \cM''\to \cM\to \cM'\to 0$ be a short exact sequence of ordinary coherent sheaves. Suppose $\cM$ and $\cM'$ are locally free sheaf of finite rank, then $\cM''$ is also a locally free sheaf of finite rank. Indeed, for any closed point $y$ in $Y$, we have an exact sequence
	\[\cH^{-1}(i_y^*\cM')=0\to \cH^0(i_y^*\cM'')\to \cH^0(i_y^*\cM) \to \cH^0(i_y^*\cM')\to 0.\]
	Thus $\dim \cH^0(i_y^*\cM'')$ is the constant $\rank \cM-\rank \cM'$. Then by the sublemma above, we can deduce that $\cM''$ is a locally free sheaf of finite rank.
	
	\ 
	
	Now we come back to the original problem. Since $Y$ is smooth, $\cC$ can be represented by a perfect complex $0\to \cM^{m}\to \cM^{m+1}\to \dots \to \cM^{n}\to 0$  (i.e. each $\cM^{i}$ is a locally free sheaf of finite rank). Let $\mathcal{P}^{k}:=\operatorname{image} \left(\cM^{k}\to \cM^{k+1}\right)$, and $\mathcal{Q}^k:=\operatorname{ker} \left(\cM^{k}\to \mathcal{M}^{k+1}\right)$. We have two short exact sequences
	\begin{equation}\label{eq:yyc}
		0\to \mathcal{Q}^k\to \cM^k\to \mathcal{P}^k\to 0,
	\end{equation}
	\begin{equation}\label{eq:hyh}
	0\to \mathcal{P}^k\to \mathcal{Q}^{k+1} \to \cH^{k+1}(\cC)\to 0.
	\end{equation}

	We will show by induction on $k$ that $\mathcal{P}^k,\mathcal{Q}^k,\cH^k(\cC)$ are all locally free sheaves of finite rank, and $\rank \cH^k(\cC)=r_k$. When $k=n+1$, there is nothing to prove. Suppose for the $k+1$ case this is true. The second short exact sequence \eqref{eq:hyh} implies $\mathcal{P}^k$ is a locally free sheaf of finite rank by the second sublemma; then the first short exact sequence \eqref{eq:yyc} implies $\mathcal{Q}^k$ is a locally free sheaf of finite rank by the second sublemma. 
	
	For any closed point $y\in Y$, applying $i_y^*$ to \eqref{eq:hyh}, we know $i_y^*\mathcal{P}^k\to i_y^*\mathcal{Q}^{k+1}$ is injective; applying $i_y^*$ to the $k+1$ case of \eqref{eq:yyc}, we know $i_y^*\mathcal{Q}^{k+1}\to i_y^*\mathcal{M}^{k+1}$ is injective. Therefore we have
	\[ i_y^*\mathcal{Q}^k\overset{\eqref{eq:yyc}}{\cong} \operatorname{ker}( i_y^*\cM^k\to i_y^*\mathcal{P}^k) \cong \operatorname{ker}( i_y^*\cM^k\to i_y^*\cM^{k+1}),\]
	thus 
	\[ \cH^k(i_y^*\cC)\cong \cH^k(\dots \to i_y^*\cM^{k-1}\to i_y^*\cM^{k} \to i_y^*\cM^{k+1} \to \dots)\cong  \operatorname{coker}(i_y^*\mathcal{M}^{k-1}\to i_y^*\mathcal{Q}^k),\]
	and the right hand side is isomorphic to the non-derived fiber $\cH^0(i_y^*\cH^{k}(\cC))$ by the right exactness of $i_y^*$ and the exact sequence
	\[\cM^{k-1}\to \mathcal{Q}^k\to \cH^{k}(\cC)\to 0.\]
	Therfore $\cH^0(i_y^*\cH^{k}(\cC))$ is constant $r_k$ for any closed point $y$ in $Y$, and by the first sublemma we can deduce $\cH^{k}(\cC)$ is a locally free sheaf of rank $r_k$. Then we finish the proof for the $k$ case.

\end{proof}

\begin{lemma}\label{lem:2}
	Let $R=\bigoplus_{d\in \mathbb{N}}R_d$ be a graded $\C$-algebra such that $R_0=\C$. Let $M$ be a finitely generated graded projective $R$-module. Write the maximal ideal $\mathfrak{m}:=\bigoplus_{d>0}R_d$, and $\overline{M}:=M/\mathfrak{m}M$. Then $M$ is a graded free $R$-module, and isomorphic to $R\otimes_{\C}\overline{M}$ as graded $R$-modules;
\end{lemma}
\begin{proof}
	We choose a grading preserving section $s:\overline{M}\hookrightarrow M$ of the projection $p:M\twoheadrightarrow \overline{M}$ (i.e. $p\circ s=\id$). Then $s$ produces a graded $R$-module homomorphism
	\[\phi:R\otimes_{\C}\overline{M}\to M.\]
	This map is an isomorphism after base change to $R/\mathfrak{m}$, thus $\phi$ is surjective by Nakayam's lemma for graded $R$-modules. Since $M$ is projective, the map $\phi$ has a grading-preserving section $\psi$. Then $\psi$ is also an isomorphism after base change to $R/\mathfrak{m}$, thus $\psi$ is also surjective. Therefore $\psi$ and $\phi$ produce an isomorphism between $R\otimes_{\C}\overline{M}$ and $M$.
\end{proof}

\subsection{Proof of Theorem \ref{main thm}}
After giving the category $\KP_0$ a neutral Tannakian structure, we can apply \cite[Thm. 2.11]{DM82}, and obtain an equivalence of neutral Tannakian categories
\[\KP_0\cong \Rep{G'}\]
for some (classical) affine group sheme $G'$ over $\C$. To finish the proof of the main Theorem \ref{main thm}, we need to identify $G'$ with ${\check{G}}$.

\begin{lemma}\label{lem:G'}
	The affine group scheme $G'$ is algebraic, i.e. finite type over $\C$. Moreover, $G'$ is connected and reductive. 
\end{lemma}
\begin{proof}
	To prove $G'$ is algebraic, we use \cite[Prop. 2.20]{DM82}: we need to find a tensor generator $X\in \Rep{G'}\cong \KP_0$, which means $X$ generates the whole category $\KP_0$ under taking direct sums, convolutions, subquotients and duals. We choose a set of generators $\lav_1,\dots,\lav_n$ of $\Pv_+$ as monoid, and take $X= \bigoplus_{1\le i\le n}\cL_{\lav_i}$, then this assertion follows from the definition of $\KP_0$ and the formula in Theorem \ref{thm: determine KP_0 on K-group}.
	
	To prove $G'$ is connected, we use \cite[Cor. 2.22]{DM82}: we need to show that for any non-unit object $X\in \Rep{G'}\cong \KP_0$, the subcategory $\lr{X}$ consisting of subquotients of $X^{\oplus n},n\in \mathbb{N}$ is not stable under convolutions. By Theorem \ref{thm: determine KP_0 on K-group}, the composition factors of $X$ have the form $\left\{\cL_{\lav_1},\dots,\cL_{\lav_n}\right\}$. Then any object in $\lr{X}$ is also an iterated extension of simples in $\left\{\cL_{\lav_1},\dots,\cL_{\lav_n}\right\}$. Note that $\cL_{m\lav_i}$ appears as a subquotient of $\cL_{\lav_i}^{\conv m}$ by the formula in Theorem \ref{thm: determine KP_0 on K-group}, and $\cL_{m\lav_i}$ doesn't lie in $\left\{\cL_{\lav_1},\dots,\cL_{\lav_n}\right\}$ for $\lav_i\neq 0$ and $m$ large enough, thus $\lr{X}$ is not stable under convolution.
	
	Let $R_u(G')$ be the unipotent radical of this connected algebraic group $G'$, and $L'$ be the quotient of $R_u(G')\hookrightarrow G'$, which is a connected reductive group. Then $\Rep{L'}$ is generated by simple objects under taking direct sums, tensors, subquotients and duals. The quotient $G'\twoheadrightarrow L'$ induces a fully faithful embedding $\Rep{L'}\hookrightarrow \Rep{G'}$, and all the simple objects in $\Rep{G'}$ come from $\Rep{L'}$. By the definition of $\KP_0$, $\Rep{G'}\cong \KP_0$ is generated by simple objects under taking direct sums, tensors, subquotients and duals. So the embedding $\Rep{L'}\hookrightarrow \Rep{G'}$ is essentially surjective, which implies $G'\xrightarrow{\cong} L'$ and $G'$ is reductive.
\end{proof}

\begin{remark}\label{rmk:semisimple}
	From the argument in the Lemma \ref{lem:G'} above, we see that the neutral Tannakian structure on $\KP_0$ immediately implies $\cL_{\lav}\conv \cL_{\muv}$ is semisimple for simple objects $\cL_{\lav}, \cL_{\muv}$ in $\KP_0$, which is not directly evident.
\end{remark}

For a monoidal abelian category $\cA$, let $K^+(\cA)$ denote its the Grothendieck semiring. In other word, $K^+(\cA)$ is the set of equivalence classes of objects in $\cA$, equipped with $\oplus$ as addition and $\otimes$ as multiplication. If $\cA$ is semisimple and of finite length, we can view $K^+(\cA)$ as a sub-semiring of $K(\cA)$. For a reductive group $H$, let $K^+[H]$ denote $K^+(\Rep{H})$.

We will use the following theorem in \cite{KLV}.

\begin{theorem}[{\cite[Thm. 1.2]{KLV}}]
	Let $k$ be an algebraically closed field of characteristic zero. For two connected reductive groups $H_1$ and $H_2$ over $k$, let $\phi:K^+[H_1]\xrightarrow{\cong} K^+[H_2]$ be an isomorphism of semirings.
	
	Then there exists an isomorphism $\rho: H_2\xrightarrow{\cong} H_1$, such that $\rho^*=\phi$. (Moreover, such isomorphism is unique up to conjugation.)
\end{theorem}

\begin{proof}[Proof of Theorem \ref{main thm}]
	According to Theorem \ref{thm: determine KP_0 on K-group}, we have an isomorphism of semirings:
	\begin{equation*}
		\begin{aligned}
			\phi:K^+[G']&\to K^+[{\check{G}}]\\
			[\cL_{\lav}]&\mapsto [V_{\lav}].
		\end{aligned}
	\end{equation*}
	Then \cite[Thm. 1.2]{KLV} gives an isomorphism ${\check{G}}\cong G'$.
\end{proof}

\appendix

\section{Hodge modules and associated graded constructions}\label{sec:appendix}
In this appendix, we review some basic facts on Hodge modules and associated graded constructions used in the main context. In this appendix, we rarely use derived algebraic geometry and we work with ordinary categories. Thus we will use some traditional notations, for example, the derived category of quasi-coherent sheaves on a scheme $X$ will be denoted as $\rD_{\mathrm{qcoh}}(\O_{X})$, and its full subcategory consisting of complexes with coherent cohomologies will be denoted by $\rD_{\mathrm{coh}}(\O_{X})$; they correspond to the homotopy categories of $\QCoh(X)$ and $\Coh(X)$ in the main context respectively. We will also write "sheaves" for ordinary sheaves. We will only consider those $\cD$-modules (together with their filtration pieces) quasi-coherent as $\O$-modules. All perverse sheaves refer to perverse constructible sheaves in this appendix.

\subsection{Categories of Hodge modules}\label{subsec:Hodge}
Let $Y$ be a smooth variety over $\C$. The basic object in the Hodge module theory is filtered regular holonomic $\cD_Y$-modules with $\Q$-structure, given by a quadruples $\mbf{M}=(\cM,F_{\bullet}\cM,K,\alpha)$, consisting of
\begin{itemize}\label{quadruple}
	\item a regular holonomic right $\cD_Y$-module $\cM$,
	\item a good filtration $F_{\bullet}\cM$ on $\cM$ (i.e. $(\cM,F_{\bullet}\cM)$ is a filtered $(\cD_Y,F_{\bullet}\cD_Y)$-module for the PBW-filtration on $\cD_Y$, and $\gr ^F\cM$ is a coherent $\gr ^F\cD_{Y}$-module),
	\item a $\Q$-coefficient perverse sheaf $K$ on the complex manifold $Y^{\mathrm{an}}$, with an isomorphism of sheaves $\alpha: \DR^{\mathrm{an}}(\cM^{\mathrm{an}})\cong \C\otimes_{\Q}K$. 
\end{itemize}
Here, $\DR^{\mathrm{an}}$ is the de Rham functor in Riemann-Hilbert correspondence
\[\DR^{\mathrm{an}}:\cD_{Y}\text{-}\mod_{rh}\xrightarrow{\cong} \Perv(Y,\C),\cM\mapsto \O_{Y^{\mathrm{an}}}\otimes_{\cD_{Y^{\mathrm{an}}}}\cM^{\mathrm{an}}.\]
We write $\FM_{\mathrm{rs}}^{\Q}(\cD_{Y})$ for the category of such quadruples. Define the $\Rat$ functor 
\begin{equation}\label{eq:Rat}
	\Rat:\FM_{\mathrm{rs}}^{\Q}(\cD_{Y})\to \Perv(Y,\Q)
\end{equation}
by $\Rat(\cM,F_{\bullet}\cM,K,\alpha)=K$. This functor is faithful.

In \cite{Sai88}, the author defines the category $\HM(Y,w)$ of (polarizable)\footnote{We will only consider polarizable (mixed) Hodge modules, so we omit the word by default.} pure Hodge modules of weight $w$ on $Y$, which is a full subcategory of $ \FM_{\mathrm{rs}}^{\Q}(\cD_Y)$. The category $\HM(Y,w)$ is an abelian category. \cite{Sai88} also proved the category of pure Hodge modules on smooth varieties satisfy Kashiwara's equivalence. Then the definition of $\HM(Y,w)$ can be extended to general varieties over $\C$. Namely, for a possibly singular variety $Y$, we choose a closed embedding of $Y$ into some smooth variety $Y'$, and then define $\HM(Y,w)$ to be the full subcategory of $\HM(Y',w)$ consisting of Hodge modules set-theoretically supported on $Y$. The $\Rat$ functor \eqref{eq:Rat} can also be defined for this possibly singular variety $Y$. It can be checked that the definition of $\HM(Y,w)$ is independent of the choice of embeddings; see for example \cite[Sec. B]{Sch16}.

Let $X$ be an integral subscheme of $Y$. A Hodge module $\mbf{M}\in \HM(Y,w)$ is said to have strict support $X$ if $\Rat(\mbf{M})$ is an IC-perverse sheaf with support $X$. Then $\mbf{M}$ is a variation of pure Hodge structure generically on $X$. Conversely, a variation of pure Hodge structure generically on $X$ can be uniquely IC extended to a pure Hodge module on $X$. In particular, there is a pure Hodge module of weight $\dim X$ on $X$, called the IC-Hodge module on $X$:
\[\ICHg_X\]
such that $\Rat(\ICHg_X)=\rmIC_X^{\Q}$\label{ICHg}. Moreover, each Hodge module in $\HM(Y,w)$ has a unique decomposition by strict support, and $\HM(Y,w)$ is a semisimple category.

We also review Saito's mixed Hodge module theory. In addition to the data in \ref{quadruple}, the definition in mixed Hodge module consists of a data called weight filtration. For a variety $Y$ over $\C$, Saito \cite{Sai90} defines the abelian category of mixed (polarizable) Hodge modules $\mathrm{MHM}(Y)$, and its bounded derived category is denoted by $\rD^b\mathrm{MHM}(Y)$. The $\Rat$ functor \ref{eq:Rat} can also be extended to a t-exact functor from $\rD^b\mathrm{MHM}(Y)$ to the bounded derived category of perverse sheaves:
\[\Rat:\rD^b\mathrm{MHM}(Y)\to \rD^b\Perv(Y,\Q).\]

Define $\rD^b\HM(Y,w)$ to be the full subcategory consisting of complexes $\mbf{M}^{\bullet}$ such that $\cH^i(\mbf{M}^{\bullet})\in \HM(Y,w+i)$. (\textbf{Warning}: this is not the derived category of $\HM(Y,w)$.) There is a grading shift functor on $\rD^b\mathrm{MHM}(Y)$ called the Tate twist, which is denoted by $(1)$. This functor send $\rD^b\HM(Y,w)$ to $\rD^b\HM(Y,w-2)$. For a quadruple $(\cM,F_{\bullet}\cM,K,\alpha)\in \FM_{\mathrm{rs}}^{\Q}(\cD_{Y})$ underlines a Hodge module $\mbf{M}$, the Tate twist $(1)$ acts as 
\begin{equation}\label{eq:Tate twist}
	(\cM,F_{\bullet}\cM,K,\alpha)(1)=(\cM,F_{\bullet-1}\cM,K(1),\alpha)
\end{equation}
where $K(1):=K\otimes_{\Q}(2\pi i) \Q\subset K\otimes_{\Q}\C$.

Six functors formalism is established in mixed Hodge module theory \cite{Sai90}). We review a few of the functors relevant to us. Let $Y$ be a variety over $\C$. There is a t-exact duality functor 
\begin{equation}\label{eq:wjy}
	\bD_Y: \rD^b\mathrm{MHM}(Y)\to \rD^b\mathrm{MHM}(Y),
\end{equation}
compatible with the Verdier dual for $\rD^b\Perv(Y,\Q)$ under $\Rat$. This duality functor also restricts to $\rD^b\HM(Y,w)$ as:
\[\bD_Y: \rD^b\HM(Y,w)\to \rD^b\HM(Y,-w).\]

Let $f:X\to Y$ be a proper morphism between varieties over $\C$. There is a pushforward functor:
\[f_{*,\Hg}:\rD^b\mathrm{MHM}(X)\to \rD^b\mathrm{MHM}(Y)\]
compatible with the pushforward functor on $\rD^b\Perv(-,\Q)$ under $\Rat$. This pushforward functor also restricts to $\rD^b\HM(Y,w)$ as:
\[f_{*,\Hg}: \rD^b\HM(X,w)\to \rD^b\HM(Y,w).\footnotemark\]
{\footnotetext{In fact, $f_{*,\Hg},f_{!,\Hg}$ are defined on $\rD^b\mathrm{MHM}(-)$ for general morphism $f$, and $f_{*,\Hg}$ doesn't decrease weights while $f_{!,\Hg}$ doesn't increase weights.} }

When $X$ and $Y$ are smooth, the pushforward functor $f_{*,\Hg}$ is compatible with the pushforward functor for their underlying filtered $\cD$-modules. More precisely, define 
\[\cD_{\hbar}:=\mathrm{Rees}(\cD):=\bigoplus_{n\ge 0}\hbar^n F_{n}\cD\subset \cD\otimes_{\C}\C[\hbar].\]
Using Rees construction, we can view a filtered $\cD$-module as a graded $\cD_{\hbar}$-module. This grading can also be viewed as a $\Gm$-equivariance; we also denote this $\Gm$ by $\Gmdil$ since it is compatible with the $\Gmdil$-equivariance under the associated graded construction \eqref{eq:lxy}. We write
\begin{equation}\label{eq:U}
	\mathrm{U}_X:\mathrm{MHM}(X)\to \cD_{X,\hbar}{\text{-}}\mathrm{mod}^{\mathrm{gr}}
\end{equation}
for the functor sending a mixed Hodge module to its underlying graded $\cD_{X,\hbar}$-module; this is an exact functor. We denote its derived functor by the same symbol. Define the transfer $\cD_{\hbar}$-module
\[\cD_{X\to Y,\hbar}:=f^*\cD_{Y,\hbar},\]
and pushforward functor
\begin{equation}\label{eq:int-push}
	\int_{f}(-):=f_*\left((-)\otimes_{\cD_{X,\hbar}}\cD_{X\to Y,\hbar} \right):\rD^b(\cD_{X,\hbar}\text{-}\mathrm{mod}^{\gr})\to \rD^b(\cD_{Y,\hbar}\text{-}\mathrm{mod}^{\gr})
\end{equation}
Then by the definition of pushforward functors for mixed Hodge modules, there is a natural isomorphism
\begin{equation}\label{eq:U&push}
	\mathrm{U}_Y\circ f_{*,\Hg}\cong \int_f \circ \mathrm{U}_X.
\end{equation}

Let $g:Y\to Z$ be a smooth morphism of relative dimension $d$ between varieties over $\C$. There are pullback functors:
\[g^{!,\Hg},g^{*,\Hg}:\rD^b\mathrm{MHM}(Z)\to \rD^b\mathrm{MHM}(Y)\]
compatible with the corresponding pullback functors for $\rD^b\Perv(-,\Q)$ under $\Rat$. Since the morphism $g$ is smooth, we have an isomorphism of functors $g^{!,\Hg}\cong g^{*,\Hg}[2d](d)$. They also restricts to pure Hodge modules as:
\[g^{!,\Hg}\cong g^{*,\Hg}[2d](d): \rD^b\HM(Z,w)\to \rD^b\HM(Y,w).\]
Define $g^{\dagger,\Hg}:=g^{*,\Hg}[d]\cong g^{!,\Hg}[-d](-d)$. This functor is t-exact, thus we can restrict it to the heart:
\[g^{\dagger,\Hg}:\HM(Z,w)\to \HM(Y,w+d).\]

When $Y$ and $Z$ are smooth, the pullback functor $g^{\dagger}$ is compatible with the pullback functor for their underlying filtered $\cD$-modules. More precisely, define the pullback functor 
\[g^{\dagger}:=\Omega^{\mathrm{top}}_{Y/Z}\otimes g^*(-)\lr{-d}:\rD^b(\cD_{Z,\hbar}\text{-}\mathrm{mod}^{\gr})\to \rD^b(\cD_{Y,\hbar}\text{-}\mathrm{mod}^{\gr}).\]
Here $\Omega^{\mathrm{top}}_{Y/Z}$ is the top relative differential sheaf, and $\lr{-d}$ is the grading shift functor with respect to the $\Gmdil$-action. Then by the definition of pullback functors for Hodge modules, there is a natural isomorphism
\begin{equation}\label{eq:U&pull}
	\mathrm{U}_Y\circ g^{\dagger,\Hg}\cong g^{\dagger} \circ \mathrm{U}_Z.
\end{equation}

\subsection{Graded de Rham functor}\label{subsec:grDR}
Let $Y$ be a variety over $\C$, and we choose a closed embedding $i:Y\hookrightarrow Y'$ into some smooth variety $Y'$\footnote{In the remaining sections, a variety labeled with a superscript prime (such as $Y'$) is implicitly assumed to be smooth.} (such smooth ambient space always exists since we assume $Y$ is quasi-projective over $\C$). 

Let $\cM'$ be a $\cD_{Y'}$-module. Define the de Rham complex of $\cM'$ to be the complex of $\C$-vector spaces
\[\DR_{Y'}(\cM'):=\left[0\to \cM'\otimes_{\O_{Y'}} \bigwedge^{d_{Y'}}\cT_{Y'} \to \dots \to \cM'\otimes_{\O_{Y'}} \bigwedge^{1}\cT_{Y'}\to \cM'\to 0\right],\]
which lives in cohomological degrees $-d,\dots,0$. 

Let $(\cM',F_{\bullet}\cM')$ be a filtered $\cD_{Y'}$-module. Then the de Rham complex inherits a filtration:
\[ F_{p}\DR_{Y'}(\cM')=\left[0\to F_{p-d_{Y'}}\cM'\otimes_{\O_{Y'}} \bigwedge^{d_{Y'}}\cT_{Y'} \to \dots \to F_{p-1}\cM'\otimes_{\O_{Y'}} \bigwedge^1\cT_{Y'}\to F_p\cM'\to 0\right].\]
The associated graded is
\begin{equation}\label{grDR}
	\gr^F_{p}\DR_{Y'}(\cM')=\left[0\to \gr^F_{p-d_{Y'}}\cM'\otimes_{\O_{Y'}} \bigwedge^{d_{Y'}}\cT_{Y'} \to \dots \to \gr^F_{p-1}\cM'\otimes_{\O_{Y'}} \bigwedge^1\cT_{Y'}\to \gr^F_p\cM'\to 0\right],
\end{equation}
which is a complex of $\O_{Y'}$-modules. This complex is called the \textit{graded de Rham complex}.

This complex can also be interpreted in the following way. Consider the graded $\mathrm{Sym}_{\O_{Y'}}\cT_{Y'}$-module 
\begin{equation}\label{eq:lxy}
	\gr_{Y'}\cM':=\bigoplus_{p\in \Z}\gr_{p}^F\cM'.
\end{equation}
It can also be constructed by
\[\cM'\otimes_{\C[\hbar]}\C_0, \ \ \C_0:=\C[\hbar]/(\hbar)\]
when viewing $\cM'$ as a $\cD_{\hbar}$-module. Then using the Koszul resolution, the derived tensor product
\[\gr_{Y'}\cM'\underset{\mathrm{Sym}_{\O_{Y'}}\cT_{Y'}}{\bigotimes}\O_{Y'}\]
is computed by the graded de Rham complex $$\gr\DR_{Y'}(\cM'):=\bigoplus_{p\in \Z}\gr^F_{p}\DR_{Y'}(\cM')$$
in the derived category of graded $\mathrm{Sym}_{\O_{Y'}}\cT_{Y'}$-modules. Let $p_{Y'}:\rT^*Y'\to Y'$ be the natural projection of the cotangent bundle $\rT^*Y'$, and $\sigma_{Y'}:Y'\to \rT^*Y'$ be the zero section embedding. We have $\mathrm{Sym}_{\O_{Y'}}\cT_{Y'}={p_{Y'}}_*\O_{\mathrm{T}^*Y'}$ and $\gr\cM'$ can be equivalently viewed as a $\Gmdil$-equivariant quasi-coherent sheaf on $\rT^*Y'$. Then the above constructions can be reinterpreted as 
\[\gr\DR_{Y'}(\cM')\cong \sigma_{Y'}^*(\gr_{Y'}\cM')\in \rD^{b,\Gmdil}_{\mathrm{coh}}(\O_{Y'}).\]

Now we restrict to $Y$. We first fix some notations of morphisms
% https://q.uiver.app/#q=WzAsNCxbMCwwLCJcXHJUXipZJ3xfe1l9Il0sWzEsMCwiXFxyVF4qWSciXSxbMSwxLCJZJyJdLFswLDEsIlkiXSxbMCwxLCJpIl0sWzEsMiwicF97WSd9IiwyXSxbMCwzLCJwX3tZfSJdLFszLDIsImkiLDJdLFszLDAsIlxcc2lnbWFfe1l9IiwwLHsiY3VydmUiOi0xfV0sWzIsMSwiXFxzaWdtYV97WSd9IiwyLHsiY3VydmUiOjF9XV0=
\[\begin{tikzcd}
	{\rT^*Y'|_{Y}} & {\rT^*Y'} \\
	Y & {Y';}
	\arrow["i", from=1-1, to=1-2]
	\arrow["{p_{Y}}", from=1-1, to=2-1]
	\arrow["{p_{Y'}}"', from=1-2, to=2-2]
	\arrow["{\sigma_{Y}}", curve={height=-6pt}, from=2-1, to=1-1]
	\arrow["i", from=2-1, to=2-2]
	\arrow["{\sigma_{Y'}}"', curve={height=6pt}, from=2-2, to=1-2]
\end{tikzcd}\]
here $\rT^*Y'|_Y:=\rT^*Y'\times_{Y'}Y$. 

Assume $(\cM',F_{\bullet}\cM')$ underlines a mixed Hodge module $\mbf{M}'$ set theoretically supported on $Y\subset Y'$. Then the conditions for mixed Hodge modules force the $\O_{Y'}$-modules $\gr\cM'$ scheme-theoretically supported on $Y$, more precisely, the ideal $\mathcal{I}\subset \mathcal{O}_{Y'}$ defining $Y$ annihilates $\gr\cM'$ (see \cite[Lem. 3.2.6]{Sai88}). This indicates $i^{-1}\gr_{Y'}\cM'\in \O_{\rT^*Y'|_{Y}}\text{-}\mathrm{mod}^{\Gmdil}_{\mathrm{coh}}$ and $\gr_{Y'}\cM'=i_*i^{-1}\gr_{Y'}\cM'$, where $i^{-1}$ is the usual inverse image functor for abelian sheaves. Thus we get a functor
\[i^{-1}\gr_{Y'}(-):\mathrm{MHM}(Y)\to \O_{\rT^*Y'|_{Y}}\text{-}\mathrm{mod}^{\Gmdil}_{\mathrm{coh}}\]
which is exact by definitions. We also write
\[i^{-1}\gr_{Y'}(-):\rD^b\mathrm{MHM}(Y)\to \rD^{b,\Gmdil}_{\mathrm{coh}}(\O_{\rT^*Y'|_{Y}})\]
for its derived functor. 

We further apply $\sigma_{Y'}^*$. Notice that
\begin{equation}\label{eq:grDR}
	\gr\DR_{Y'}(\cM')\cong \sigma_{Y'}^*(i_*i^{-1}\gr_{Y'}\cM')\cong i_*\sigma_{Y}^*(i^{-1}\gr_{Y'}\cM').
\end{equation}
Thus we define the \textit{graded de Rham functor} $\gr\DR_{Y,i}$ to be the composition of functors
\begin{equation}\label{eq:lyr}
	\rD^b\mathrm{MHM}(Y)\xrightarrow{i^{-1}\gr_{Y'}(-)} \rD^{b,\Gmdil}_{\mathrm{coh}}(\O_{\rT^*Y'|_{Y}})\xrightarrow{\sigma_Y^*}\rD^{b,\Gmdil}_{\mathrm{coh}}(\O_{Y}),
\end{equation}
which fits into the following commutative diagram:
% https://q.uiver.app/#q=WzAsNCxbMCwxLCJcXHJEXmJcXG1hdGhybXtNSE19KFknKSJdLFsxLDEsIlxcckRee2IsXFxHbWRpbH1fe1xcbWF0aHJte2NvaH19KFxcT197WSd9KS4iXSxbMCwwLCJcXHJEXmJcXG1hdGhybXtNSE19KFkpIl0sWzEsMCwiXFxyRF57YixcXEdtZGlsfV97XFxtYXRocm17Y29ofX0oXFxPX3tZfSkiXSxbMiwwLCJpX3sqLFxcSGd9IiwyXSxbMywxLCJpXyoiXSxbMiwzLCJcXGdyXFxEUl97WSxpfSJdLFswLDEsIlxcZ3JcXERSX3tZJ30iXV0=
\[\begin{tikzcd}
	{\rD^b\mathrm{MHM}(Y)} & {\rD^{b,\Gmdil}_{\mathrm{coh}}(\O_{Y})} \\
	{\rD^b\mathrm{MHM}(Y')} & {\rD^{b,\Gmdil}_{\mathrm{coh}}(\O_{Y'}).}
	\arrow["{\gr\DR_{Y,i}}", from=1-1, to=1-2]
	\arrow["{i_{*,\Hg}}"', from=1-1, to=2-1]
	\arrow["{i_*}", from=1-2, to=2-2]
	\arrow["{\gr\DR_{Y'}}", from=2-1, to=2-2]
\end{tikzcd}\]
The $p$-th grading piece of this functor is denoted by
\[\gr_{p}\DR_{Y,i}(-): \rD^b\mathrm{MHM}(Y)\to \rD^{b}_{\mathrm{coh}}(\O_{Y}).\]
It is shown in \cite[Prop. 2.33]{Sai90}  (or see \cite[Lem. 7.3]{Sch16} for more details) that for another embedding $i':Y\hookrightarrow Y''$, there is a natural isomorphism $\gr_{p}\DR_{Y,i}(\mbf{M})\cong \gr_{p}\DR_{Y,i'}(\mbf{M})$. In summary, for any $p\in \Z$, we have a well-defined functor called the {graded de Rham functor}, denoted by
\[\gr_p\DR_Y:\rD^b(\mathrm{MHM}(Y))\to \rD^b_{\mathrm{coh}}(\O_Y).\]
We will omit the subscript $Y$ if there is no confusions.

We also remark that, according to \eqref{eq:Tate twist},  
\begin{equation*}
	\gr_p \DR(\mbf{M}(1))=\gr_{p-1}\DR(\mbf{M}).
\end{equation*}

The following properties about graded de Rham functors are used in this paper.
\begin{proposition}[{\cite[Sec. 2.4]{Sai88}}]
	Let $\omega_{Y}\in \rD^b_{\mathrm{coh}}(\O_Y)$ be the dualizing complex $a^!\O_{\pt}$, and $\D_Y(-):=\cHom(-,\omega_Y): \rD^b_{\mathrm{coh}}(\O_{Y})^{\mathrm{op}}\to \rD^b_{\mathrm{coh}}(\O_{Y})$ be the Grothendieck-Serre dual functor. Then there is an isomorphism
	\[\D_Y\circ \gr_p\DR \cong \gr_{-p}\DR\circ \bD_Y: \rD^b\mathrm{MHM}(Y)^{\mathrm{op}}\to \rD^b_{\mathrm{coh}}(\O_{Y}). \]
	In particular, we apply this isomorphism to the IC-Hodge module $\ICHg_{Y}\in \HM(Y,d_{Y})$. A choice of polarization on $\ICHg_Y$ induces an isomorphism $\bD(\ICHg_Y)\cong \ICHg_Y(d_Y)$, thus
	\begin{equation}\label{eq:D(grDR)}
		\D_Y\left(\gr_p\DR(\ICHg_Y)\right)\cong \gr_{-p}\DR\circ\bD_Y(\ICHg_Y)\cong \gr_{-p}\DR(\ICHg_Y(d_Y))=\gr_{-p-d_Y}\DR(\ICHg_Y).
	\end{equation}
\end{proposition}

\begin{proposition}[{\cite[Sec. 2.3.7]{Sai88}}]\label{prop:push}
	Let $f:X\to Y$ be an embeddable proper morphism between varieties $X,Y$ (i.e. we can choose closed embeddings $i:X\hookrightarrow X',j:Y\hookrightarrow Y'$ for smooth $X',Y'$, and proper morphism $f':X'\to Y'$ compatible with $i,j,f$). Then there is an isomorphism
	\[f_*\circ \gr_p\DR \cong \gr_p\DR\circ f_{*,\Hg}:\rD^b\mathrm{MHM}(X)\to \rD^b_{\mathrm{coh}}(\O_Y).\]
\end{proposition}
\begin{remark}\label{rmk:embed}
	According to our assumption that $X,Y$ are quasi-projective over $\C$, this embeddable condition always holds. The reason is as follows. We choose a locally closed embedding $X\to \mathbb{P}^n$ for some $n$. Then consider the composition $X\hookrightarrow X\times Y \to \mathbb{P}^n\times Y$, where the first morphism is the graph of $f$. Since the composition $X\to \mathbb{P}^n\times Y \to Y$ is the proper morphism $f$, we know the locally closed embedding $X\to \mathbb{P}^n\times Y$ is closed. Now we choose a closed embedding $j:Y\hookrightarrow Y'$ into some smooth variety $Y'$; define $X':=\mathbb{P}^n\times Y'$, $i:X\hookrightarrow X'$ to be the composition $X\hookrightarrow \mathbb{P}^n\times Y \hookrightarrow \mathbb{P}^n\times Y'$, and $f'$ to be the projection $\mathbb{P}^n\times Y'\to Y'$.

\end{remark}

We will write down a proof of \ref{prop:push}; later we will generalize it to a weak equivariant version.
\begin{proof}
	Let the correspondence of cotangent bundles be as follows:
	% https://q.uiver.app/#q=WzAsMyxbMSwwLCJcXHJUXipZJ1xcdGltZXNfe1knfVgnIl0sWzIsMCwiXFxyVF4qWSciXSxbMCwwLCJcXHJUXipYJyJdLFswLDEsInAiXSxbMCwyLCJcXHBpX2YiLDJdXQ==
	\[\begin{tikzcd}
		{\rT^*X'} & {\rT^*Y'\times_{Y'}X'} & {\rT^*Y'.}
		\arrow["{\pi_{f'}}"', from=1-2, to=1-1]
		\arrow["p_{f'}", from=1-2, to=1-3]
	\end{tikzcd}\]
	Then according to \cite[Cor. 3.9]{CDK}, for $\cM'\in \rD^b(\cD_{\hbar,X'}{\text{-}}\mathrm{mod}^{\gr})$, there is a natural isomorphism
	\[(\int_{f'}\cM')\otimes_{\C[\hbar]}\C_0\cong p_{f'*}\pi_{f'}^*(\cM'\otimes_{\C[\hbar]}\C_0).\footnotemark\]
	\footnotetext{In their formula, there is another term $\omega_{X/Y}$, since they use left $\cD$-modules}
	Moreover, we restrict the above diagram to $X,Y$:
% https://q.uiver.app/#q=WzAsNCxbMSwwLCIgXFxyVF4qWSd8X3tZfVxcdGltZXNfe1l9WCJdLFsyLDAsIlxcclReKlknfF97WX0iXSxbMCwwLCJcXHJUXipYJ3xfe1h9Il0sWzEsMSwiXFxyVF4qWSdcXHRpbWVzX3tZJ31YIl0sWzAsMSwicCJdLFswLDIsIlxccGlfZiIsMl0sWzAsMywiIiwwLHsib2Zmc2V0IjotMSwic2hvcnRlbiI6eyJzb3VyY2UiOjIwLCJ0YXJnZXQiOjIwfSwic3R5bGUiOnsiaGVhZCI6eyJuYW1lIjoibm9uZSJ9fX1dLFswLDMsIiIsMCx7Im9mZnNldCI6MSwic2hvcnRlbiI6eyJzb3VyY2UiOjIwLCJ0YXJnZXQiOjIwfSwic3R5bGUiOnsiaGVhZCI6eyJuYW1lIjoibm9uZSJ9fX1dXQ==
\[\begin{tikzcd}
	{\rT^*X'|_{X}} & { \rT^*Y'|_{Y}\times_{Y}X} & {\rT^*Y'|_{Y}.} \\
	& {\rT^*Y'\times_{Y'}X}
	\arrow["{\pi_f}"', from=1-2, to=1-1]
	\arrow["p_f", from=1-2, to=1-3]
	\arrow[shift left, between={0.2}{0.8}, no head, from=1-2, to=2-2]
	\arrow[shift right, between={0.2}{0.8}, no head, from=1-2, to=2-2]
\end{tikzcd}\]
	Let $\mbf{M}'=i_{*,\Hg}\mbf{M}$ for some $\mbf{M}\in \rD^b\mathrm{MHM}(X)$ and $\cM'=\mathrm{U}_{X'}(\mbf{M}')\in \rD^b(\cD_{\hbar,X'}{\text{-}}\mathrm{mod}^{\gr}) $ . Combined with the compatibility \eqref{eq:U&push}, we obtain
	\[\gr \left(f'_{*,\Hg}(\mbf{M}')\right)\cong p_{f'*}\pi_{f'}^* \gr(\mbf{M}').\]
	Then according to the interpretation of $\gr\DR$ (\ref{eq:grDR}), we can compute that
	\[\gr\DR f_{*,\Hg}(\mbf{M})=\sigma_Y^*j^{-1}\gr f'_{*,\Hg}(\mbf{M}')\cong \sigma_{Y}^*j^{-1} p_{f'*}\pi_{f'}^* \gr(\mbf{M}')=\sigma_{Y}^*j^{-1} p_{f'*}\pi_{f'}^* i_*i^{-1}\gr(\mbf{M}').\]
	Since $\rT^*X'$, $\rT^*Y'\times_{Y'}X'$ are flat over $X'$, we can do base change along $i:X\hookrightarrow X'$, and get the Cartesian diagram in the derived sense (i.e. in the $\infty$-category of derived schemes):
	% https://q.uiver.app/#q=WzAsNCxbMSwwLCJcXHJUXipZJ1xcdGltZXNfe1knfVgnIl0sWzAsMCwiXFxyVF4qWCciXSxbMCwxLCJcXHJUXipYJ3xfe1h9Il0sWzEsMSwiXFxyVF4qWSdcXHRpbWVzX3tZJ31YIl0sWzAsMSwiXFxwaV97Zid9IiwyXSxbMiwxLCJpIl0sWzMsMCwiXFx3dHtpfSIsMl0sWzMsMiwiXFxwaV97Zn0iXV0=
	\begin{equation}\label{eq:zqx}
		\begin{tikzcd}
			{\rT^*X'} & {\rT^*Y'\times_{Y'}X'} \\
			{\rT^*X'|_{X}} & {\rT^*Y'\times_{Y'}X.}
			\arrow["{\pi_{f'}}"', from=1-2, to=1-1]
			\arrow["i", from=2-1, to=1-1]
			\arrow["{\wt{i}}"', from=2-2, to=1-2]
			\arrow["{\pi_{f}}"', from=2-2, to=2-1]
		\end{tikzcd}
	\end{equation}
	Then we can continue the computation:
	\begin{align*}
		\sigma_{Y}^*j^{-1} p_{f'*}\pi_{f'}^* i_*i^{-1}\gr(\mbf{M}')\cong \sigma_{Y}^*j^{-1} p_{f'*}\wt{i}_*\pi_{f}^* i^{-1}\gr(\mbf{M}')\cong \sigma_{Y}^*j^{-1} j_*p_{f*}\pi_{f}^* i^{-1}\gr(\mbf{M}')\cong \sigma_{Y}^*p_{f*}\pi_{f}^* i^{-1}\gr(\mbf{M}');
	\end{align*}
	here, the first isomorphism is the base change isomorphism applied to \eqref{eq:zqx}, and the second isomorphism follows from the equality $p_{f'}\circ \wt{i}=j\circ p_f$.
	
	Using similar reasons as before, we have the following Cartesian diagram in the derived sense
	% https://q.uiver.app/#q=WzAsNCxbMCwwLCJcXHJUXipZJ3xfe1l9XFx0aW1lc197WX1YIl0sWzAsMSwiWCJdLFsxLDEsIlkiXSxbMSwwLCJcXHJUXipZJ3xfe1l9Il0sWzEsMCwiXFx3dHtcXHNpZ21hX3tZfX0iXSxbMSwyLCJmIiwyXSxbMiwzLCJcXHNpZ21hX1kiLDJdLFswLDMsInBfZiJdXQ==
	\begin{equation}\label{eq:jbm}
		\begin{tikzcd}
			{\rT^*Y'|_{Y}\times_{Y}X} & {\rT^*Y'|_{Y}} \\
			X & {Y.}
			\arrow["{p_f}", from=1-1, to=1-2]
			\arrow["{\wt{\sigma_{Y}}}", from=2-1, to=1-1]
			\arrow["f", from=2-1, to=2-2]
			\arrow["{\sigma_Y}"', from=2-2, to=1-2]
		\end{tikzcd}
	\end{equation}
	Then we can finish the computation:
	\[\sigma_{Y}^*p_{f*}\pi_{f}^* i^{-1}\gr(\mbf{M}')\cong f_*\wt{\sigma_Y}^*\pi_{f}^*i^{-1}\gr(\mbf{M}')\cong f_*\sigma_X^*i^{-1}\gr(\mbf{M}')=f_*\gr\DR(\mbf{M});\]
	here, the first isomorphism is the base change isomorphism applied to \eqref{eq:jbm}, and the second isomorphism follows from the equality $\sigma_X=\pi_f\circ \wt{\sigma_Y}$.
\end{proof}

\subsection{Equivariant Hodge modules and weak equivariant graded de Rham functors}\label{subsec:equivariant}
We fix a connected algebraic group $H$, i.e. a connected affine group scheme of finite type over $\C$. 
\subsubsection{Equivariant Hodge modules and $\cD_{\hbar}$-modules}

Let $H$ act on a variety $Y$ over $\C$. Write $a:H\times Y\to Y$ for the action map and $p:H\times Y\to Y$ for the projection map. An \textit{$H$-equivariant mixed Hodge module} on $Y$ is a pair $(\mbf{M},\beta)$ where $\mbf{M}\in \mathrm{MHM}(Y)$, and $\beta:a^{\dagger}\mbf{M}\xrightarrow{\cong} p^{\dagger}\mbf{M}$ is an isomorphism satisfying the usual cocycle conditions. Let $$\mathrm{MHM}^H(Y)$$ 
denote the (abelian) category of $H$-equivariant mixed Hodge modules on $Y$. 

The equivariant structure on a mixed Hodge module equips its underlying graded $\cD_{\hbar}$-module an equivariant structure. But it should be noted that there are two kinds of definitions of equivariant $\cD_{\hbar}$-module: a strong version and a weak version. For a smooth variety $Y'$ with an $H$-action, a \textit{strong $H$-equivariant} graded $\cD_{\hbar,Y'}$-module is a pair $(\cM',\beta')$ where $\cM'\in \cD_{\hbar,Y'}\text{-}\mathrm{mod}^{\mathrm{gr}}$, and $\beta':a^{\dagger}\cM'\xrightarrow{\cong}p^{\dagger}\cM'$ is an isomorphism {as graded $\cD_{\hbar,H\times Y'}=\cD_{\hbar,H}\boxtimes_{\C[\hbar]}\cD_{\hbar,Y'}$-modules}, satisfying the usual cocycle conditions. A \textit{weak $H$-equivariant} graded $\cD_{\hbar,Y'}$-module is a pair $(\cM',\beta')$ where $\cM'\in \cD_{\hbar,Y'}\text{-}\mathrm{mod}^{\mathrm{gr}}$, and $\beta':a^{\dagger}\cM'\xrightarrow{\cong}p^{\dagger}\cM'$ is an isomorphism {as graded $\O_{\hbar,H}\boxtimes_{\C[\hbar]}\cD_{\hbar,Y'}=\O_{H}\boxtimes \cD_{\hbar,Y'}$-modules}, satisfying the usual cocycle conditions. 
Let
\[\cD_{\hbar,Y'}\text{-}\mathrm{mod}^{\mathrm{gr},H\text{-strong}} \text{ \ \ \ and \ \ \ }\cD_{\hbar,Y'}\text{-}\mathrm{mod}^{\mathrm{gr},H\text{-weak}}\]
denote the (abelian) category of strong and weak $H$-equivariant graded $\cD_{\hbar,Y'}$-modules respectively. There is a natural exact forgetful functor
\[\mathrm{For}^{\text{s}\to \text{w}}:\cD_{\hbar,Y'}\text{-}\mathrm{mod}^{\mathrm{gr},H\text{-strong}} \to \cD_{\hbar,Y'}\text{-}\mathrm{mod}^{\mathrm{gr},H\text{-weak}}.\]
For an $H$-equivariant mixed Hodge module $\mbf{M}'\in \MHM^H(Y')$, its underlying graded $\cD_{\hbar}$-module has a natural strong $H$-equivariant structure. Thus there is a natural functor, which is the (strong) equivariant version of \eqref{eq:U}:
\begin{equation}\label{eq:yyy}
	\mathrm{U}_{Y'}^{H\text{-strong}}:\mathrm{MHM}^H(Y')\to \cD_{\hbar,Y'}\text{-}\mathrm{mod}^{\mathrm{gr},H\text{-strong}} .
\end{equation}
We also write 
\[\mathrm{U}_{Y'}^{H\text{-weak}}:=\mathrm{For}^{\text{s}\to \text{w}} \circ\mathrm{U}_{Y'}^{H\text{-strong}}:\mathrm{MHM}^H(Y')\to \cD_{\hbar,Y'}\text{-}\mathrm{mod}^{\mathrm{gr},H\text{-weak}}.\]

These equivariant constructions all have triangulated categories versions. The equivariant derived category of mixed Hodge modules 
$$\rD^b\MHM^H(Y)$$ 
and the strong equivariant derived category of graded $\cD_{\hbar}$-modules $$\rD^{b,\gr,H\strong}(\cD_{\hbar,Y'})$$
can be constructed following the idea of \cite{BL94}; a detailed treatment can be found in \cite{Ach} and \cite[Sec. 4]{Kas08}\footnote{More precisely, \cite{Kas08} considers $\cD$-modules, and the constructions apply to $\cD_{\hbar}$-modules.} respectively\footnote{We only use (equivariant) $\cD_{\hbar}$-modules for smooth varieties. For mixed Hodge modules, although \cite{Ach} only considers the case when $Y$ is smooth, one can use the derived version of Kashiwara's lemma to extend this theory to singular cases. In fact, $\rD^b\MHM^H(Y)$ should be viewed as the derived category of mixed Hodge modules on the quotient stack $[H\backslash Y]$, and this theory can be extended to general algebraic stacks by \cite{Tub25}.}. It should be \textbf{warned} that these equivariant derived categories are not the derived categories of equivariant abelian categories unless $H$ is unitary. There is also a natural functor
\begin{equation}\label{eq:stq}
	\mathrm{U}_{Y'}^{H\text{-strong}}:\rD^b\mathrm{MHM}^H(Y')\to \rD^{b,\gr,H\strong}(\cD_{\hbar,Y'}),
\end{equation}
which is t-exact and recover \eqref{eq:yyy} after restricting to the hearts.

In the weak equivariant case, the derived category of $\cD_{\hbar,Y'}\text{-mod}^{\gr,H\weak}$ is the "correct" category according to \cite[Sec. 3.4-3.8]{Kas08}\footnote{The name used there is quasi-equivariant $\cD$-modules.}, thus let's define
\[\rD^{b,\gr,H\weak}(\cD_{\hbar,Y'}):=\rD^b(\cD_{\hbar,Y'}\text{-mod}^{\gr,H\weak}).\]
There is also a natural forgetful functor
\begin{equation}\label{eq:lpc}
	\mathrm{For}^{\text{s}\to \text{w}}: \rD^{b,\gr,H\strong}(\cD_{\hbar,Y'}) \to \rD^{b,\gr,H\weak}(\cD_{\hbar,Y'}),
\end{equation}
which is t-exact and recover \eqref{eq:yyy} after restricting to the hearts. Again
\[\mathrm{U}_{Y'}^{H\text{-weak}}:=\mathrm{For}^{\text{s}\to \text{w}} \circ\mathrm{U}_{Y'}^{H\text{-strong}}:\rD^b\mathrm{MHM}^H(Y')\to \rD^{b,\gr,H\weak}(\cD_{\hbar,Y'}).\]

We will use the pushforward functors for these equivariant derived categories defined in above references. They are compatible with the functors \eqref{eq:stq} \eqref{eq:lpc}, and compatible with the pushforward functors for non-equivariant derived categories after forgetting the equivariances. We also remark that the pushforward functor for weak equivariant derived categories has the same formula as in \eqref{eq:int-push}.

\subsubsection{Weak equivariant graded de Rham functor}\label{subsubsec:weakequiv_grDR}
Now we want to generalize the constructions in subsection \ref{subsec:grDR} to weak equivariant versions. Here the word "weak" indicates weak equivariant $\cD_{\hbar}$-modules. 

For a smooth $H$-variety $Y'$, define the \textit{weak equivariant graded deRham functor} $\gr\DR^H_{Y'}$ to be the composition
\[\gr\DR_{Y'}^H:\rD^b\MHM^H(Y')\xrightarrow{U_{Y'}^{H\weak}} \rD^{b,\gr,H\weak}_{\mathrm{coh}}(\cD_{\hbar, Y'})\xrightarrow{\otimes_{\C[\hbar]}\C_0} \rD^{b,H\times\Gmdil}_{\mathrm{coh}}(\O_{\rT^*Y'})\xrightarrow{\sigma_{Y'}^*}\rD^{b,H\times\Gmdil}_{\mathrm{coh}}(\O_{Y'}).\]

For a possibly singular $H$-variety $Y$, we assume there is an $H$-equivariant closed embedding $i:Y\hookrightarrow Y'$ into some smooth $H$-variety $Y'$ (when $Y$ is normal, this can always be done using \cite[Thm. 5.1.25]{CG}). Given $\mbf{M}\in \MHM^H(Y)$, let $\mbf{M}':=i_{*,\Hg}\mbf{M}\in \MHM^H(Y')$. As in the previous subsection \ref{subsec:grDR}, $$\gr_{Y'}^{H}(\mbf{M}'):=\mathrm{U}_{Y'}^{H\weak}(\mbf{M}')\otimes_{\C[\hbar]}\C_0\in \O_{\rT^*Y'}\text{-mod}^{H\times\Gmdil}_{\mathrm{coh}}$$
supports on $Y$ by \cite[Lem. 3.2.6]{Sai88}, which indicates $i^{-1}\gr_{Y'}^{H}(\mbf{M}')\in \O_{\rT^*Y'|_{Y}}\text{-mod}_{\mathrm{coh}}^{H\times \Gmdil}$ and $\gr_{Y'}^{H}(\mbf{M}')=i_*i^{-1}\gr_{Y'}^{H}(\mbf{M}')$. We further define the \textit{weak equivariant graded de Rham functor} in this \textit{possibly singular case}, by the composition
\begin{equation}\label{eq:lfy}
	\gr\DR_Y^H:\MHM^H(Y)\xrightarrow{i^{-1}\gr_{Y'}^{H}}\O_{\rT^*Y'|_{Y}}\text{-mod}_{\mathrm{coh}}^{H\times\Gmdil}\xrightarrow{\sigma_Y^*} \rD^{b,H\times\Gmdil}_{\mathrm{coh}}(\O_{Y}),
\end{equation}
which fits into the following commutative diagram:
% https://q.uiver.app/#q=WzAsNCxbMCwxLCJcXHJEXmJcXG1hdGhybXtNSE19XkgoWScpIl0sWzEsMSwiXFxyRF57YixIXFx0aW1lc1xcR21kaWx9X3tcXG1hdGhybXtjb2h9fShcXE9fe1knfSkuIl0sWzAsMCwiXFxtYXRocm17TUhNfV5IKFkpIl0sWzEsMCwiXFxyRF57YixIXFx0aW1lc1xcR21kaWx9X3tcXG1hdGhybXtjb2h9fShcXE9fe1l9KSJdLFsyLDAsImlfeyosXFxIZ30iLDJdLFszLDEsImlfKiJdLFsyLDMsIlxcZ3JcXERSXkhfe1l9Il0sWzAsMSwiXFxnclxcRFJeSF97WSd9Il1d
\[\begin{tikzcd}
	{\mathrm{MHM}^H(Y)} & {\rD^{b,H\times\Gmdil}_{\mathrm{coh}}(\O_{Y})} \\
	{\mathrm{MHM}^H(Y')} & {\rD^{b,H\times\Gmdil}_{\mathrm{coh}}(\O_{Y'}).}
	\arrow["{\gr\DR^H_{Y}}", from=1-1, to=1-2]
	\arrow["{i_{*,\Hg}}"', from=1-1, to=2-1]
	\arrow["{i_*}", from=1-2, to=2-2]
	\arrow["{\gr\DR^H_{Y'}}", from=2-1, to=2-2]
\end{tikzcd}\]
This functor doesn't depend on the choice of smooth ambient space once we prove the Proposition \eqref{prop:equi_version} below. As before, we write $\gr_p\DR_{Y}^H$ for the $p$-th grading piece of $\gr\DR_{Y}^H$ for any $p\in \mathbb{Z}$, and omit the subscript $Y$ if there is no confusion.

\begin{remark}
	Since $\rD^b\MHM^H(Y)$ is not the derived category of $\MHM^H(Y)$, it's more tricky to define $\gr\DR^H_Y$ for the whole equivariant derived category when $Y$ is possibly singular. Given an $H$-equivariant smooth ambient space $Y\hookrightarrow Y'$, we have the composition functor 
	\begin{equation}\label{eq:xr}
		\rD^b\MHM^H(Y)\xrightarrow{i_{*,\Hg}} \rD^b\MHM^H(Y') \xrightarrow{\gr\DR_{Y'}^H} \rD^{b,H\times\Gmdil}_{\mathrm{coh}}(\O_{Y'});
	\end{equation}
	we hope this functor factors through $\rD^{b,H\times \Gmdil}_{\mathrm{coh}}(\O_{Y})\xrightarrow{i_*} \rD^{b,H\times \Gmdil}_{\mathrm{coh}}(\O_{Y'})$. Note that this was done after forgetting the $H$-equivariance by \eqref{eq:lyr}. 
		
	The version \eqref{eq:lfy} is enough for our purpose.
\end{remark}

Now we discuss the weak equivariant version of the Proposition \ref{prop:push}.
\begin{proposition}\label{prop:equi_version}
\footnote{Once the whole derived version of weak equivariant graded de Rham functor in possibly singular case \eqref{eq:lfy} is defined, we should also have a more uniform statement of this proposition.}\ 
	\begin{enumerate}
		\item Let $f':X'\to Y'$ be a proper $H$-equivariant morphism between smooth $H$-varieties $X',Y'$. Then there is an isomorphism
		\[f'_*\circ \gr_p\DR^H \cong \gr_p\DR^H\circ f'_{*,\Hg}:\rD^b\MHM^H(X')\to \rD^{b,H}_{\mathrm{coh}}(\O_{Y'}).\]
		
		\item Let $f:X\to Y$ be an embeddable $H$-equivariant proper morphism between $H$-varieties $X,Y$ (i.e. we can choose $H$-equivariant closed embeddings $i:X\hookrightarrow X',j:Y\hookrightarrow Y'$ for smooth $H$-varieties $X',Y'$, and $H$-equivariant proper morphism $f':X'\to Y'$ compatible with $i,j,f$). Let $\mbf{M}\in \MHM^H(X)$. Assume $f_{*,\Hg}(\mbf{M})\in \MHM^H(Y)$. Then there is an isomorphism
		\[f_*\circ \gr_p\DR^H(\mbf{M}) \cong \gr_p\DR^H\circ f_{*,\Hg}(\mbf{M})\]
		in $\rD^{b,H}_{\mathrm{coh}}(\O_{Y})$.
	\end{enumerate}
\end{proposition}

\begin{remark}
	When $X$ and $Y$ are normal, they can $H$-equivariantly embedd to some projective space $\mathbb{P}^n$ using \cite[Thm. 5.1.25]{CG}. Then we can use the method in Remark \ref{rmk:embed} to construct some embedding of $f$.
\end{remark}

\begin{proof}[Proof of Prop. \ref{prop:equi_version}]
	The proof is completely parallel to the non-equivariant case. We only need to imitate the proof of \cite[Cor. 3.9]{CDK}, and get the following natural isomorphism
	\[(\int_{f'}\cM')\otimes_{\C[\hbar]}\C_0\cong p_{f'*}\pi_{f'}^*(\cM'\otimes_{\C[\hbar]}\C_0), \ \ \forall \cM'\in \rD^{b,\gr,H\weak}_{\mathrm{coh}}(\cD_{\hbar,Y'})\]
	in $\rD^{b,\gr,H\weak}_{\mathrm{coh}}(\O_{\rT^*Y'})$
\end{proof}

We finally remark on the compatibility under a change of group action, which directly follows from the constructions. Let $H_1\to H_2$ be a homomorphism of algebraic groups over $\C$. Let $H_2$ acts on a variety $Y$, so $Y$ inherites an action of $H_1$. There are natural forgetful functors
\[\rD^b\MHM^{H_2}(Y)\to\rD^b\MHM^{H_1}(Y), \ \ \ \rD^{b,H_2}_{\mathrm{coh}}(\O_{Y})\to \rD^{b,H_1}_{\mathrm{coh}}(\O_{Y}),\]
which are compatible with the weak equivariant graded de Rham functor, i.e. the following diagram commutes:
% https://q.uiver.app/#q=WzAsNCxbMCwwLCJcXHJEXmJcXEhNXntIXzJ9KFkpIl0sWzEsMCwiXFxPX1lcXHRleHR7LX1cXG1vZF57SF8yfV97XFxtYXRocm17Y29ofX0iXSxbMCwxLCJcXHJEXmJcXEhNXntIXzF9KFkpIl0sWzEsMSwiXFxPX1lcXHRleHR7LX1cXG1vZF57SF8xfV97XFxtYXRocm17Y29ofX0iXSxbMCwxLCJcXGdyX3BcXERSXntIXzJ9Il0sWzAsMl0sWzIsMywiXFxncl9wXFxEUl57SF8xfSIsMl0sWzEsM11d
\begin{equation}\label{eq:compatible}
	\begin{tikzcd}
		{\MHM^{H_2}(Y)} & {\rD^{b,H_2}_{\mathrm{coh}}(\O_{Y})} \\
		{\MHM^{H_1}(Y)} & {\rD^{b,H_1}_{\mathrm{coh}}(\O_{Y}).}
		\arrow["{\gr_p\DR^{H_2}}", from=1-1, to=1-2]
		\arrow[from=1-1, to=2-1]
		\arrow[from=1-2, to=2-2]
		\arrow["{\gr_p\DR^{H_1}}", from=2-1, to=2-2]
	\end{tikzcd}
\end{equation}

\subsection{Hodge module Satake categories}\label{subsec: Satake equivalence for Hodge modules}
In this subsection, we review some results on Hodge module Satake categories \cite{Fe21}  required for the Section \ref{subsec: Tools coming from Hodge module}, in particular the proof of Corollary \ref{cor:m*(cIC)}. 

In \cite{Fe21}, the author extends the equivariant Hodge module theory $\HM^H(Y,w)$ to the case when $H$ is some pro-algebraic group (i.e. $H$ is a limit $\lim\limits_{+\infty\leftarrow m}H^{(m)}$ along surjective morphisms of algebraic groups; they also assume each $H^{(m)}$ and kernel of transition map is connected). In fact, if $H_1\to H_2$ is a surjective morphism of algebraic groups with connected kernel, and $H_2$ acts on $Y$, then the forgetful functor $\HM^{H_2}(Y,w)\to \HM^{H_1}(Y,w)$ is an equivalence (\cite[Lem. 2.15]{Fe21}). Therefore one can define $\HM^H(Y,w):=\colim\limits_{m>>0}\HM^{H^{(m)}}(Y,w)$, which is in fact equivalent to $\HM^{H^{(m)}}(Y,w)$ for any $m>>0$. Similar construction holds for $\MHM^H(Y)$.

\cite{Fe21} also extends the theory $\HM^H(Y,w)$ to the case when $Y$ is an ind-variety. In particular, the author defined and studied $\HM^{G_{\O}}(\Gr_G,w)$ together with
\[\HM^{G_{\O}}(\Gr_G):=\bigoplus_{w\in \Z}\HM^{G_{\O}}(\Gr_G,w).\]
This category admits a monoidal structure given by the convolution product.

Define $\Tate^{G_{\O}}(\Gr_G)$ to be the full subcategory of $\HM^{G_{\O}}(\Gr_G)$ whose objects are isomorphic to direct sums of Tate twists of $$\ICHg_{\lav}:=\ICHg_{\overline{\Gr}_{\lav}}\in \HM^{G_{\O}}(\overline{\Gr}_{\lav})\subset \HM^{G_{\O}}(\Gr_G)$$
for all $\lav\in \Pv_+$.
\begin{proposition}[{\cite[Prop. 4.3]{Fe21}}]\label{prop:Tate_subcat}
	The full subcategory $\Tate^{G_{\O}}(\Gr_G)\subset \HM^{G_{\O}}(\Gr_G)$ is closed under the convolution product.
\end{proposition}

\begin{corollary}\label{cor:convolution of Hodge modules}
	Let $m:\overline{\Gr}_{\lav,\muv}:=\overline{\Gr}_{\lav}\ttimes \overline{\Gr}_{\muv}\to \overline{\Gr}_{\lav+\muv}$ be the convolution map, and $\ICHg_{\lav,\muv}:=\ICHg_{\overline{\Gr}_{\lav,\muv}}\in \HM^{G_{\O}}(\overline{\Gr}_{\lav,\muv})\subset \HM^{G_{\O}}(\Gr\ttimes \Gr)$. Then
	\[m_{*,\Hg}(\ICHg_{\lav,\muv})\cong \bigoplus_{\nuv\in \Pv_+}(\ICHg_{\nuv})^{\oplus \clamunu}(\tfrac12 d_{\nuv}-\tfrac12 d_{\lav,\muv})\]
	in $\HM^{G_{\O}}(\overline{\Gr}_{\lav+\muv})\subset \HM^{G_{\O}}(\Gr_G)$.
\end{corollary}
\begin{proof}
	By the formula of convolution product, $m_{*,\Hg}(\ICHg_{\lav,\muv})$ is just the convolution product $\ICHg_{\lav}\conv \ICHg_{\muv}$. Then according to Proposition \ref{prop:Tate_subcat}, $m_{*,\Hg}(\ICHg_{\lav,\muv})$ has the form
	\[\bigoplus_{\nuv\in \Pv_+}(\ICHg_{\nuv})^{\oplus a_{\nuv}}(w_{\nuv})\]
	for some $a_{\nuv},w_{\nuv}\in \mathbb{N}$. By comparing the weights, we have $w_{\nuv}=\tfrac12 d_{\nuv}-\tfrac12 d_{\lav,\muv}$. By applying the functor $\Rat: \HM^{G_{\O}}(\Gr_G)\to \Perv^{G_{\O}}(\Gr_G,\Q)$, we obtain $a_{\nuv}=\clamunu$ from the geometric Satake equivalence \cite{MV07}.
\end{proof}

\begin{remark}
	As in the main context, let $\hGO=(G_{\O}\rtimes\Gmrot)\times \Gmdil$. We can replace the $G_{\O}$-equivariance in above Corollary \ref{cor:convolution of Hodge modules} by the $\hGO$-equivariance, since $\ICHg_{\lav,\muv},\ICHg_{\nuv}$ all have $\hGO$-equivariant structure, and it's known that the natural forgetful functor $\HM^{\hGO}(\Gr)\to \HM^{G_{\O}}(\Gr)$ is fully faithful (see e.g. \cite[Prop. 2.16]{Fe21}).
\end{remark}

\bibliographystyle{alphaurl}
\bibliography{ref}

\end{document}